\documentclass[letterpaper,10pt]{article}

\usepackage{graphicx,verbatim}

\newcommand{\pmat}[1]{\begin{pmatrix} #1 \end{pmatrix}}

\renewcommand{\baselinestretch}{0.98}

\usepackage{wrapfig} 
\usepackage{amsmath,verbatim,amsthm}
\usepackage{subfig}

\usepackage{amssymb}  
\usepackage[dvipsnames,usenames]{color}
\usepackage[colorlinks,%
linkcolor=BrickRed,%
filecolor =BrickRed,%
citecolor=RoyalPurple,%
]{hyperref}

\usepackage{amsmath} 
\usepackage{amssymb}  
\usepackage{color}
\usepackage{cite}

\def\mbR{\mathbb{R}}
\newtheorem{theorem}{Theorem}
\newtheorem{example}{Example}
\newtheorem{lemma}{Lemma}

\newtheorem{assumption}{Assumption}
\newtheorem{proposition}{Proposition}
\newtheorem{corollary}{Corollary}

\newcommand{\an}[1]{{\color{black}#1}}

\newcommand{\Prob}[1]{\mathsf{Prob}#1 } 
\newcommand{\remove}[1]{}
\newcommand{\EXP}[1]{\mathsf{E}\!\left[#1\right] }
\newcommand{\EXPz}[1]{\mathsf{E}_{z,\xi}\!\left[#1\right] }

\def\sF{\mathcal{F}}

\def\Real{\mathbb{R}}
\def\g{\gamma}

\def\e{\epsilon}
\def\a{\alpha}

\def\an#1{{{\color{black}#1}}}

\def\sfy#1{{{\color{black}#1}}}

\usepackage{ulem}
\author{Farzad~Yousefian,   
        Angelia~Nedi\'c, and   
		Uday V.~Shanbhag\thanks{The first two authors are with the Department of Industrial and Enterprise
Systems Engineering, University of Illinois, Urbana, IL 61801, USA,
		while the last author is with the Department of Industrial and
			Manufacturing Engineering, Pennsylvania State University,
						  University Park, PA 16802, USA. They are
							  contactable at 
{\tt\small \{yousefi1,angelia\}@illinois.edu} and {\tt \small
	udaybag@psu.edu}. Nedi\'{c} and
Shanbhag gratefully acknowledge the support of the NSF through the award
NSF CMMI 0948905 ARRA. Additionally, Nedi\'{c} has been funded by NSF
award CMMI-0742538 and Shanbhag has been supported by NSF award
CMMI-1246887.}}

\title{Distributed adaptive steplength stochastic approximation schemes
	for 
	Cartesian stochastic variational inequality problems}
\begin{document}
\maketitle
\thispagestyle{empty}
\pagestyle{plain}
\begin{abstract}
Motivated by problems arising in decentralized control problems and
non-cooperative Nash games, we consider a class of strongly monotone
Cartesian variational inequality (VI) problems, where the mappings either contain
expectations or their evaluations are corrupted by error. Such
complications are captured under the umbrella of Cartesian stochastic variational
inequality problems and we consider \an{solving} such
problems via stochastic approximation {(SA)} schemes. Specifically, we propose
a scheme wherein the
steplength sequence is derived by a rule that depends on problem
parameters such as monotonicity and Lipschitz constants. The proposed
scheme is seen to produce sequences that are guaranteed to converge
almost surely to the unique solution of the problem. To cope with
networked multi-agent generalizations, we provide requirements under
which independently chosen steplength rules still possess
desirable almost-sure convergence properties. In the second part
of this paper, we consider a regime where Lipschitz constants on the map
are either unavailable or difficult to derive. Here, we present a local
randomization technique that allows for deriving an approximation of
the original mapping, {which is then shown} to be Lipschitz continuous
with a prescribed constant. Using this technique, we introduce a
locally randomized SA {algorithm} and provide almost sure convergence
theory for the resulting sequence of iterates to an approximate solution
of the original variational inequality problem. Finally, the paper
concludes with some preliminary numerical results on a stochastic rate
allocation problem and a stochastic Nash-Cournot game.
	\end{abstract} \maketitle

\section{Introduction}\label{sec:introduction}
Multi-agent system-theoretic problems can collectively capture a range
of problems arising from decentralized control problems and
noncooperative games. In static regimes, where agent problems are
convex and agent feasibility sets are uncoupled, the associated
solutions of such problems are given by the solution of a suitably
defined Cartesian variational inequality problem. Our interest lies in
settings where the mapping arising in such problems is strongly monotone
and one of the following hold: (i) Either the mapping contains
expectations whose analytical form is unavailable; or (ii) The
evaluation of such a mapping is corrupted by error. In either case, the
appropriate problem of interest is given by a stochastic
variational inequality problem VI$(X,F)$ that requires determining an $x^* \in X$
such that
\begin{align}
	(x-x^*)^T F(x^*) \geq 0 \qquad  \hbox{for all }x \in X, 
\end{align}
where 
\begin{align}
\label{def-F} F(x) \triangleq \pmat{ \EXP{\Phi_1(x,\xi)} \\
							\vdots \\
						  \EXP{\Phi_N(x,\xi)}},  
\end{align}
$\Phi_i:  {\cal D}_i \times \Real^d \to
\Real^{n_i}$, ${\cal D}_i \subseteq \Real^{n_i}$, $X$ is a closed and convex
set, ${\cal D}_i$ is an open set in $\Real^{n_i}$ and $\sum_{i=1}^N n_i
= n$. Furthermore, $\xi: \Omega\to \Real^d$ is a random variable, where
$\Omega$ denotes the associated sample space and $\EXP{\cdot}$ denotes
the expectation with respect to $\xi$.

Variational inequality problems assume relevance in capturing the
solution sets of convex optimization and equilibrium problems~\cite{facchinei02finite}. Their
Cartesian specializations arise from specifying the set $X$ as a Cartesian
product, i.e., $X \triangleq \prod_{i=1}^N X_i$.
Such problems arise in the modeling of multi-agent
decision-making problems such as rate allocation problems in communication
networks~\cite{Kelly98,Srikant04,ShakSrikant07}, 
noncooperative Nash games in communication
networks~\cite{alpcan02game,alpcan03distributed,yin09nash2}, 
competitive interactions in cognitive radio
networks~\cite{aldo1,aldo2,scutari10monotone,koshal11single2}, and
strategic behavior in power markets~\cite{KShKim11,KShKim12,SIGlynn11}. Our
interest lies in regimes complicated by uncertainty, which could arise
as a result of agents facing expectation-based objectives that
do not have tractable analytical forms. Naturally, 
the Cartesian stochastic variational inequality problem framework
represents an expansive model for capturing a range of such problems. 

Two broad avenues exist for
solving such a class of problems. Of these, the first approach, referred
to as the sample-average approximation (SAA) method. In adopting this approach, one
uses a set of $M$ samples $\{\xi_1, \hdots,
	\xi_M\}$ and considers the sample-average problem where an expected mapping 
	$\EXP{\Phi(x,\xi)}$ is replaced by the sample-average
	${\sum_{j=1}^M \Phi(x,\xi^j)}/{M}$. The resulting 
problem is deterministic and its solution provides an estimator for the
solution of the true problem. The asymptotic behavior of these
estimators has been studied extensively in the context of stochastic
optimization and variational
problems~\cite{linderoth02empirical,shap03sampling}. The other
approach, referred to as stochastic approximation, also has a long
tradition. First proposed by Robbins and Monro~\cite{robbins51sa} for
root-finding problems and by Ermoliev for stochastic
programs~\cite{Ermoliev76,Ermoliev83,Ermoliev88}, significant effort has
been applied towards theoretical and algorithmic examination of such
schemes (cf.~\cite{Borkar08,Kush03,Spall03}). Yet, there has been markedly
little on the application of such techniques to solution of stochastic
variational inequalities, exceptions
being~\cite{Houyuan08,koshal10single}. Standard stochastic
approximation schemes provide little guidance regarding the choice of a
steplength sequence, {denoted by $\{\gamma_k\}$}, apart from requiring
that the sequence satisfies $$ \sum_{k=0}^{\infty} \gamma_k = \infty
\quad\mbox{ and }\quad \sum_{k=0}^{\infty} \gamma_k^2 < \infty.$$

The behavior of stochastic approximation schemes is closely tied to the
choice of steplength sequences. Generally, there have been two avenues
traversed in choosing steplengths: (i) {\em Deterministic steplength
	sequences:} Spall \cite[Ch.~4, pg.~113]{Spall03} considered
	diverse choices of the form
$\gamma_k=\frac{\beta}{(k+1+a)^\alpha}$, where 
$\beta>0$, $0<\alpha \leq 1$, and $a \geq 0$ is
a stability constant. In related work in the context of approximate
dynamic programming, Powell~\cite{Powell10} examined several
deterministic update rules. However, much of these results are not
provided with convergence theory. (ii) {\em Stochastic steplength
	sequences:} An alternative to a deterministic rule is a stochastic
	scheme that updates steplengths based on observed data. Of note is
	recent work by George et al. \cite{George06} where an
	adaptive stepsize rule is proposed that minimizes the mean squared
	error. In a similar vein, Cicek et al. \cite{Zeevi11}
	develop an adaptive Kiefer-Wolfowitz SA algorithm and derive general upper
	bounds on its mean-squared error.

Before proceeding, we note the relationship of the present work to three
specific references. In ~\cite{koshal10single}, Cartesian
stochastic variational inequality problems with Lipschitzian mappings were
considered with a focus towards integrating Tikhonov and prox-based
regularization techniques with standard stochastic gradient methods.
However, the steplength sequences were ``non-adaptive'' since the
choices did not adapt to problem parameters. Two problem-specific
adaptive rules were developed in our earlier work on stochastic convex
programming. Additionally, local smoothing techniques were examined for
addressing the lack of smoothness. Of these, the first, referred to as
the {\em recursive steplength} SA scheme, forms the inspiration for a
generalization pursued in the current work. Finally, in \cite{Farzad2},
we extended this recursive rule to accommodate stochastic variational
inequality problems. Note that the qualifier ``adaptive'' implies that
the steplength rule adapts to problem parameters such as Lipschitz
constant, monotonicity constant and the diameter of the set. In this
paper, our goal lies in developing a distributed adaptive stochastic
approximation scheme (DASA) that can accommodate networked
multi-agent implementations and cope with non-Lipschitzian mappings. More
specifically, the main contributions of this paper are as
follows: 
\vspace{-0.15in}
\paragraph{(i) DASA schemes for Lipschitzian CSVIs:} We begin with a
simple extension of the adaptive stepength rule presented in
\cite{Farzad1} to the variational regime under a Lipschitzian
requirement on the map. Yet, implementing this rule in a centralized
regime is challenging and this motivates the need for distributed
counterparts that can be employed on Cartesian problems. Such a
distributed rule is developed and produces sequences of iterates that
are guaranteed to converge to the solution in almost-sure sense.
\vspace{-0.15in}
\paragraph{(ii) DASA schemes for non-Lipschitzian CSVIs:} Our second
goal lies in addressing the absence or unavailability of a Lipschitz
constant by leveraging locally randomized smoothing techniques, again
inspired by our efforts to solve nonsmooth stochastic optimization
problems~\cite{Farzad1}. In this part of the paper, we generalize this
natively centralized scheme for optimization problems to a distributed
version that can cope with Cartesian stochastic variational inequality
problems. 

The remainder of this paper is organized as follows. 
In Section~\ref{sec:formulation}, we provide a canonical
formulation for the problem of interest and motivate this formulation
through two sets of examples. An adaptive steplength SA
scheme for stochastic variational inequality problems with Lipschitzian
mappings and its distributed generalization are provided in Section~\ref{sec:convergence-Lip}. 
By leveraging a locally randomized smoothing technique, in Section~\ref{sec:convergence-no-Lip}, we
extend these schemes to a regime where Lipschitzian assumptions do not
hold. Finally, the paper concludes with some preliminary numerics in
Section~\ref{sec:numerics}. 

\textbf{Notation:} Throughout this paper, a vector $x$ is assumed to be
a column vector. We write $x^T$ to denote the transpose of a vector $x$,
$\|x\|$ {to denote} the Euclidean vector norm, i.e.,
$\|x\|=\sqrt{x^Tx}$, {$\|x\|_1$ {to denote} the $1$-norm, i.e.,
$\|x\|_1=\sum_{i=1}^n|x_i|$ for $x \in \mathbb{R}^n$, and
$\|x\|_\infty$ {to denote} the infinity vector norm, i.e.,
$\|x\|_\infty=\max_{i=1,\ldots,n}|x_i|$ for $x \in \mathbb{R}^n$}.
We use $\Pi_X(x)$ to denote the Euclidean projection of a vector $x$ on
a set $X$, i.e., $\|x-\Pi_X(x)\|=\min_{y \in X}\|x-y\|$. {For a convex
function $f$ with domain dom$f$, a vector $g$ is a \textit{subgradient} of $\bar x \in \hbox{dom}f$ if $f(\bar
x) +g^T(x-\bar x) \leq f(x)$ holds} for all $x \in \hbox{dom}f.$ The set
of all subgradients of $f$ at $\bar x$ is denoted by $\partial f(\bar
		x)$. We write \textit{a.s.} as the abbreviation for ``almost
surely''. {We use $\Prob(A)$ to denote the probability of an event $A$
	and} $\EXP{z}$ to denote the expectation of a random variable~$z$.
	The \texttt{Matlab} notation $(u_1;u_2;u_3)$ refers to a column
	vector with components $u_1$, $u_2$ and $u_3$, respectively.

\section{Formulation and source problems}\label{sec:formulation}
In Section~\ref{sec:2.1}, we formulate the Cartesian stochastic variational
inequality (CSVI) problem and outline
the stochastic approximation algorithmic framework. A motivation
for studying CSVIs is provided through two examples
in Section~\ref{sec:2.2}, while a review of the
main assumptions is given in Section~\ref{sec:2.3}.

\subsection{Problem formulation and algorithm outline}\label{sec:2.1}
Given a set $X \subseteq \Real^n$ and a mapping $F:X \to \Real^n$, the 
variational inequality problem, denoted {by} VI$(X,F)$, requires determining a vector $x^* \in X$ such that 
$
(x-x^*)^TF(x^*) \geq 0 
$ holds for all $x \in X$. When the underlying set $X$ is given by a Cartesian product, as
	articulated by the definition $X \triangleq \prod_{i=1}^N X_i$, where
	$X_i \subseteq \Real^{n_i}$, then
		the associated variational inequality is qualified as a {\em
			Cartesian} variational inequality problem. 
			Now suppose that ${x^*=(x_1^*; x_2^*;\ldots; x_N^*)} \in X$ 
satisfies the following system of inequalities:
 \begin{align}\label{eqn:CVI}
(x_i-x_i^*)^T\EXP{\Phi_i(x^*,\xi_i)} \geq 0 \qquad \hbox{ for all $x_i \in
	X_i$ and all $i=1,\ldots,N$},
\end{align} 
where $\xi_i:{\Omega_i} \rightarrow R^{d_i}$ is a random vector with some
	 probability distribution for $i = 1, \hdots, N$. 
Naturally, problem (\ref{eqn:CVI}) may be reduced to VI$(X,F)$ by noting that $F$
may be defined as in~\eqref{def-F}, 
 where $n=\sum_{i=1}^N n_i$ and 
$F:X\to\Real^{n}$. Then, VI$(X,F)$ is a stochastic variational
inequality problem on the Cartesian product of {the} sets $X_i$
with a solution {$x^*=(x_1^*; x_2^*;\ldots; x_N^*)$}.

Much of the interest in this paper pertains to the development of
stochastic approximation schemes for 
VI$(X,F)$ when the components the map $F$ is defined by
\eqref{def-F}. For such a problem, we consider the following distributed stochastic
approximation scheme:
\begin{align}
\begin{aligned}
x_{{k+1},i} & =\Pi_{X_i}\left(x_{k,i}-\g_{k,i} ( F_i(x_k)+w_{k,i})\right),\cr 
w_{k,i} &\triangleq  \sfy{\Phi_i(x_{k}},\xi_{k,i})-F_i(x_k),
\end{aligned}\label{eqn:algorithm_different}
\end{align}
for all $k\ge 0$ and $i=1,\dots,N$, where {$F_i(x)\triangleq \EXP{\Phi_i(x,\xi_{i})}$ for $i=1,\ldots,N$,}
$\g_{k,i} >0$ is the stepsize {for }the $i$th {index} at iteration $k$,
{$x_{k,i}$ denotes the solution for the $i$-th index at iteration $k$, and
$x_k =(x_{k,1};\  x_{k,2};\ \ldots;\ x_{k,N})$.} Moreover, $x_0 \in X$ is a random
initial vector independent of {any other random variables in the scheme and such}
that $\EXP{\|x_0\|^2}<\infty$. 

\subsection{Motivating examples}\label{sec:2.2}
We consider two problems that can be addressed by Cartesian stochastic
variational inequality framework.
\begin{example}[Networked stochastic Nash-Cournot game]
\label{ex:st-nc}
\rm  A classical example of a Nash game is a networked Nash-Cournot
game~\cite{metzler03nash-cournot,kannan10online}.
Suppose a collection of $N$ firms compete over a network of $M$ nodes
wherein the production and sales for firm $i$ at node $j$ are denoted by
$g_{ij}$ and $s_{ij}$, respectively. Suppose firm $i$'s cost of production at node
$j$ is denoted by the uncertain cost function $c_{ij}(g_{ij},\xi)$.
Furthermore, goods sold by firm $i$ at node $j$ fetch a random revenue
defined by $p_j(\bar s_j,\xi)s_{ij}$ where $p_j(\bar s_j,\xi)$
denotes the uncertain sales price
at node $j$ and $\bar s_j = \sum_{i=1}^N s_{ij}$ denotes the aggregate
sales at node $j$. Finally, firm $i$'s production at node $j$ is capacitated
by $\mathrm{cap}_{ij}$ and its optimization problem is given by the
following\footnote{Note that the transportation costs are assumed to be
	zero.}:
\begin{align*}
\displaystyle \mbox{minimize} & \qquad  {\EXP{f_i(x,\xi)}} \\
\mbox{subject to} & \qquad x_i \in X_i,
\end{align*}
where \sfy{$x=(x_1;\ldots; x_{N})$} with $x_i=(g_i; s_i)$, \sfy{$g_i=(g_{i1};\ldots;g_{iM})$, 
$s_i=(s_{i1};\ldots;s_{iM})$}, and
\begin{align*}
f_i(x,\xi) & \triangleq \sum_{j =1}^M \left( c_{ij}
		(g_{ij},\xi) - p_j(\bar s_j,\xi)
		s_{ij}\right), 
		\\ 
X_i & \triangleq  \left\{
(g_i,s_i) \mid \sum_{j=1}^M
g_{ij}   = \sum_{j=1}^M s_{ij},  \quad g_{ij},\,s_{ij}  \geq 0,\quad g_{ij} \leq \mathrm{cap}_{ij},\  j = 1,
	\hdots, M\right\}. \hfill \qed
\end{align*}
\end{example}
Under the validity of the interchange between the expectation and the
derivative operator, the resulting equilibrium conditions of this stochastic Nash-Cournot
game are compactly captured by the variational inequality VI$(X,F)$
where 
$ X \triangleq \prod_{i=1}^N X_i$ and \sfy{$F(x) = (F_1(x);\ldots; F_N(x))$}
with $F_i(x)=\EXP{\nabla_{x_i} f_i(x,\xi)}$.
\begin{example}[Stochastic composite minimization problem]
\rm Consider a generalized min-max optimization problem given by 
\begin{align}\label{gen-minmax}
\displaystyle \mbox{minimize} & \qquad 
\Psi(\psi_1(x), \hdots, \psi_m(x)) \\
\mbox{subject to} & \qquad x \in X \triangleq \prod_{i=1}^N X_i,
\end{align}
where  
$\Psi(u_1, \hdots, u_m)$ is defined as 
$$ \Psi(u_1, \hdots, u_m) \triangleq  \max_{y \in {\mathcal Y}} \left\{
\sum_{i=1}^m u_i^T(A_i y + b_i) - \beta(y)\right\},$$ 
while ${\psi_i(x)} \triangleq \EXP{{\phi_i}(x,\xi)}$, 
$\nabla_{x_j} \psi_i(x) = \EXP{ H_{ji}(x,\xi)}$ for $i = 1, \hdots, m,$ 
and ${\beta(y)}$ is a Lipschitz continuous convex function of $y$. 
\hfill \qed\end{example}
Under the assumption that the derivative and the {e}xpectation operator
can be interchanged, it can be seen that the solution to this optimization problem can be obtained by solving a
Cartesian stochastic variational inequality problem VI$(X\times {\mathcal Y}, F)$ where 
$$ F(x,y) \triangleq \displaystyle \pmat{ \sum_{i=1}^m \nabla_{x_1} {\psi_i(x)} (A_i y
		+b_i) \\
		\vdots \\\sum_{i=1}^m \nabla_{x_N} {\psi_i(x)} (A_i y
		+b_i) \\
		-\sum_{i=1}^m A_i^T \psi_i(x) + \nabla_y \beta(y)} 
		= \pmat{	\mathsf{E} \left[\sum_{i=1}^m H_{1i}(x,\xi) (A_i y
		+b_i)\right] \\
		\vdots \\ \mathsf{E} \left[\sum_{i=1}^m  H_{Ni}(x,\xi) (A_i y
		+b_i)\right]\\
		\mathsf{E} \left[-\sum_{i=1}^m A_i^T \phi_i(x,\xi) + \nabla_y
		\beta(y) \right]}.$$
Note that the specification that ${\psi_i(x)}$ and its Jacobian are
expectation-valued may be a consequence of not having access to
noise-free evaluations of either object. In particular, one only has
access to evaluations $\phi_i(x,\xi)$ and Jacobian evaluations given
by $H_{ji}(x,\xi) = \nabla_{x_j} {\phi_i(x,\xi)}$.

\subsection{Assumptions}\label{sec:2.3}
Our interest lies in the development of distributed
stochastic approximation schemes for Cartesian stochastic variational
inequality problems as espoused
		by~\eqref{eqn:algorithm_different} and the associated global convergence
theory in regimes where the mappings
are single-valued mappings that are not necessarily Lipschitz continuous. 
We let 
\[X=\prod_{i=1}^N X_i,\]
and make the following assumptions on the set $X$ and the mapping $F$.
\begin{assumption}\label{assum:different_main}
Assume the following:
\begin{enumerate} 
\item[(a)] The set $X_i \subseteq \Real^{n_i}$ is closed and convex for
$i = 1, \hdots, N$.
\item[(b)] The mapping $F(x)$ is a single-valued Lipschitz continuous
over the set $X$ with a constant $L$.
\item[(c)] The mapping $F(x)$ is strongly monotone with a constant $\eta >0$: 
$$ (F(x)-F(y))^T(x-y) \geq \eta \|x-y\|^2\qquad \hbox{for all }x,y \in {X}.$$
\end{enumerate}
\end{assumption}
Since $F$ is strongly monotone, 
the existence and uniqueness of the solution to VI$(X,F)$ is guaranteed by 
Theorem 2.3.3 of~\cite{facchinei02finite}. 
We let $x^*$ denote the solution of VI$(X,F)$. 

Regarding the method in~\eqref{eqn:algorithm_different}, we let
$\sF_k$ denote the history of the method up to time $k$, i.e., 
$\sF_k=\{x_0,\xi_0,\xi_1,\ldots,\xi_{k-1}\}$ for $k\ge 1$ and $\sF_0=\{x_0\}$,
where $\xi_k=(\xi_{k,1}\ ;  \xi_{k,2}\ ; \ldots \ ;\xi_{k,N})$. In terms of this definition, we note that 
\[\EXP{w_{k,i} \mid \sF_k} =\EXP{\sfy{\Phi_i(x_{k}},\xi_{k,i}) \mid\sF_k}-F_i(x_k)=0\qquad
\hbox{for all $k\ge0$ and all $i$}.\]
We impose some further conditions on the stochastic errors $w_{k,i}$ of the algorithm, as follows.
\begin{assumption}\label{assum:w_k_bound} 
The errors {$w_k=(w_{k,1}; w_{k,2}; \ldots; w_{k,N})$} are such that 
for some (deterministic) $\nu >0$,
\[\EXP{\|w_k\|^2 \mid \sF_k} \le \nu^2 \qquad \hbox{\textit{a.s.} for all $k\ge0$}.\]  
\end{assumption}
We use the following Lemma in establishing the convergence of
method~(\ref{eqn:algorithm_different}) and its extensions. This result
may be found in~\cite{Polyak87} (cf.~Lemma 10, page 49).  
\begin{lemma}\label{lemma:probabilistic_bound_polyak}
Let $\{v_k\}$ be a sequence of nonnegative random variables, where 
$\EXP{v_0} < \infty$, and let $\{\a_k\}$ and $\{\mu_k\}$
be deterministic scalar sequences such that:
\[\EXP{v_{k+1}|v_0,\ldots, v_k} \leq (1-\alpha_k)v_k+\mu_k
\qquad a.s. \ \hbox{for all }k\ge0,\]
\[0 \leq \alpha_k \leq 1, \quad\ \mu_k \geq 0, \qquad \hbox{for all }k\ge0,\]
\begin{align*}
\sum_{k=0}^\infty \alpha_k =\infty, 
\quad\ \sum_{k=0}^\infty \mu_k < \infty, 
\quad\ \lim_{k\to\infty}\,\frac{\mu_k}{\alpha_k} = 0. 
\end{align*}
Then, $v_k \rightarrow 0$ almost surely{, 
$\lim_{k\to\infty}\EXP{v_k}= 0$, and for any $\epsilon >0$ and for all $k>0$,
\[\Prob\left(\{v_j \leq \e \hbox{ for all } j \geq k\}\right)
\geq 1 - \frac{1}{\e}\left(\EXP{v_k}+\sum_{i=k}^\infty \beta_i\right)
.\]}
\end{lemma} 

\section{Distributed adaptive SA schemes for Lipschitzian
	mappings}\label{sec:convergence-Lip}
In this section, we restrict our attention to settings where the mapping $F(x)$ is
a single-valued Lipschitzian map. In Section~\ref{sec:asa}, we begin by developing an
adaptive steplength rule for deriving steplength sequences from problem
parameters such as monotonicity constant, Lipschitz constant etc., where
the qualifier adaptive implies that the steplength choices ``adapt'' or
are ``self-tuned'' to problem parameters. Unfortunately, in distributed
regimes, such a rule requires prescription by a central coordinator, a
relatively challenging task in multi-agent regimes. This motivates the
development of a distributed counterpart of the aforementioned adaptive
rule and provide convergence theory for such a generalization in
Section~\ref{sec:dasa}. 

\subsection{An adaptive steplength SA (ASA) scheme} \label{sec:asa}
Stochastic approximation algorithms 
require stepsize sequences to be square summable
but not summable. These algorithms provide little advice regarding the choice of
such sequences. One of the most common choices has been the harmonic
steplength rule which takes the form of $\gamma_k=\frac{\theta}{k}$ where $\theta>0$ is a
constant. Although, this choice guarantees almost-sure convergence, it does
not leverage problem parameters. Numerically, it has been observed that
such choices can perform quite poorly in practice. Motivated by this
shortcoming, we present a steplength scheme for a centralized variant of
algorithm~\eqref{eqn:algorithm_different}:
\begin{align}
\begin{aligned}
x_{{k+1}} & =\Pi_{X}\left(x_{k}-\g_{k} ( F(x_k)+w_{k})\right),\cr 
w_{k} &\triangleq  \sfy{\Phi(x_{k}},\xi_{k})-F(x_k),
\end{aligned}\label{eqn:algorithm_centralized}
\end{align}
for $k \ge 0$. The proposed scheme derives a rule for updating steplength
sequences that adapts to problem parameters while guaranteeing
almost-sure convergence of $x_k$ to the unique solution of VI$(X,F)$. 

A key challenge in practical implementations of stochastic approximation
lies in choosing an appropriate diminishing steplength sequence
$\{\g_k\}$. In \cite{Farzad1}, we developed a rule for selecting such a
sequence in a convex stochastic optimization regime by leveraging three
parameters: (i) Lipschitz constant of the gradients; (ii) strong
convexity constant; and (ii) diameter of the set $X$. Along similar
directions, such a rule is constructed for strongly monotone stochastic variational
inequality problems and the results in this subsection bear significant
similarity to those presented in \cite{Farzad1} with some
key distinctions. First, these results are presented for 
strongly monotone stochastic variational inequality problems and second,
		 co-coercivity of the mappings is not assumed, leading to a
tighter requirement on the choice of steplengths. 
\begin{lemma}\label{lemma:error_bound_cent}
Consider algorithm~\eqref{eqn:algorithm_centralized}, and let
	Assumptions~\ref{assum:different_main} and~\ref{assum:w_k_bound}
	hold. Then, the following relation holds almost surely for all $k
	\ge 0$:
\begin{align} \label{rate_cent_nu}
\EXP{\|x_{k+1}-x^*\|^2\mid\sF_k} 
\leq 	 (1-2\eta\gamma_k+L^2\gamma_k^2)\|x_k-x^*\|^2+{\gamma_k^2\nu^2}.
\end{align} 
\end{lemma}

\begin{proof}
By the definition of algorithm~\eqref{eqn:algorithm_centralized} and the non-expansiveness property of the
projection operator, we have for all $k\ge0$,
\begin{align*}
 \|x_{k+1}-x^*\|^2  
 &=\|\Pi_{X}(x_{k}-\g_{k}(F(x_k)+w_{k}))  -\Pi_{X}(x^*-\g_{k} F(x^*))\|^2 \\
 & \le \|x_{k}-x^*-\g_{k}(F(x_k)+w_{k}-F(x^*))\|^2.
\end{align*}	
Taking expectations conditioned on the past, and by employing 
$\EXP{w_{k}\mid \sF_k}=0$, we have
\begin{align*}
 \EXP{\|x_{k+1}-x^*\|^2\mid\sF_k} &\le \|x_{k}-x^*\|^2
 +\gamma_{k}^2\|F(x_k)-F(x^*)\|^2
 +\g_{k}^2\EXP{\|w_{k}\|^2\mid \sF_k}\\ 
&-2\gamma_{k} (x_{k}-x^*)^T(F(x_k)-F(x^*)) \\
& \leq (1-2\eta \gamma_k + \gamma_k^2 L^2) \|x_k-x^*\|^2 + \g_k^2 \nu^2,
\end{align*}
where the second inequality is a consequence of the strong monotonicity
and Lipschitz continuity
of $F(x)$ over $X$ as well as the boundedness of $\EXP{\|w_{k}\|^2\mid
	\sF_k}.$
\end{proof}

The upper bound~\eqref{rate_cent_nu} can be used to construct an adaptive stepsize
rule. Note that inequality (\ref{rate_cent_nu}) holds for any $\gamma_k >0$.
When the stepsize is further restricted so that $0 < \g_k
\le\frac{\eta}{L^2}$, we have
\[1- \g_k (2\eta-\g_k L^2) \leq 1-\eta  \g_k.\]
Thus, for $0 < \g_k \le\frac{\eta}{L^2}$ and by taking expectations, inequality
(\ref{rate_cent_nu}) reduces to
 \begin{equation}\label{rate_nu_est}
\EXP{\|x_{k+1}-x^*\|^2} 
\le (1-\eta \g_k)\EXP{\|x_k-x^*\|^2}+ \g_k^2\nu^2\qquad\hbox{for all }k \geq 0.
\end{equation}
We begin by viewing 
$\EXP{\|x_{k+1}-x^*\|^2}$ as an error $e_{k+1}$ arising
from employing the stepsize sequence $\g_0,\g_1,\ldots, \g_k$.
Furthermore, the worst case error arises when~\eqref{rate_nu_est} holds
as an equality and satisfies the following recursive relation:
\[e_{k+1}(\g_0,\ldots,\g_k) = (1-\eta \gamma_k) e_k(\g_0,\ldots,\g_{k-1}) 
+ \gamma_k^2 \nu^2.\]
Motivated by this relationship, our interest lies in examining whether 
the stepsizes
$\g_0,\g_1\ldots,\g_k$ can be chosen so as to minimize the error
$e_{k}$. {Our goal lies in constructing a stepsize scheme that allows
	for claiming the almost sure convergence of the sequence $\{x_k\}$
		produced by algorithm (\ref{eqn:algorithm_centralized}) to the unique solution $x^*$ of VI$(X,F)$.  We formalize this approach by defining real-valued error functions $e_k(\g_0,\ldots,\g_{k-1})$ as follows:
\begin{align}\label{def:e_k}
  e_{k}(\g_0,\ldots,\g_{k-1})
& \triangleq (1-\eta \g_{k-1})e_{k-1}(\g_0,\ldots,\g_{k-2}) +\g_{k-1}^2\nu^2
\qquad\hbox{for $k\ge 1$},
\end{align}
where $e_0$ is a positive scalar, $\eta$ is the strong monotonicity
constant and $\nu^2$ is an upper bound for the second moments of the error 
norms $\|w_k\|$. We consider a choice of $\{\gamma_0, \gamma_1, \hdots,
	  \g_{k-1}\}$ based on minimizing an upper bound on the mean-squared
		  error, namely $e(\g_0,\g_1, \hdots, \g_{k-1})$, as captured by
		 the following optimization problem:
\begin{equation*}
\begin{split}
&\begin{array}{ll}
  \hbox{minimize } & e_{k}(\gamma_0,\ldots,	\gamma_{k-1})\cr
  \hbox{subject to } & 0<\gamma_j \leq \frac{\eta}{L^2} \hbox{ for all }j=0,\ldots, k-1.
				 \end{array}
\end{split}\end{equation*}
To ensure convergence in an almost-sure sense, the sequence $\{\g_k\}$
	needs to satisfy $\sum_{j=0}^\infty \g_j =\infty$ and
	$\sum_{j=0}^\infty \g_j^2 <\infty$. As the next two propositions
	show, these can indeed be achieved. In fact, the error $e_{k+1}$ at
	the next iteration can also be minimized by selecting $\g_{k}$ as a
	function of only the most recent stepsize $\g_{k-1}$.} In what
	follows, we consider the sequence $\{\g^*_k\}$ given by 
\begin{align}
& \gamma_0^*=\frac{\eta}{2\nu^2}\,e_0\label{eqn:optimal_self_adaptive_2}
\\
& \gamma_{k}^*=\gamma_{k-1}^*\left(1-\frac{\eta}{2}\gamma_{k-1}^*\right) 
\quad \hbox{for all }k\ge 1.\label{eqn:optimal_self_adaptive}
\end{align}
We provide a result showing that the stepsizes $\g_i$,
$i=0,\ldots,k-1,$ minimize 
$e_k$ over $(0,\frac{\eta}{L^2}]^{k}$, where $L$ is the Lipschitz
		constant associated with $F(x)$ over $X$. 
		
\begin{proposition}[An adaptive steplength SA (ASA) scheme]\label{prop:adap_cent_results}
Let the error function $e_k(\g_0,\ldots,\g_{k-1})$ be defined as in~\eqref{def:e_k},
where {$e_0\geq 0$} is such that {$\nu \geq L\,\sqrt{\frac{e_0}{2}}$},
where $L$ is the Lipschitz constant of~$F$.
Let the sequence $\{\g^*_k\}$ be given by 
\eqref{eqn:optimal_self_adaptive_2}--\eqref{eqn:optimal_self_adaptive}.
Then, the following hold:
\begin{itemize}
\item [(a)] {For all $k\ge0$, the error $e_k$ satisfies 
$e_k(\g_0^*,\ldots,\g_{k-1}^*) = \frac{2\nu^2}{\eta}\,\g_k^*$.}
\item [(b)] For each $k\ge 1$, the vector 
$(\gamma_0^*, \gamma_1^*,\ldots,\gamma_{k-1}^*)$ 
is the minimizer of the function $e_k(\g_0,\ldots,\g_{k-1})$ over the set
$$\mathbb{G}_k 
\triangleq \left\{\alpha \in \Real^k : 0 < \alpha_j \leq \frac{\eta}{L^2}
\hbox{ for }j =1,\ldots, k\right\}.$$
More precisely, for any $k\ge 1$ and any 
$(\g_0,\ldots,\g_{k-1})\in \mathbb{G}_k$, we have
	\begin{align*}
	e_{k}(\g_0,\ldots,\g_{k-1}) 
         - e_{k}(\g_0^*,\ldots,\g_{k-1}^*) \geq \nu^2(\g_{k-1}-\g_{k-1}^*)^2.
	\end{align*}	
\end{itemize}
\end{proposition} 
The almost-sure convergence of the produced sequence holds for a family
of steplength rules, as captured by the folloing result.

\begin{proposition}[Almost-sure convergence of ASA scheme]
\label{prop:self_adaptive}\label{prop:rec_convergence}
Let Assumptions~\ref{assum:different_main} and \ref{assum:w_k_bound} hold. Assume that the stepsize sequence $\{\g_k\}$ is generated by
the following adaptive scheme:
\begin{align*}
\gamma_{k}=\gamma_{k-1}(1-c\gamma_{k-1}) 
\qquad \hbox{for all }k \geq 1,
\end{align*}
where $c>0$ is a scalar and the initial stepsize is such that
$0<\g_0<\frac{1}{c}$. Then, the sequence $\{x_k\}$ generated by
algorithm (\ref{eqn:algorithm_centralized}) converges almost surely to a random
point that belongs to the optimal set.
\end{proposition}
The proofs of Propositions~\ref{prop:adap_cent_results} and~\ref{prop:rec_convergence} are omitted, 
as they follow 
		from a more general results for a distributed SA method, as discussed
		in the next subsection. 

\subsection{A distributed adaptive steplength SA (DASA)
	scheme}\label{sec:dasa}
Unfortunately, in multi-agent regimes, the implementation of the stepsize rule (\ref{eqn:optimal_self_adaptive_2})-(\ref{eqn:optimal_self_adaptive})
requires a central coordinator who can prescribe and enforce such rules.
In this section, we extend the centralized rule to accommodate a
multi-agent setting wherein each agent chooses its own update rule,
	given the global knowledge of problem parameters. In such a regime, given that 
	the set $X$ is the Cartesian product of closed and
	convex sets $X_1, \hdots, X_N$,  our interest lies
	in developing steplength update rules in the context of
	method~\eqref{eqn:algorithm_different} where the $i$-th
	agent chooses its steplength, denoted by $\gamma_{k,i}$, as per 
\[\g_{k,i}=\g_{k-1,i}(1-c_i\g_{k-1,i}),\]
with $c_i>0$ being a constant associated with agent $i$ mapping $F_i(x)$, while 
the initial stepsize $\g_{0,i}$ is suitably selected. The
	following assumption imposes requirements on the stepsizes
		$\gamma_{k,i}$ in~\eqref{eqn:algorithm_different}.
\begin{assumption}\label{assum:step_error_sub}
Assume that the following hold:
\begin{enumerate} 
 \item[(3a)] The stepsize sequences {$\{\gamma_{k,i}\}$, $i=1,\ldots, N$,}
 are such that $\gamma_{k,i}>0$ for all $k$ and $i$, with
$\sum_{k=0}^\infty \gamma_{k,i}=\infty$ and $\sum_{k=0}^\infty
\gamma_{k,i}^2 < \infty$ for all~$i$.
\item [(3b)] If $\{\delta_k\}$ and $\{\Gamma_k\}$ are positive
	sequences such that 
$\delta_k \leq \min_{1\le i\le N}{\g_{k,i}}$ and $\Gamma_k \geq \max_{1\le i\le N}{\g_{k,i}}$ for all 
$k\ge0$, then 
\[\frac{\Gamma_k-\delta_k}{\delta_k}\le \beta 
\qquad\hbox{for all $k\ge 0$},\]
where $\beta$ is a scalar satisfying $0\le \beta<\frac{\eta}{L}$.
\end{enumerate}
\end{assumption}
{\textbf{Remark:} }Assumption~ (\ref{assum:step_error_sub}a) is a
standard requirement on steplength sequences while
Assumption~(\ref{assum:step_error_sub}b) provides an additional condition on 
the discrepancy between the stepsize values $\gamma_{k,i}$ at each iteration $k$. 
This condition is satisfied, for instance, 
when $\gamma_{k,1}=\ldots=\gamma_{k,N}$, in which case $\beta=0$.

When deriving an adaptive rule, we use Lemma~\ref{lemma:probabilistic_bound_polyak} 
and a distributed generalization of Lemma~\ref{lemma:error_bound_cent}, which is 
given below.
	
\begin{lemma}\label{lemma:error_bound_1}
Consider algorithm~\eqref{eqn:algorithm_different}. Let
	Assumptions~\ref{assum:different_main} and~\ref{assum:w_k_bound}
	hold. 
\begin{enumerate}
	\item[(a)] The following relation holds almost surely for all $k\ge 0$:
\begin{align*}
\EXP{\|x_{k+1}-x^*\|^2\mid\sF_k} 
\leq 	 (1-2(\eta+L)\delta_k+2L\Gamma_k+L^2\Gamma_k^2)\|x_k-x^*\|^2+\an{\Gamma_k^2\nu^2},
\end{align*} 
where $\{\delta_k\}$ and $\{\Gamma_k\}$ are positive sequences such that
$\delta_k \leq \min_{1\le i\le N}{\g_{k,i}}$ and $\Gamma_k \geq \max_{1\le i\le N}{\g_{k,i}}$ for all $k$.
\item [(b)] If Assumption (\ref{assum:step_error_sub}b) holds, then 
the following relation is valid for all $k \ge 0$:
\begin{align*}
\EXP{\|x_{k+1}-x^*\|^2} 
\leq (1-2(\eta-\beta L)\delta_k+(1+\beta)^2L^2\delta_k^2)\EXP{\|x_	k-x^*\|^2}
+(1+\beta)^2\delta_k^2\nu^2.
\end{align*} 
\end{enumerate}
\end{lemma}

\begin{proof}
\noindent (a) 
From the properties of the projection operator, we know that a vector $x^*$ solves VI$(X,F)$ 
problem if and only if $ x^*$ satisfies $x^*=\Pi_X(x^*-\g F(x^*))$ for any $\gamma>0$.
By the definition of algorithm \eqref{eqn:algorithm_different} and the non-expansiveness property of the
projection operator, we have for all $k\ge0$ and all~$i$,
\begin{align*}
 \|x_{k+1,i}-x_i^*\|^2  
 &=\|\Pi_{X_i}(x_{k,i}-\g_{k,i}(F_i(x_k)+w_{k,i}))  -\Pi_{X_i}(x_i^*-\g_{k,i} F_i(x^*))\|^2 \\
 & \le \|x_{k,i}-x_i^*-\g_{k,i}(F_i(x_k)+w_{k,i}-F_i(x^*))\|^2.
\end{align*}	
Taking the expectation conditioned on the past, and using
$\EXP{w_{k,i}\mid \sF_k}=0$, we have
\begin{align*}
 \EXP{\|x_{k+1,i}-x_i^*\|^2\mid\sF_k} &\le \|x_{k,i}-x_i^*\|^2
 +\gamma_{k,i}^2\|F_i(x_k)-F_i(x^*)\|^2
 +\g_{k,i}^2\EXP{\|w_{k,i}\|^2\mid \sF_k}\\ 
&-2\gamma_{k,i} (x_{k,i}-x_i^*)^T(F_i(x_k)-F_i(x^*)).
\end{align*}
Now, by summing the preceding relations over $i$, we have
\begin{align*}
\EXP{\|x_{k+1}-x^*\|^2\mid\sF_k} & \le  \|x_{k}-x^*\|^2 
+\sum_{i=1}^N\gamma_{k,i}^2\|F_i(x_k)-F_i(x^*)\|^2 +\sum_{i=1}^N\g_{k,i}^2
\EXP{\|w_{k,i}\|^2\mid \sF_k}\cr 
&-2\sum_{i=1}^N\gamma_{k,i} (x_{k,i}-x_i^*)^T(F_i(x_k)-F_i(x^*)).
\end{align*}
Using $\g_{k,i}\le \Gamma_k$ and Assumption~\ref{assum:w_k_bound}, we can see that
$\sum_{i=1}^N\g_{k,i}^2\EXP{\|w_{k,i}\|^2\mid \sF_k}\le \Gamma_k^2\nu^2$ almost surely for all 
$k\ge0$.
Thus, from the preceding relation, we have
\begin{align}\label{ineq:main1}
\EXP{\|x_{k+1}-x^*\|^2\mid\sF_k} \le & \|x_{k}-x^*\|^2 +\underbrace{\sum_{i=1}^N\gamma_{k,i}^2\|F_i(x_k)-F_i(x^*)\|^2}_{\bf Term\,1} 
+{\Gamma_k^2\nu^2}\cr
&\underbrace{-2\sum_{i=1}^N\gamma_{k,i} (x_{k,i}-x_i^*)^T(F_i(x_k)-F_i(x^*))}_{\bf Term\,2}.
\end{align}
Next, we estimate Term $1$ and Term $2$ in (\ref{ineq:main1}).
By using the definition of $\Gamma_k$ and by leveraging the Lipschitzian
	property of mapping $F$, we obtain
\begin{align}\label{ineq:main2}
\hbox{Term}\,1 \le  \Gamma_k^2\|F(x_k)-F(x^*)\|^2\le\Gamma_k^2L^2\|x_k-x^*\|^2. 
\end{align}
By adding and subtracting
$-2\sum_{i=1}^N\delta_k(x_{k,i}-x_i^*)^T(F_i(x_k)-F_i(x^*))$ {from} Term $2$, 
{and using $\sum_{i=1}^N(x_{k,i}-x^*_i)^T(F_i(x_k)-F_i(x^*))=(x_{k}-x^*)^T(F(x_k)-F(x^*))$},
we further obtain
\begin{equation*}
\begin{split}
\hbox{Term}\,2  \le 
 & -2\delta_k(x_{k}-x^*)^T(F(x_k)-F(x^*)) -2\sum_{i=1}^N(\g_{k,i}
 -\delta_k)(x_{k,i}-x_i^*)^T(\an{F_i(x_{k}) - F_i(x^*)}).
\end{split}
\end{equation*} 
By Cauchy-Schwartz inequality, the  preceding relation yields
\begin{equation*}
\begin{split}
\hbox{Term}\,2  \le & -2\delta_k(x_{k}-x^*)^T(F(x_k)-F(x^*)) +2(\g_{k,i}-\delta_k)\sum_{i=1}^N\|x_{k,i}-x_i^*\|\|\an{F_i(x_{k}) - F_i(x^*)}\|\cr
 \le & -2\delta_k(x_{k}-x^*)^T(F(x_k)-F(x^*)) +2(\Gamma_k-\delta_k)\|x_{k}-x^*\|\|F(x_{k})-F(x^*)\|,
\end{split}
\end{equation*}
where in the last relation, we use the definition of $\Gamma_k$ and {H\"{o}lder's} inequality. Invoking strong monotonicity of the mapping $F$ for bounding the first term and by
utilizing the Lipschitzian property of the second term of the preceding relation, we have
\begin{equation}\label{ineq:term2}
\begin{split}
\hbox{Term}\,2  \le -2\eta\delta_k\|x_{k}-x^*\|^2 +2(\Gamma_k-\delta_k)L\|x_{k}-x^*\|^2.
\end{split}
\end{equation}
The desired inequality is obtained by combining 
relations (\ref{ineq:main1}), (\ref{ineq:main2}), and (\ref{ineq:term2}).

\noindent (b)  Assumption \ref{assum:step_error_sub}b implies that $\Gamma_k \leq
(1+\beta)\delta_k$. Combining this observation with the result of part (a), we obtain
almost surely for all $k \ge 0$,
\begin{align*}
\EXP{\|x_{k+1}-x^*\|^2\mid\sF_k} 
\leq & (1-2(\eta-\beta L)\delta_k+(1+\beta)^2L^2\delta_k^2)\|x_	k-x^*\|^2
+(1+\beta)^2\delta_k^2\an{\nu^2}.
\end{align*}  
Taking expectations in the preceding inequality, we obtain the desired relation.
\end{proof}
The following proposition proves the almost-sure convergence of the
	distributed SA scheme when the steplength sequences satisfy the
		bounds prescribed by Assumption ~\ref{assum:step_error_sub}b.
		
\begin{proposition}[Almost-sure convergence of distributed SA scheme]
\label{prop:rel_bound} 
Let Assumptions \ref{assum:different_main}, \ref{assum:w_k_bound}, and
\ref{assum:step_error_sub} hold. Then, the sequence $\{x_k\}$ generated
by algorithm~\eqref{eqn:algorithm_different} converges almost surely
to the unique solution of VI$(X,F)$.
\end{proposition}
\begin{proof}
Consider the relation of Lemma~\ref{lemma:error_bound_1}(a). For this relation,
we show that the conditions of Lemma
\ref{lemma:probabilistic_bound_polyak} are satisfied, which will allow us to
claim  the almost-sure convergence of $x_k$ to $x^*$. 
Let us define $v_k
\triangleq \an{\|x_{k}-x^*\|^2}$, 
and 
\begin{equation}\label{eqn:almudef}
\alpha_k \triangleq 2(\eta-\beta
		L)\delta_k-L^2\delta_k^2(1+\beta)^2,
		\qquad
\mu_k \triangleq (1+\beta)^2\delta_k^2\nu^2.
\end{equation}
Next, we show that $0\le \a_k\le 1$ for $k$ sufficiently large.  Since
$\g_{k,i}$ tends to zero for all $i=1,\ldots,N$, we may conclude that
	$\delta_k$ goes to zero as $k$ grows. In turn, as $\delta_k$ goes to
	zero, for $k$ large enough, say $k\ge k_1$, we have
\[1-\frac{(1+\beta)^2L^2\delta_k}{2(\eta-\beta L)}>0.\]
By Assumption \ref{assum:step_error_sub}b we have $\beta < \frac{\eta}{L}$, which implies
$\eta-\beta L>0$. Thus,  we have $\alpha_k \geq 0$ for $k\ge k_1$. 
Also, for $k$ large enough, say $k\ge k_2$, we have $\alpha_k \leq 1$
since $\delta_k\to0$. Therefore, when $k\ge\max\{k_1,k_2\}$ we have $0 \leq \alpha_k \leq 1$. 
Obviously, $v_k, \mu_k \geq 0$. 

From Assumption~\ref{assum:step_error_sub}b we have
$\delta_k\le \gamma_k\le (1+\beta)\delta_k$ for all $k$. Using these relations 
and the conditions on $\g_{k,i}$
given in Assumption~\ref{assum:step_error_sub}a, we can show that
$\sum_{k=0}^\infty \delta_k=\infty$ and 
$\sum_{k=0}^\infty \delta_k^2 <\infty.$
Furthermore, from the preceding properties of the sequence $\{\delta_k\}$, and 
the definitions of $\alpha_k$ and $\mu_k$ in~\eqref{eqn:almudef},
we can see that $\sum_{k=0}^\infty\alpha_k = \infty$ and $\sum_{k=0}^\infty\mu_k < \infty$.
Finally, by the definitions of $\alpha_k$ and $\mu_k$ 
we have 
\[{\lim_{k \rightarrow	\infty}}\frac{\mu_k}{\alpha_k}={\lim_{k
	\rightarrow	\infty}}\left(\frac{(1+\beta)^2\delta_k\nu^2}{2(\eta-\beta L)
\left(1-\frac{(1+\beta)^2L^2\delta_k}{2(\eta-\beta L)}\right)}\right)={\frac{(1+\beta)^2(\lim_{k \rightarrow	\infty}\delta_k)\nu^2}{2(\eta-\beta L)
\left(1-\frac{(1+\beta)^2L^2(\lim_{k \rightarrow	\infty}\delta_k)}{2(\eta-\beta L)}\right)},}\]
implying that $\lim_{k \to\infty}\frac{\mu_k}{\alpha_k}=0$ \an{since $\delta_k\to0$}.
Hence, all conditions of Lemma \ref{lemma:probabilistic_bound_polyak}
are satisfied and we may conclude that $\|x_k-x^*\|^2\rightarrow 0$ almost surely. 
\end{proof} 
Proposition~\ref{prop:rel_bound} states that under specified assumptions
	on the set $X$ and mapping $F$, the stochastic errors $w_k$, and the
		stepsizes $\g_{k,i}$, the distributed SA scheme is guaranteed to
		converge to the unique solution of VI$(X,F)$ almost surely. Our
		goal in the remainder of this section lies in providing a stepsize rule that
		aims to minimize a suitably defined error function of the algorithm, while
		satisfying Assumption \ref{assum:step_error_sub}.
		To begin our analysis, we consider the
result of  Lemma \ref{lemma:error_bound_1}b for all $k \geq 0$:
\begin{align}\label{equ:bound_recursive_01}
\EXP{\|x_{k+1}-x^*\|^2} \leq (1-2(\eta-\beta L)\delta_k+(1+\beta)^2L^2\delta_k^2)\EXP{\|x_	k-x^*\|^2}+(1+\beta)^2\delta_k^2\nu^2,
\end{align} 
where $\delta_k\le\min_{1\le i\le N}\g_{k,i}$.
When the stepsizes {$\gamma_{k,i}$} are further restricted so that 
$0<\delta_k\le\frac{\eta-\beta L}{(1+\beta)^2L^2},$ we have
\[
1-2(\eta-\beta L)\delta_k  +(1+\beta)^2L^2\delta_k^2  \leq 1-(\eta -\beta L) \delta_k.
\]
Thus, for  {$0 < \delta_k \le\frac{\eta-\beta L}{(1+\beta)^2L^2}$,}  from inequality (\ref{equ:bound_recursive_01}) we obtain
 \begin{align}\label{rate_nu_est_different}
\EXP{\|x_{k+1}-x^*\|^2} 
\le (1-(\eta-\beta L) \delta_k)\EXP{\|x_k-x^*\|^2} + (1+\beta)^2\delta_k^2\nu^2
\qquad\hbox{for all }k \geq 0.
\end{align}
Similar to the discussion in Section~\ref{sec:asa} in the context of the ASA
	scheme, let us view the quantity
$\EXP{\|x_{k+1}-x^*\|^2}$ as an error $e_{k+1}$ of the method arising
from the use of the lower bounds $\delta_0,\delta_1,\ldots, \delta_k$ for
the stepsize values $\g_{0,i},\g_{1,i}\cdots,\g_{k,i}$, $i=1,\ldots,N$.
Relation~\eqref{rate_nu_est_different} gives us an error estimate for 
algorithm (\ref{eqn:algorithm_different}) in terms of the lower bounds $\delta_0,\delta_1,\ldots, \delta_k$.
We use this estimate to develop an adaptive stepsize procedure. Consider the case when~\eqref{rate_nu_est_different} holds with
	equality, which is the worst case error.  In this case, the error satisfies the following recursive
		relation:
\begin{align*}
 e_{k+1} = (1-(\eta-\beta L) \delta_k) e_k+(1+\beta^2)\nu^2\delta_k^2 
 \qquad\hbox{for all }k \geq 0.
\end{align*}
Let us assume that we want to run the algorithm
(\ref{eqn:algorithm_different}) for a fixed number of iterations, say $K$. 
The preceding relation shows that $e_{K}$ depends on {the lower bound
values up to the $K$th iteration. This motivates us to view the lower
	bounds $\delta_0,\delta_1,\ldots, \delta_{K-1}$ as decision
	variables that can be used to minimize the corresponding } upper
	bound on the mean-squared error of the
algorithm up to iteration $K$. Thus, the variables are {$\delta_0, \delta_1, \ldots,
	\delta_{K-1}$} and the objective function is the error function
	$e_K(\delta_0, \delta_1, \ldots, \delta_{K-1})$. 
	{We proceed to derive a rule for generating lower bounds 
	$\delta_0, \delta_1, \ldots, \delta_{K}$ by minimizing the error $e_{K+1}$.
	Importantly, it turns out that $\delta_{K}$ is a function of only the most
	recent bound} $\delta_{K-1}$. We define the real-valued
	error function $e_k(\delta_0, \delta_1, \ldots, \delta_{k-1})$ by 
	{{considering an equality in}~\eqref{rate_nu_est_different}}: 
\begin{align}\label{def:e_k_different}
 e_{k+1}(\delta_0,\ldots,\delta_k) \triangleq &(1-(\eta-\beta L) \delta_k) e_k(\delta_0,\ldots,\delta_{k-1}) 
 +(1+\beta^2) \nu^2\delta_k^2\qquad\hbox{for all }k \geq 0,
\end{align}
where $e_0$ is a positive scalar, $\{\delta_k\}$ is a sequence of positive scalars 
such that $0 < \delta_k\le\frac{\eta-\beta L}{(1+\beta)^2L^2}$, $L$
	is the Lipschitz constant of 
the mapping $F$, $\eta$ is the strong monotonicity parameter of $F$, 
and $\nu^2$ is the upper bound for the second moment of the error 
norms $\|w_k\|$ (cf.~Assumption~\ref{assum:w_k_bound}).

Now let us consider the stepsize sequence $\{\delta^*_k\}$ given by 
\begin{align}
& \delta_0^*=\frac{\eta-\beta L}{2(1+\beta)^2\nu^2}\,e_0\label{eqn:gmin0}
\\
& \delta_k^*=\delta_{k-1}^*\left(1-\left(\frac{\eta-\beta L}{2}\right)\delta_{k-1}^*\right) 
\quad \hbox{for all }k\ge 1,\label{eqn:gmink}
\end{align}
where $e_0$ is the same initial error as for the errors $e_k$ in~\eqref{def:e_k_different}.
In what follows, we often abbreviate $e_k(\delta_0,\ldots,\delta_{k-1})$ by $e_k$
whenever this is unambiguous. The next proposition shows that 
the lower bound sequence $\{\delta_k^*\}$ for $\gamma_{k,i}$ given by
(\ref{eqn:gmin0})--(\ref{eqn:gmink}) minimizes the errors $e_k$ over 
$[0,\frac{\eta -\beta L}{(1+\beta)^2L^2}]^{k}$.

\begin{proposition}[An adaptive lower bound steplength SA scheme]\label{prop:rec_results}
Let $e_k(\delta_0,\ldots,\delta_{k-1})$ be defined as in~\eqref{def:e_k_different},
where $e_0$ is a given {positive} scalar, $\nu$ is an upper bound defined in Assumption
\ref{assum:w_k_bound}, $\eta$ and $L$ are the strong monotonicity and
	Lipschitz constants of {the} mapping $F$ respectively { and $\nu$
		is chosen such that $ \nu \geq L\sqrt{\frac{e_0}{2}}$.} 
Let $\beta$ be a scalar such that $0\le \beta <\frac{\eta}{L}$,
and let the sequence $\{\delta^*_k\}$ be given by 
\eqref{eqn:gmin0}--\eqref{eqn:gmink}.
Then, the following hold:
\begin{itemize}
\item [(a)] {For all $k\ge0$, the error $e_k$  satisfies $e_k(\delta_0^*,\ldots,\delta_k^*) = \frac{2(1+\beta)^2\nu^2}{\eta-\beta L}\,\delta_k^*.$}
\item [(b)] For any $k\ge 1$, the vector $(\delta_0^*, \delta_1^*,\ldots,\delta_{k-1}^*)$ is 
the minimizer of the function $e_k(\delta_0,\ldots,\delta_{k-1})$ over the set
$$\mathbb{G}_k 
\triangleq \left\{\alpha \in \Real^k : {0< }\alpha_j \leq \frac{\eta-\beta L}{(1+\beta)^2L^2}, j =1,\ldots, k\right\}.$$
More precisely, for any $k\ge 1$ and any $(\delta_0,\ldots,\delta_{k-1})\in \mathbb{G}_k,$ we have
	\begin{align*}
	e_{k}(\delta_0,\ldots,\delta_{k-1}) 
         - e_{k}(\delta_0^*,\ldots,\delta_{k-1}^*)\geq (1+\beta)^2\nu^2(\delta_{k-1}-\delta_{k-1}^*)^2.
	\end{align*}	
\end{itemize}
\end{proposition} 
\begin{proof}
(a) \ To show the result, we use induction on $k$. 
Trivially, it holds for 
$k=0$ from \eqref{eqn:gmin0}. Now, suppose that we have
$e_k(\delta_0^*,\ldots,\delta_{k-1}^*) = \frac{2(1+\beta)^2\nu^2}{\eta-\beta L}\,\delta_k^*$ for some $k$,
and consider the case for $k+1$. From the definition of the error $e_k$ 
in~\eqref{def:e_k_different}, we have
\begin{align*}
 e_{k+1}(\delta_0^*,\ldots,\delta_k^*)  
 &=(1-(\eta-\beta L)\delta_k^*)e_{k}(\delta_0^*,\ldots,\delta_{k-1}^*)+(1+\beta)^2\nu^2\an{(\delta_k^*)^2}\cr 
&=(1-(\eta-\beta L)\delta_k^*)
\frac{2(1+\beta)^2\nu^2}{\eta-\beta L}\,\delta_k^*+(1+\beta)^2\nu^2\an{(\delta_k^*)^2},\end{align*}
where the second equality follows by the inductive hypothesis.
Thus,
\begin{align*}
e_{k+1}(\delta_0^*,\ldots,\delta_k^*) 
=\frac{2(1+\beta)^2\nu^2}{\eta-\beta L}\,\delta_k^*\left(1-\frac{\eta-\beta L}{2}\,\delta_k^*\right) =\frac{2(1+\beta)^2\nu^2}{\eta-\beta L}\,\delta_{k+1}^*,
\end{align*}
where the last equality follows by the definition of $\delta_{k+1}^*$ in
\eqref{eqn:gmink}. Hence, the result holds for all $k \ge 0$.

\noindent (b) \ {First we need to show that
	$(\delta^*_{0},\ldots,\delta_{k-1}^*)\in\mathbb{G}_k$. By \an{our assumption
		on $e_0$, we have  {$0<e_0 \le \frac{2\nu^2}{L^2}$}, 
		which by the definition of $\delta_0^*$ in~\eqref{eqn:gmin0} implies \sfy{that} {$0< \delta_0^*\leq \frac{\eta- \beta L}{(1+\beta)^2 L^2}$}, i.e., $\delta_0^*\in \mathbb{G}_1$.}
Using {the induction on $k$}, from relations (\ref{eqn:gmin0})--(\ref{eqn:gmink}), it can be shown that 
{$0 < \delta_{k}^*<\delta^*_{k-1}$} for all $k\ge1$. 
Thus, $(\delta^*_{0},\ldots,\delta_{k-1}^*)\in\mathbb{G}_k$ for all $k\ge1$.}
Using the induction on $k$ again,
we now show that {the} vector $(\delta_0^*,\delta_1^*,\ldots,\delta_{k-1}^*)$  
minimizes the error {$e_k(\delta_0,\ldots,\delta_{k-1})$} for all $k\ge1$. 
From the definition of the error $e_1$ and the relation
$e_1(\delta_0^*)=\frac{2(1+\beta)^2\nu^2}{\eta-\beta L}\,\delta_1^*$ shown in part (a), we have
\begin{align*}
e_1(\delta_0) -e_1(\delta_0^*)
 &=(1-(\eta-\beta L) \delta_0)e_0+(1+\beta)^2\nu^2 \delta_0^2-\frac{2(1+\beta)^2\nu^2}{\eta-\beta L}\,\delta_1^*.
\end{align*}
Using $\delta_1^*=\delta_0^*\left(1-\frac{\eta-\beta L}{2}\,\delta_0^*\right)$ {(cf.~\eqref{eqn:gmink})}, 
we obtain
\begin{align*}
e_1(\delta_0) -e_1(\delta_0^*)
 &=(1-(\eta-\beta L) \an{\delta_0})e_0+(1+\beta)^2\nu^2 \delta_0^2-\frac{2(1+\beta)^2\nu^2}{\eta-\beta L}\,\delta_0^*+(1+\beta)^2\nu^2(\delta_0^*)^2.
\end{align*}
Since $e_0=\frac{2(1+\beta)^2\nu^2}{\eta-\beta L}\,\delta_0^*$ (cf.~\eqref{eqn:gmin0}), it follows that
\begin{align*}
e_1(\delta_0) -e_1(\delta_0^*)
 &=\an{ -2(1+\beta)^2\nu^2\delta_0\delta_0^* }
 +(1+\beta)^2\nu^2 \delta_0^2+(1+\beta)^2\nu^2(\delta_0^*)^2
 =(1+\beta)^2\nu^2\left(\delta_0-\delta_0^*\right)^2,
\end{align*}
showing that the inductive hypothesis holds for $k = 1$. 
Now, suppose that 
{\begin{align}\label{eqn:ind}
	e_{k}(\delta_0,\ldots,\delta_{k-1}) 
         - e_{k}(\delta_0^*,\ldots,\delta_{k-1}^*)\geq (1+\beta)^2\nu^2(\delta_{k-1}-\delta_{k-1}^*)^2.
	\end{align}
holds for some $k$ and for all $(\delta_0,\ldots,\delta_{k-1}) \in\mathbb{G}_k$. 
We next show that relation~\eqref{eqn:ind} holds
for $k+1$ and for all $(\delta_0,\ldots,\delta_k) \in\mathbb{G}_{k+1}$.} 
To simplify the notation, 
we use $e_{k+1}^*$ to denote the error $e_{k+1}$ evaluated at 
$(\delta_{0}^*,\delta_1^*,\ldots,\delta_k^*)$, and $e_{k+1}$ when 
evaluating at an arbitrary vector 
$(\delta_0,\delta_1,\ldots,\delta_k)\in\mathbb{G}_{k+1}$.
Using (\ref{def:e_k_different})  and part (a), we have
\begin{align*}
e_{k+1}- e_{k+1}^*
& = (1-(\eta-\beta L) \delta_k)e_k
+(1+\beta)^2\nu^2\delta_k^2-\frac{2(1+\beta)^2\nu^2}{\eta-\beta L} \delta_{k+1}^*.
\end{align*}
Under the inductive hypothesis, we have $e_k\ge e_k^*$ ({cf.}~\eqref{eqn:ind}). 
When $(\delta_0,\delta_1,\ldots,\delta_k)\in \mathbb{G}_k$, we have
	$\delta_k \leq \frac{(\eta-\beta L)}{(1+\beta)^2L^2}.$ Next, we show
		that $\frac{(\eta-\beta L)}{(1+\beta)^2L^2} \leq
		\frac{1}{\eta-\beta L}$. By the definition of strong
		monotonicity and Lipschitzian property, we have $\eta \leq L$. Using $\eta \leq L$
		and $0\leq \beta \leq \frac{\eta}{L}$ we obtain
\begin{align*}
&\eta \leq (1+\beta)L \Rightarrow \eta -\beta L \leq (1+\beta)L \cr 
 \Rightarrow & (\eta -\beta L)^2 \leq (1+\beta)^2L^2 \Rightarrow \frac{(\eta-\beta L)}{(1+\beta)^2L^2} \leq \frac{1}{\eta-\beta L}. \end{align*}
This implies that for $(\delta_0,\delta_1,\ldots,\delta_k)\in \mathbb{G}_k$, we have 
$\delta_k \leq\frac{1}{\eta-\beta L}$ or equivalently $ 1-(\eta-\beta L)\delta_k\geq0 $.
Using this, the relation $e_k^*=\frac{2(1+\beta)^2\nu^2}{\eta-\beta L}\an{\delta_k^*}$
of part (a), and the definition of $\delta_{k+1}^*$, we obtain
\begin{align*}
e_{k+1} - e_{k+1}^* 
  & \geq  (1-(\eta-\beta L) \delta_k)\frac{2(1+\beta)^2\nu^2}{\eta-\beta L}\delta_k^*+(1+\beta)^2\nu^2\delta_k^2 -\frac{2(1+\beta)^2\nu^2}{\eta-\beta L} \delta_k^*\left(1-\frac{\eta-\beta L}{2}\delta_k^*\right) 
 \cr 
 &=(1+\beta)^2 \nu^2(\delta_k-\delta_k^*)^2.
\end{align*}
Hence, we have 
$
e_{k} - e_{k}^*\ge
(1+\beta)^2\nu^2(\delta_{k-1}-\delta_{k-1}^*)^2
$ for all $k\ge1$ and all 
$(\delta_0,\ldots,\delta_{k-1})\in\mathbb{G}_k$. 
\end{proof}
\textbf{Remark:} From Proposition~{\ref{prop:rec_results}},
the minimizer {$(\delta_0^*,\ldots,\delta_{k-1}^*)$} of $e_k$
over $\mathbb{G}_k$ is unique up to a scaling by a factor $\rho\in(0,1)$.
Specifically, the solution {$(\delta_0^*,\ldots,\delta_{k-1}^*)$} is obtained
for an initial error $e_0\geq 0$ satisfying $\nu \geq L\,\sqrt{\frac{e_0}{2}}$,
where $e_0$ can be chosen to be arbitrarily large by scaling $\nu$
appropriately. Suppose that in the definition of the sequence
{$\{\delta_k^*\}$}, $\rho e_0$ is employed instead of $e_0$ for some
$\rho\in(0,1)$. Then it can be seen (by following the proof) that, for
the resulting sequence, Proposition~\ref{prop:rec_results} would still
hold.\hfill $\square$
 
We have just provided an analysis in terms of a lower bound sequence $\{\delta_{k}\}$. 
We may conduct a similar analysis for \an{an upper bound sequence $\{\Gamma_{k}\}$.}
In particular, from {Lemma~\ref{lemma:error_bound_1}a} we have 
\begin{align*}
\EXP{\|x_{k+1}-x^*\|^2} \leq (1-2(\eta+L)\delta_k+2L\Gamma_k+L^2\Gamma_k^2)
\EXP{\|x_k-x^*\|^2} +\an{\Gamma_k^2\nu^2}\qquad\an{\hbox{for all $k\ge0$}}.
\end{align*} 
When $\frac{\Gamma_k-\delta_k}{\delta_k}\le \beta$ with $0\le \beta<\frac{\eta}{L}$, we have
$\frac{\Gamma_k}{1+\beta}\le\delta_k$, and we obtain the following relation:
\begin{align*}
\EXP{\|x_{k+1}-x^*\|^2} \leq (1-\frac{2(\eta+L)}{1+\beta}\Gamma_k+2L\Gamma_k
+L^2\Gamma_k^2)\EXP{\|x_k-x^*\|^2}+\Gamma_k^2\nu^2.
\end{align*}	
When $\Gamma_k$ is further restricted so that $0<\Gamma_k\le\frac{\eta-\beta L}{(1+\beta)L^2}$, 
we have
\begin{align*}
\EXP{\|x_{k+1}-x^*\|^2} 
 \le (1-\frac{(\eta-\beta L)}{1+\beta} \Gamma_k)\EXP{\|x_k-x^*\|^2}
 + \Gamma_k^2\nu^2\qquad\hbox{for all }k \geq 0.
\end{align*}
Using the preceding relation and \an{following a similar analysis as in the proof of 
Proposition~\ref{prop:rec_results}, we can show that the optimal 
choice of the sequence $\{\Gamma^*_{k}\}$} is given by 
\begin{align}
& \Gamma_{0}^*=\frac{\eta-\beta L}{2(1+\beta)\nu^2}\,e_0,\label{eqn:gmax0}
\\
& \Gamma_k^*=\Gamma_{k-1}^*\left(1-\frac{\eta-\beta L}{2(1+\beta)}\Gamma_{k-1}^*\right)\qquad \hbox{for all } k \ge 1, 
\label{eqn:gmaxk}
\end{align}
where $e_0$ is such that {$0< e_0\le\frac{2\nu^2}{L^2}$.

In the following lemma, we derive a relation between two recursive sequences, 
which is employed within our main convergence result for {adaptive stepsizes $\{\gamma_{k,i}\}$. }
\begin{lemma}\label{lemma:two-rec}
Suppose that sequences 
\an{$\{\lambda_k\}$ and $\{\g_k\}$} are given with the following recursive equations 
\an{for all  $k\geq 0$,}
\begin{align*}
\lambda_{k+1} =\lambda_{k}(1-\lambda_{k}),  \qquad
 \g_{k+1}  =\g_{k}(1-{\bar c}\g_{k}),
\end{align*}
\an{where {$\bar c>0$ is a given constant and} $\lambda_0={\bar c}\g_0$}. Then for all $k \geq 0$,
$\lambda_{k}={\bar c}\g_k.$
\end{lemma}
\begin{proof}
We use the induction on $k$. For $k=0$, the relation holds since $\lambda_0={\bar c}\g_0$. 
Suppose that for some $k \geq 0$ the relation holds. Then, we have 
\begin{align*}
&\g_{k+1}=\g_{k}(1-{\bar c}\g_{k}) 
\Rightarrow  \,{\bar c}\g_{k+1}={\bar c}\g_{k}(1-{\bar c}\g_{k})
\Rightarrow  \,{\bar c}\g_{k+1}=\lambda_{k}(1-\lambda_{k})
\Rightarrow  \,{\bar c}\g_{k+1}=\lambda_{k+1}.
\end{align*}
Hence, the result holds for $k+1$ implying that it holds for all $k \ge 0$.
\end{proof}

Using Lemma~\ref{lemma:two-rec}, we now present a relation between the
lower and upper bound sequences given by $\{\delta_k^*\}$ and
	$\{\Gamma_k^*\}$, respectively. 
\begin{lemma}\label{lemma:min_maxm_relation}
Suppose that \an{the} sequences 
$\{\delta_k^*\}$ and $\{\Gamma_k^*\}$ are given by relations \eqref{eqn:gmin0}--\eqref{eqn:gmink} and \eqref{eqn:gmax0}--\eqref{eqn:gmaxk}, \an{respectively, 
where {$0< e_0 \le \frac{2 \nu^2}{L^2}$} and $0\le \beta<\frac{\eta}{L}$. Then, for all $k \geq 0$,}
{$\Gamma_k^*=(1+\beta)\delta_k^*.$}
\end{lemma}
\begin{proof}
Suppose that \an{$\{\lambda_k\}$} is defined by the following recursive equation
\begin{align*}
\lambda_{k+1}=\lambda_{k}(1-\lambda_{k}), \quad \hbox{for all }  k\geq 0,
\end{align*}
where $\lambda_0=\frac{(\eta-\beta L)^2}{4(1+\beta)^2\nu^2}\,e_0$.
To obtain the result,  we apply Lemma~\ref{lemma:two-rec} 
to  sequences $\{\lambda_k\}$ and $\{\delta_k^*\}$, and then to sequences 
$\{\lambda_k\}$ and $\{\Gamma_k^*\}$. Specifically,
Lemma \ref{lemma:two-rec} implies that 
$ \lambda	_k =  \frac{\eta-\beta L}{2}\delta_k^*$ for all $k \geq 0$.
Invoking Lemma \ref{lemma:two-rec} for sequences $\{\lambda_k\}$ and $\{\Gamma_k^*\}$, 
we have
$ \lambda	_k =\frac{\eta-\beta L}{2(1+\beta)}\Gamma_k^*.$
From the preceding two relations, we conclude that $\Gamma_k^*=(1+\beta)\delta_k^*$
for all $k \geq 0$.   
\end{proof}

The relations~\eqref{eqn:gmin0}--\eqref{eqn:gmink} and
\eqref{eqn:gmax0}--\eqref{eqn:gmaxk}, respectively, are essentially
adaptive rules for determining the best upper and lower bounds for
stepsize sequences $\{\g_{k,i}\}$, where  "best'' corresponds to the
minimizers of the associated error bounds.  Having provided this
intermediate result, our main result is stated next and shows the
almost-sure convergence of the distributed adaptive SA (DASA) scheme. 

\begin{theorem}[A class of distributed adaptive steplength SA rules]\label{prop:DASA}
Suppose that Assumptions~\ref{assum:different_main} and~\ref{assum:w_k_bound} hold, and  
assume that the set $X$ is bounded.
Suppose that, for all $i=1, \ldots,N$, 
the stepsizes $\{\g_{k,i}\}$ in algorithm~\eqref{eqn:algorithm_different}
are given by the following recursive equations:
\begin{align}
& {\gamma_{0,i}=r_ic\,\frac{D^2}{\left(1+\dfrac{\eta-2c}{L}\right)^2\nu^2}},\label{eqn:gi0}
\\
& \gamma_{k,i}=\gamma_{k-1,i}\left(1-\frac{c}{r_i}\gamma_{k-1,i}\right) 
\quad \hbox{for all }k\ge 1.\label{eqn:gik}
\end{align}
where $D \triangleq \max_{x \in X}\|x-x_0\|$, 
$c$ is a scalar satisfying $c\in(0,\frac{\eta}{2})$,
 $r_i$ is a parameter such that $r_i \in [1, 1+\frac{\eta-2c}{L}]$,
 $\eta$ is the strong monotonicity parameter of the
mapping $F$,
 $L$ is the Lipschitz constant of $F$, and $\nu$ is the upper bound {defined in} Assumption
\ref{assum:w_k_bound}{. We assume that the constant $\nu$ is chosen large enough such that $\nu \geq \frac{DL}{\sqrt{2}}$.} Then, the following hold:
\begin{itemize}
\item [(a)] For any $i,j=1,\ldots, N$ and $k \geq 0$, {$\dfrac{\g_{k,i}}{r_i}=\dfrac{\g_{k,j}}{r_j}.$}
\item [(b)] Assumption \ref{assum:step_error_sub}b holds with $\beta=\frac{\eta -2c}{L}$, $\delta_k=\delta_k^*$, $\Gamma_k=\Gamma_k^*$, and $e_0=D^2$, where $\delta_k^*$ and $\Gamma_k^*$ are given by \eqref{eqn:gmin0}--\eqref{eqn:gmink} and \eqref{eqn:gmax0}--\eqref{eqn:gmaxk}, respectively.
\item [(c)] The sequence $\{x_k\}$
generated by algorithm (\ref{eqn:algorithm_different}) converges almost surely
to the unique solution of VI$(X,F)$.
\item [(d)] The results of Proposition~\ref{prop:rec_results} hold for $\delta_k^*$ when $e_0=D^2$ 
and $\beta=\frac{\eta -2c}{L}$.  
\end{itemize}
\end{theorem}
\begin{proof}
\noindent (a) \
Consider the sequence \an{$\{\lambda_k\}$} given by
\begin{align*}
& \lambda_0=\frac{c^2}{(1+\frac{\eta-2c}{L})^2\nu^2}\,D^2,
\\
& \lambda_{k+1}=\lambda_k(1-\lambda_k) 
\quad \hbox{for all }k\ge 1.
\end{align*}
Since for any ${i=1,\ldots, N}$, we have $\lambda_0=({c}/{r_i})\,
	  \gamma_{0,i}$, using Lemma \ref{lemma:two-rec} we obtain
	  $\lambda_k=({c}/{r_i})\g_{k,i}$ for
	  all $i=1,\ldots, N$ and $k \geq 0$. Hence, the desired relation follows.

\noindent (b) \ First we show that $\delta_k^*$ and $\Gamma_k^*$
are well defined. Consider the relation of part (a). Let $k\ge 0$ be
arbitrarily fixed. If $\g_{k,i}>\g_{k,j}$ for some $i \neq j$, then we
have $r_{i}>r_{j}.$ Therefore, the minimum possible $\g_{k,i}$ {is
	obtained} with $r_i=1$ and the maximum possible $\g_{k,i}$ {is
		obtained} with $r_i=1+\frac{\eta-2c}{L}$. Now, consider
		\eqref{eqn:gi0}--\eqref{eqn:gik}. If, $r_i=1$, and 
		$D^2$ is replaced by $e_0$, and $c$ by $\frac{\eta-\beta L}{2}$, we get the
		same recursive sequence defined by
		\eqref{eqn:gmin0}--\eqref{eqn:gmink}. Therefore, since the
		minimum possible $\g_{k,i}$ {is achieved} when $r_i=1$, we
		conclude that $\delta_k^* \leq \min_{i=1,\ldots,N} \g_{k,i}$ for
		any $k\ge 0$. This shows that $\delta_k^*$ is well-defined in
		the context of Assumption \ref{assum:step_error_sub}b.
		Similarly, it can be shown that $\Gamma_k^*$ is also
		well-defined in the context of  Assumption
		\ref{assum:step_error_sub}b. Now, Lemma~\ref{lemma:min_maxm_relation} implies that
		$\Gamma_k^*=(1+\frac{\eta-2c}{L})\delta_k^*$ for any $k \geq 0$,
		which shows that the inequality in Assumption~\ref{assum:step_error_sub}b is satisfied 
		with $\beta=\frac{\eta-2c}{L}$, 
		\an{where $0\le \beta<\frac{\eta}{L}$ since {$0<c\leq \frac{\eta}{2}$}.}

\noindent (c) \ In view of Proposition~\ref{prop:rel_bound}, to show the almost-sure convergence, 
it suffices to show that Assumption~\ref{assum:step_error_sub} holds.  Part (b) implies that 
Assumption~\ref{assum:step_error_sub}b is satisfied by the given stepsize choices. 
As seen in Proposition 3 of \cite{Farzad1}, 
Assumption~\ref{assum:step_error_sub}a holds for any positive recursive sequence $\{\lambda_k\}$ 
of the form $\lambda_{k+1}=\lambda_k(1-a\lambda_k)$. Since
each sequence $\g_{k,i}$ is a recursive sequence of this form, Assumption
\ref{assum:step_error_sub}a follows from Proposition~3 in \cite{Farzad1}.

\noindent (d) \ It suffices to show that the hypotheses of Proposition
	\ref{prop:rec_results} hold when $e_0=D^2$ and $\beta=\frac{\eta
		-2c}{L}$. Relation $\nu \geq \frac{DL}{\sqrt{2}}$ follows from 
$\nu \geq L\sqrt{\frac{e_0}{2}}$. Also, as mentioned in part (c), since $0<c\leq \frac{\eta}{2}$, the relation $0\le \beta<\frac{\eta}{L}$ holds for any choice of $c$ within that range. Therefore, the conditions of Proposition~\ref{prop:rec_results} are satisfied.
\end{proof}

{\textbf{Remark:} Theorem \ref{prop:DASA} provides a class of adaptive
stepsize rules for the distributed SA algorithm
(\ref{eqn:algorithm_different}), i.e., for any choice of parameter $c$
such that $0<c\leq\frac{\eta}{2}$, relations
\eqref{eqn:gi0}--\eqref{eqn:gik} correspond to an adaptive stepsize rule
for agents $1, \hdots, N$. Note that if $c=\frac{\eta}{2}$, these
adaptive rules will represent the centralized adaptive scheme given by
\eqref{eqn:optimal_self_adaptive_2}--\eqref{eqn:optimal_self_adaptive}. \hfill $\square$

In a distributed setting, each agent can choose its corresponding
parameter $r_i$ from the specified range $[1, 1+\frac{\eta-2c}{L}]$.
This requires that all agents agree on a fixed parameter $c$ and have a
common estimate of parameters $\eta$ and $L$. Yet, this scheme does not
allow complete flexibility for the agents and requires some global
specification of parameters such as $\eta, L,$ and $c$. In the next
section, we address the setting where the
Lipschitz constant is unavailable in a global setting or when the
mapping $F$ may not be Lipschitzian are addressed.}
\section{Non-Lipschitzian mappings and local randomization}\label{sec:convergence-no-Lip} 
A key shortcoming of the proposed DASA scheme, {given by
	(\ref{eqn:gi0})-(\ref{eqn:gik})}, is the requirement of  
the Lipschitzian property of the mapping $F$ with a
known parameter $L$. However, in a range of problem settings, the
following may arise:\\
\noindent \textbullet \, {\em Unavailability of a Lipschitz constant:} In many
settings, either the mapping may be non-Lipschitzian or the
estimation of such a constant may be problematic. It may also be that
this constant may not be available across the entire population of
agents.\\
\noindent  \textbullet \, {\em  Nonsmoothness in payoffs:} Suppose  the Cartesian stochastic variational
	inequality problem represents the optimality conditions of a
	stochastic convex program with nonsmooth (random) objectives or the
	equilibrium conditions of a stochastic Nash game in which the payoff
	functions are expectation-valued with random nonsmooth integrands.
	In either setting, the integrands associated with each component's
	expectation are multi-valued. In such a setting,   a
randomization or smoothing technique applied to each agent's payoff which leads to
an approximate mapping that can be shown to be Lipschitz and
single-valued. The associated Lipschitz constant can be specified in
terms of problem parameters and smoothing specifications, allowing us to
develop a locally randomized SA algorithm for stochastic variational inequalities without
Lipschitzian mappings. 

In Section~\ref{sec:smoothing}, we present the rudiments of our
randomization approach and discuss its {generalizations} in
Section~\ref{sec:dist_smoothing}. Finally, in
Section~\ref{sec:conv_smoothing}, we present a distributed locally
randomized SA scheme and provide suitable convergence theory.

\subsection{A {randomized} smoothing technique}\label{sec:smoothing}
In this part, we revisit a smoothing technique that has its roots
	in work by Steklov~\cite{steklov1,steklov2} in 1907. Over the
	years, it has been used by Bertsekas~\cite{Bertsekas72},
	Norkin~\cite{norkin93optimization} and more recently Lakshmanan and
		De Farias~\cite{DeFarias08}. The following proposition
		in~\cite{Bertsekas72} presents this smoothing technique for a
		nondifferentiable convex function.
\begin{proposition}\label{prop:bertsekas-approx}
Let $f:\Real^n \rightarrow \Real$ be a convex function and consider the
	function $f^{\epsilon}(x)$
\begin{equation*}
f^{\epsilon}(x) \triangleq \EXP{f(x-\omega)},
\end{equation*}
where $\omega$ belongs to the probability space $(\Real^n, B_n,P)$,
	$B_n$ is the $\sigma-$algebra of Borel sets of $\Real^n$ and
		  $P$ is a probability measure on {$B_n$} which is absolutely
		  continuous with respect to Lebesgue measure restricted on
		  $B_n$. Then, if $\EXP{f(x-\omega)}< \infty$ for all $x \in
		  \Real^n$, the function $f^{\e}$ is everywhere differentiable. 
\end{proposition} 
This technique has been employed in a number of papers such
	as~\cite{Gupal79,DeFarias08,Farzad1} to transform $f$ into a smooth
	function. In~\cite{DeFarias08}, authors consider a Gaussian
	distribution for the smoothing distribution and show that when
	function $f$ has bounded subgradients, the smooth function $f^{\e}$
	has Lipschitz gradients with a prescribed Lipschitz constant. A
	challenge in that approach is that in some situations, function $f$
	may have a restricted domain and not be defined for some
	realizations of the Gaussian random variable.
	
	Motivated by this challenge,
	   in~\cite{Farzad1}, we consider the randomized smoothing technique
		   using uniform random variables defined on an $n$-dimensional
		   ball centered on origin with radius $\e>0$. This approach is
		   called ``locally randomized smoothing technique'' and is used
		   to establish a local smoothing SA algorithm for solving
		   stochastic convex optimization problems in~\cite{Farzad1}. We
		   intend to extend this smoothing technique to the regime
		   of solving {stochastic Cartesian} variational inequality problems and exploit
		   the Lipschitzian property of the approximated mapping. In the
		   following example, we demonstrate how the smoothing technique
		   works for a piecewise linear function.		   
	\begin{figure}[htb]
 \centering
 \subfloat[The original function $f(x)$]
{\label{fig:originalf}\includegraphics[scale=.33, angle=-90]{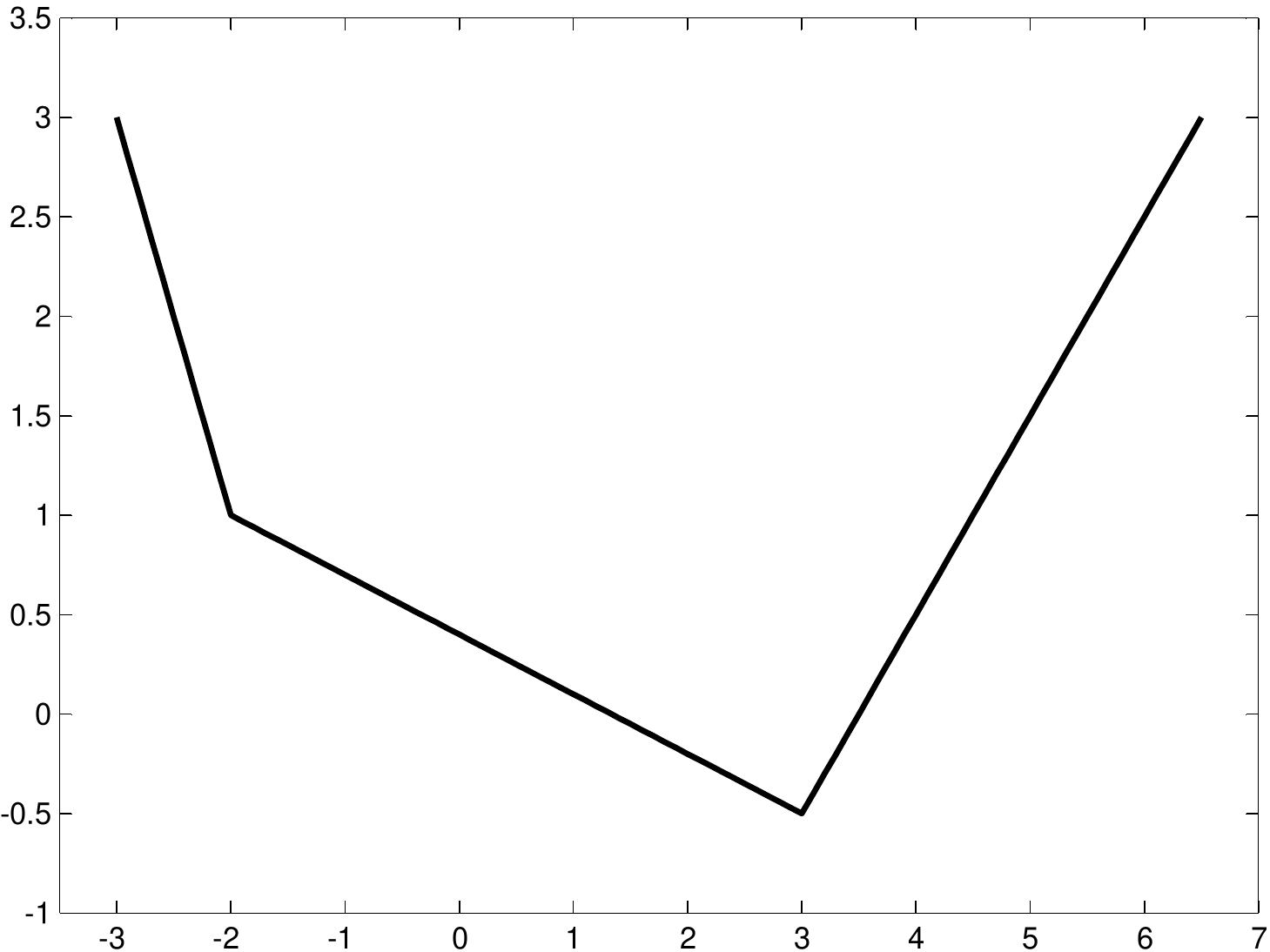}}
 \subfloat[The smoothed function $f^\e(x)$]{\label{fig:fhat}\includegraphics[scale=.33, angle=-90]{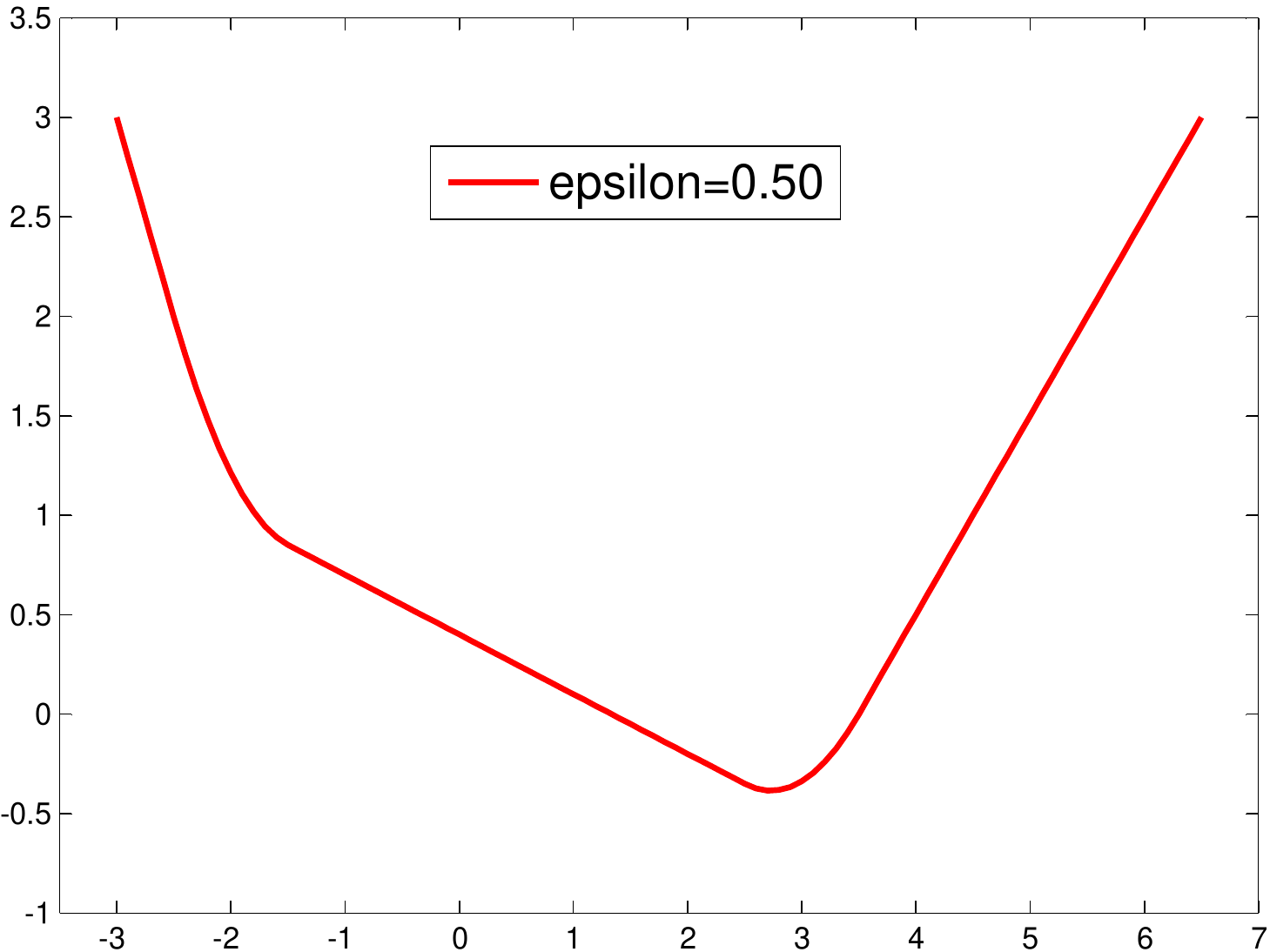}}
 \subfloat[Different smooth. parameters]{\label{fig:epsChange}\includegraphics[scale=.33, angle=-90]{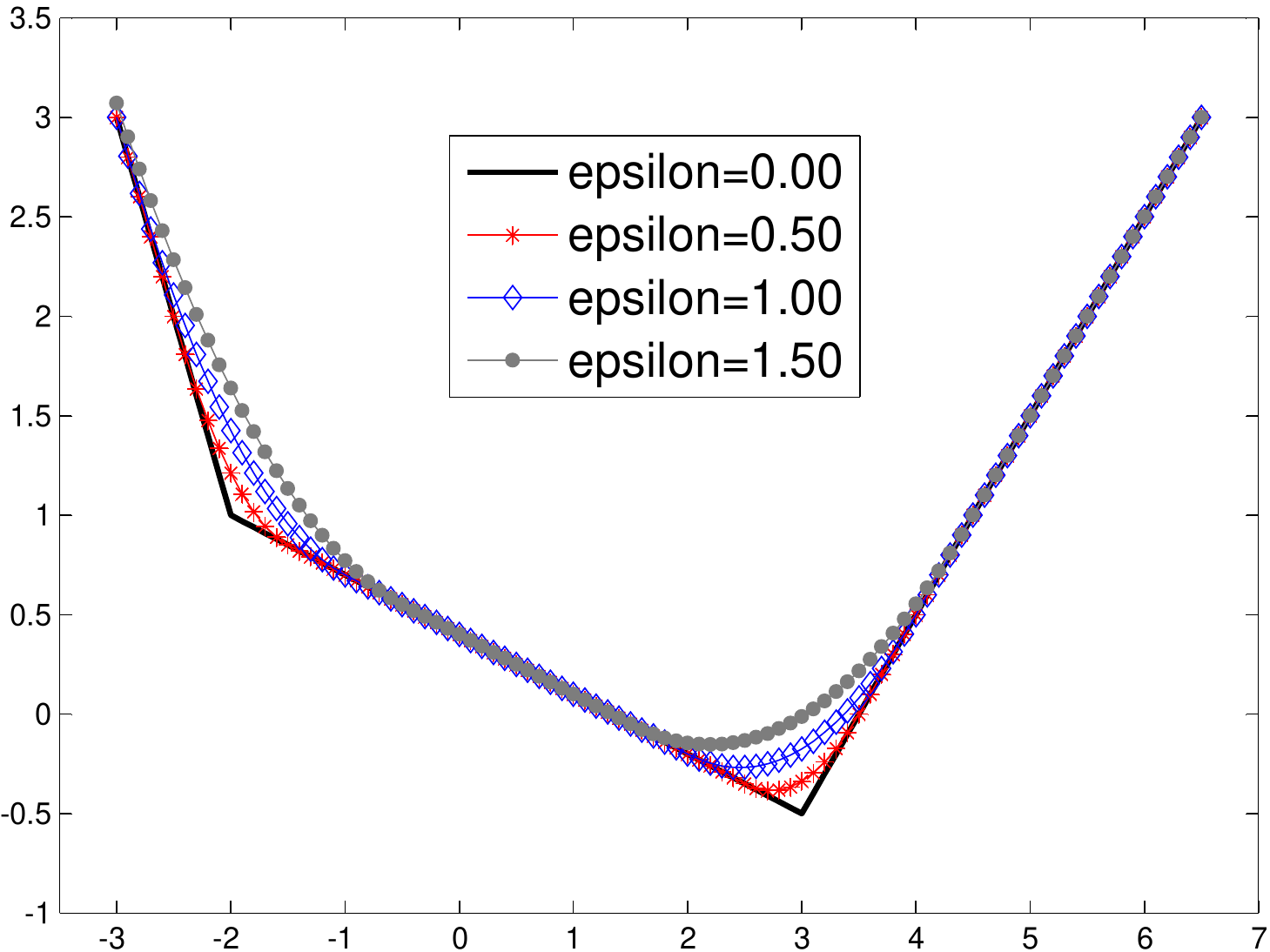}}
\caption{The smoothing technique}
\label{fig:smoothing}
\end{figure}

\begin{example}[Smoothing of a convex function]
{Consider the following piecewise linear function
\begin{align*}
f(x)=\left\{ \begin{array}{ll}
 -2x-3& \hbox{for }x < -2,\cr
-0.3x+0.4&\hbox{for }-2 \leq x <3\cr 
 x-3.5&\hbox{for } x\geq 3.
\end{array}\right.
\end{align*}
Suppose that $z$ is a uniform random variable defined on $[-\e,\e]$
where $\e>0$ is a given parameter. Consider the approximation function
$f^\e = \EXP{f(x+z)}$. Proposition~\ref{prop:bertsekas-approx} implies
that $f^\e$ is a smooth function. When $\e$ is a fixed constant
	satisfying $0< \e \leq 2.5$, the smoothed function $f^\e$ has the following form:
\begin{align*}
f^\e(x)=\left\{ \begin{array}{ll}
 -2x-3& \hbox{for }x < -2-\e,\cr
  \frac{1}{40\epsilon}\left(17x^2+68x-46x\epsilon+68-52\epsilon+17\epsilon^2\right)&\hbox{for } -2-\e \leq x <-2+\e ,\cr
  -0.3x+0.4&\hbox{for } -2+\e \leq x <3-\e ,\cr
 \frac{1}{40\epsilon}\left(13x^2-78x+14x\epsilon+117-62\epsilon+13\epsilon^2\right)&\hbox{for } 3-\e \leq x <3+\e ,\cr 
x-3.5& \hbox{for }x\geq 3+\e.
\end{array}\right.
\end{align*}
Figure \ref{fig:smoothing} shows such
a smoothing scheme. In Figure \ref{fig:originalf}, we observe that
function $f$ is nonsmooth at $x=-2$ and $x=3$. Figure \ref{fig:fhat}
shows the approximation $f^\e$ when $\e=0.5$. An immediate
	observation is that function $f^\e$ is smooth everywhere.
	Furthermore, the smoothing technique perturbs $x$
	locally at all points, including points of nonsmoothness. Finally, Figure~\ref{fig:epsChange}
shows the smoothing scheme for different values of $\e$ and illustrates
	the exactness of the approximation as $\e \to 0$. } 
\end{example}
\subsection{Locally randomized techniques}\label{sec:dist_smoothing}
Motivated by the smoothing technique described in previous part, we
introduce two distributed smoothing schemes where we simultaneously
perturb the value of vectors $x_i$ with a random vector $z_i$ for
$i=1,\ldots,N$. The first scheme is called a {\em multi-spherical randomized
(MSR) scheme}, where each random vector $z_i\in \Real^{n_i}$ is uniformly
distributed on the $n_i$-dimensional ball centered at the origin with
radius $\e_i$. In the second scheme, called {\em a multi-cubic randomized (MCR)
scheme}, we let $z_i\in \Real^{n_i}$ be uniformly
distributed on the $n_i$-dimensional cube centered at the origin with an
edge length of $2\e_i$. 

Now, consider a mapping $F$ 
that is not necessarily Lipschitz. We begin by
defining an approximation $F^\e:X \to \Real^n$ as the expectation of $F(x)$ when $x$ is perturbed
by a random vector \sfy{$z=(z_1;\ldots;z_N)$}. Specifically, $F^\e$ is given by
\begin{align}\label{eqn:DLRT-def} F^\e(x) \triangleq \pmat{\EXP{
F_1(x+z)} \\ \vdots \\ \EXP{F_N(x+z)}} \qquad \hbox{for all }x \in X,
	\end{align} 
	where $F_1,\ldots, F_N$ are coordinate-maps of $F$, 
	\sfy{$z=(z_1;\ldots;z_N)$} and the random vectors $z_i$ are given by MSR or MCR
		scheme.  

\subsubsection{{Multi-spherical randomized} smoothing}
Let us define $B_n(x,\rho) \subset \mathbb{R}^n$ as a ball centered at
	a point $x$ with a radius $\rho>0$. More precisely,
\[B_n(x,\rho)\triangleq \{y \in \mathbb{R}^n \mid  \|y-x\| \leq \rho \}.\]
In this scheme, assume that for all $i=1,\ldots,N$ random vector $z_i
\in B_{n_i}(0,\e_i)$ is uniformly distributed and {independent with
	respect to random vectors $z_j$ for $j\neq i$}. For the approximation mapping $F^\e$  to be well-defined,
	$F$ needs to be defined
	 over the set $X_S^{\e}$ given by 
	 $$X_s^{\e} \triangleq X+\prod_{i=1}^N B_{n_i}(0,\e_i).$$ 
	 This means that $X_s^{\e}=\{(x_1+z_1,\ldots,x_N+z_N)\an{\mid} x \in X, z_i \in
		\Real^{n_i}, \|z_i\| \leq \e_i\, \hbox{ for all }
	i=1,\ldots,N\},$ where the constants $\e_i >0$ are given values and 
	$\e\triangleq(\e_1,\ldots,\e_N).$ Note that the subscript $s$ stands for the MSR scheme.
		This scheme is developed based on the following assumption.
		
\begin{assumption}\label{assump:DRSS-bounded}
The mapping $F:X_s^{\e}\to\Real^n$ is bounded over the set \sfy{$X_s^{\e}$}. In particular, 
for every $i=1,\ldots,N$, there exists a constant $C_i>0$ such that  
$\| F_i(x)\| \leq C_i$ for all $x\in X_s^{\e}$.
\end{assumption}
Under this assumption, we will show that the smoothed mapping $F^\e$ produced by the 
MSR scheme is Lipschitz continuous over $X$ and we will compute its Lipschitz constant. 
To do so, we make use of the following  lemma. 

\begin{lemma}\label{lemma:spherical_Lipschitz_ineq}
Let $z \in \mathbb{R}^n$ be a random vector generated from a uniform density with zero mean over an 
$n$-dimensional ball centered at the origin with a radius $\e$. Then, the following relation holds:
\begin{equation*}
\int_{\mbR^n}|p{_u}(z -x)-p{_u}(z -y)|dz \le \kappa
\frac{n!!}{(n-1)!!} \,\frac{\|x-y\|}{\e}\qquad\hbox{for all }x,y\in\Real^n,
\end{equation*}
where $\kappa=1$ if $n$ is odd and $\kappa=\frac{2}{\pi}$ if $n$ is even, 
$n!!$ denotes double factorial of $n$, and 
$p_u$ is the probability density function of random vector $z$ given by
\begin{equation}\label{eqn:zuniform}
p{_u}(z) = \left\{\begin{array}{ll} \frac{1}{c_n \varepsilon^n}
&\hbox{for }z \in B_n(0,\e),\cr \hbox{} &\hbox{}\cr 0
&\hbox{otherwise,}\end{array}\right.
\end{equation}
where $c_{n} = \dfrac{\pi^\frac{n}{2}}{\Gamma(\frac{n}{2} + 1)}$,
and $\Gamma$ is the gamma function given by
\begin{eqnarray*}
\Gamma\left(\frac{n}{2}+1\right)= \left\{ \begin{array}{ll}
\left(\frac{n}{2}\right)! &\hbox{if $n$ is even,}\cr
\hbox{}&\hbox{}\cr \sqrt{\pi}\,\frac{n!!}{2^{(n+1)/2}} &\hbox{if $n$
is odd}.
\end{array}\right.
\end{eqnarray*}
\end{lemma}
\begin{proof}
The result is shown within the proof
	of Lemma 8 in the extended version of~\cite{Farzad1}.
\end{proof}
We next provide the main result of this subsection, which establishes the Lipschitz
	continuity and boundedness properties of the approximation mapping
		$F^\e$. It also provides the Lipschitz constant of $F^\e$ for the MSR
		scheme in terms of problem parameters. 

\begin{proposition}[Lipschitz continuity and boundedness of $F^\e$
	under the MSR scheme]\label{lemma:DRSS-Lipschitz}
 Let Assumption~\ref{assump:DRSS-bounded} hold and define vector $C\triangleq (C_1,\ldots,C_N)$. Then, for any $\e=(\e_1,\ldots,\e_N)>0$  we have the following:
 \begin{itemize}
\item [(a)] $F^\e$ is bounded over the set $X$, i.e,  $\| F^\e(x)\| \leq \|C\|$ for all $x \in X$.  
\item [(b)] $F^\e$ is Lipschitz continuous over the set $X$. More precisely, we have
\begin{align}\label{dist_Lip}\| F^\e(x) - F^\e(y)\| \leq  {\sqrt{N}}\|C\|
\max_{j\in \{1,\ldots,N\}}\left\{{\kappa_j}\frac{n_j!!}{(n_j-1)!!}\,\frac{1}{\e_j}\right\}\|x-y\|
	\qquad \hbox{for all } x, y \in X,\end{align} 
where $\kappa_j=1$ when $n_j$ is odd and $\kappa_j=\frac{2}{\pi}$ when $n_j$ is even.
\end{itemize}
\end{proposition}
\begin{proof}
(a) \ 
	We can bound the norm of $F^\e$ as follows:
\allowdisplaybreaks\begin{align*}
\|F^\e(x)\| 
			 =  \sqrt{\sum_{i=1}^N \|\EXP{F_i(x+z)}\|^2}
			 \leq \sqrt{\sum_{i=1}^N \EXP{\|F_i(x+z)\|^2}}		
			\leq \|C\|,
\end{align*}
where the first inequality follows from Jensen's inequality and the
second inequality is due to the boundedness property imposed on $F$ by
Assumption~\ref{assump:DRSS-bounded}.  
%

\noindent (b) \ From the definition of $F^\e$ in~relation (\ref{eqn:DLRT-def}) we have
\begin{align*}
\| F^\e(x) -F^\e(y)\|^2= \sum_{j=1}^N\|\EXP{F_j(x+z) -
	F_j(y+z)}\|^2=\sum_{j=1}^N\|\EXP{F_j(x+z)-
	F_j(y+z)}\|^2.
\end{align*}
We will add and subtract, sequentially, the values $F(u)$ at the vectors $u$ of the form 
$(y_1+z_1,\ldots,y_{i-1}+z_{i-1}, x_i+z_i,\ldots, x_N+z_N)$ for $i=2, \ldots, N$.
To keep the resulting expressions in a compact form,  we use the following notation. 
For an index set $J\subseteq \{1,\ldots, N\}$, we let $x_J \triangleq (x_{i})_{i\in J}$ and $x_{-J} \triangleq (x_{i})_{i\in \{1,\ldots, N\}-J}$.
By adding and subtracting the terms {$F_j((y+z)_{\{1,\ldots,i\}},(x+z)_{-\{1,\ldots,i\}})$} for all $i$, from the preceding relation we obtain
\begin{align*}
& \quad \|F^\e(x) - F^\e(y)\|^2 = \sum_{j=1}^N\left\|\underbrace{\EXP{F_j(x+z)- F_j((y+z)_{\{1\}},(x+z)_{-\{1\}})}}_{v_1}\right.\nonumber \cr 
+& \left.\underbrace{\EXP{F_j((y+z)_{\{1\}},(x+z)_{-\{1\}})-F_j((y+z)_{\{1,2\}},(x+z)_{-\{1,2\}})}}_{v_2}\right.\nonumber\cr
& \left.\vdots \right.\nonumber\cr
+& \left.\underbrace{\EXP{F_j((y+z)_{\{1,\ldots,i-1\}},(x+z)_{-\{1,\ldots,i-1\}})-F_j((y+z)_{\{1,\ldots,i\}},(x+z)_{-\{1,\ldots,i\}})}}_{v_i}\right.\nonumber\cr
& \left.\vdots \right.\nonumber\cr
+&\left. \underbrace{\EXP{F_j((y+z)_{\{1,\ldots,N-2\}},(x+z)_{-\{1,\ldots,N-2\}})-F_j((y+z)_{\{1,\ldots,N-1\}},(y+z)_{-\{1,\ldots,N-1\}})}}_{v_{N-1}}\right.\nonumber\cr
 +&\left.
 \underbrace{\EXP{F_j((y+z)_{\{1,\ldots,N-1\}},(x+z)_{-\{1,\ldots,N-1\}})-F_j(y+z)}}_{v_N}\right\|^2.
\end{align*}
Considering the definition of the vectors {$v_1,\ldots, v_N$} in the preceding relation, 
we have
\[ \|F^\e(x) - F^\e(y)\|^2 =\sum_{j=1}^N\left\|\sum_{i=1}^N v_i\right\|^2
\le N\sum_{j=1}^N \sum_{i=1}^N\|v_i\|^2,\]
where the inequality follows by the convexity of the squared-norm. By using the definitions of $v_i$ and 
exchanging the order of summations in the preceding relation, we obtain
\begin{align}\label{equ:lips_8}
& \quad \|F^\e(x) - F^\e(y)\|^2  \leq
	{N}\underbrace{\sum_{j=1}^N\left\|\EXP{F_j((x+z)_{\{1\}},
			(x+z)_{-\{1\}})-F_j((y+z)_{\{1\}},(x+z)_{-\{1\}})}\right\|^2}_{\rm Term \,
			1} \notag \\ 
&
+{N}\sum_{i=2}^{N}
\underbrace{\sum_{j=1}^N\left\|\EXP{F_j((y+z)_{\{1,\ldots,i-1\}},(x+z)_{-\{1,\ldots,i-1\}})
-F_j((y+z)_{\{1,\ldots,i\}},(x+z)_{-\{1,\ldots,i\}})}\right\|^2}_{{\rm Term} \,
			i} .
\end{align}

Next, we derive an estimate for Term 1. From our notation, it follows
	that for a vector $x$, $x_{\{1\}}=x_{1}$. In the interest of brevity, in the following, for a vector $x$, 
	we use $x_{-1}\triangleq x_{-\{1\}}$.
	Recalling the definition of $p_u$~in~\eqref{eqn:zuniform}, we write
\begin{align*}
\hbox{Term 1}&=\sum_{j=1}^N\left\|
	\int_{\Real^{n_1}}F_j(x_1+z_1,x_{-1}+z_{-1})p_u(z_1)dz_1-\int_{\Real^{n_1}}F_j(y_1+z_1,x_{-1}+z_{-1})p_u(z_1)dz_1\right\|^2\cr
&= \sum_{j=1}^N\left\|
	\int_{\Real^{n_1}}F_j(s_1,x_{-1}+z_{-1})p_u(s_1-x_1)ds_1-\int_{\Real^{n_1}}F_j(t_1,x_{-1}+z_{-1})p_u(t_1-y_1)dt_1\right\|^2\cr
&= \sum_{j=1}^N\left\|
	\int_{\Real^{n_1}}\EXP{F_j(t_1,x_{-1}+z_{-1})}(p_u(t_1-x_1)-p_u(t_1-y_1))dt_1\right\|^2,
\end{align*}	
where in the
second equality $s_1$ and $t_1$ are given by $s_1=x_1+z_1$ and
$t_1=y_1+z_1$. Using the triangle inequality and Jensen's inequality, we obtain
\begin{align*}
\hbox{Term 1}
\leq\sum_{j=1}^N\left(\int_{\Real^{n_1}}\EXP{\|F_j(t_1,x_{-1}+z_{-1})\|}\,
|p_u(t_1-x_1)-p_u(t_1-y_1)|dt_1\right)^2.
\end{align*}
By the definition of $F_j$ and Assumption~\ref{assump:DRSS-bounded}, the preceding relation yields
\begin{align*}
\hbox{Term 1}&\leq \sum_{j=1}^N \left(\int_{\Real^{n_1}}C_j\,|p_u(t_1-x_1)-p_u(t_1-y_1)|dt_1\right)^2\cr
&\leq \left(\sum_{j=1}^N C_j^2\right)\left(\kappa_1 \frac{n_1!!}{(n_1-1)!!}\frac{1}{\e_1}\|x_{1}-y_{1}\|\right)^2,
\end{align*}
where the last inequality is obtained using Lemma \ref{lemma:spherical_Lipschitz_ineq}. 
		Similarly, we may find estimates for the other terms in relation
		   (\ref{equ:lips_8}). Therefore, from relation
		   (\ref{equ:lips_8}) we may conclude that
\begin{align*}
\| F^\e(x) - F^\e(y)\|^2 
& \le N \left(\sum_{j=1}^N C_j^2\right)
\sum_{i=1}^{N}\left(\kappa_i \frac{n_i!!}{(n_i-1)!!}\frac{1}{\e_i}\|x_{i}-y_{i}\|\right)^2
\cr & \leq  N\left( \sum_{j=1}^N C_j^2 \right) 
\left(\max_{t=1,\ldots, N} \kappa_t \frac{n_t!!}{(n_t-1)!!}\frac{1}{\e_t}\right)^2
\sum_{i=1}^{N}\|x_{i}-y_{i}\|^2\cr 
&=  N\|C\|^2\left(\max_{t=1,\ldots, N}
	\kappa_t	\frac{n_t!!}{(n_t-1)!!}\frac{1}{\e_t}\right)^2\|x-y\|^2.
\end{align*}
Therefore, we have
\begin{align*}
\|F^\e(x) - F^\e(y)\| \leq \sqrt{N}\|C\|\max_{t\in \{1,\ldots, N\}}
\left\{\kappa_t\frac{n_t!!}{(n_t-1)!!}\frac{1}{\e_t}\right\}\|x-y\|.
\end{align*}
\end{proof}
{\textbf{Remark:} The MSR scheme is a generalization of the local
	randomization smoothing scheme presented in \cite{Farzad1}. Note
		that when $N=1$, the Lipschitz constant given in Proposition~\ref{lemma:DRSS-Lipschitz}b 
		is precisely the constant given by  Lemma~8 in~\cite{Farzad1}.}\hfill$\square$

\subsubsection{{Multi-cubic randomized smoothing scheme}}
We begin by defining $C_n(x,\rho) \subset \mathbb{R}^n$ as a cube
centered at a point $x$ with the edge length  $2\rho>0$ 
where the edges are along the coordinate axes. More precisely,
\[C_n(x,\rho)\triangleq \{y \in \mathbb{R}^n \mid \|y-x\|_\infty \leq \rho  \}.\]
In the MCR scheme, we assume that for any $i=1,\ldots,N$, the random
vector $z_i$ is uniformly distributed on the set $C_{n_i}(0,\e_i)$ { and
	is independent of the other random vectors $z_j$ for
		$j\neq i$}. 
		For the mapping $F$ we will assume that 
		it is well-defined over the set $X_c^\e$ given by 
		$$X_c^\e \triangleq X+\prod_{i=1}^N C_{n_i}(0,\e_i),$$ 
	where $\e_i >0$ are given values and $\e\triangleq(\e_1,\ldots,\e_N),$ while
	the subscript $c$ stands for the MCR scheme. 
	We investigate the properties of $F^\e$ for this smoothing scheme under
	the following basic assumption.

\begin{assumption}\label{assump:DRCS-bounded}
The mapping $F:X_c^{\e}\to\Real^n$ is bounded over the set $X_c^{\e}$. Specifically, 
for every $i=1,\ldots,N$, there exists a constant $C'_i>0$ such that  
$\| F_i(x)\| \leq C'_i$ for all $x\in X_c^{\e}$.
\end{assumption}

The following lemma provides a simple relation 
that will be important in establishing the main property of the density function used in the MCR scheme.

\begin{lemma}\label{lemma:product-sum-ineq}
Let the vector $p \in \mathbb{R}^m$ be such that 
$0\leq p_i \leq 1$ for all $i=1,\ldots,m$. Then, we have
\[1-\prod_{i=1}^m(1-p_i)\leq\|p\|_1.\]
\end{lemma}
\begin{proof}
We use induction on $m$ to prove this result. For $m=1$, we have
$1-\prod_{i=1}^m(1-p_i)=p_1=\|p\|_1$, implying that the result holds for
$m=1$. Let us assume that $1-\prod_{i=1}^m(1-p_i)\leq\|p\|_1$
holds for $m$. Therefore, we have 
$$\prod_{i=1}^m(1-p_i)\geq 1-\sum_{i=1}^mp_i.$$
Multiplying both sides of the preceding relation by $(1-p_{m+1})$, we
obtain $$\prod_{i=1}^{m+1}(1-p_i)\geq
(1-\sum_{i=1}^mp_i)(1-p_{m+1})=1-\sum_{i=1}^{m+1}p_i+p_{m+1}\sum_{i=1}^mp_i
\geq1-\sum_{i=1}^{m+1}p_i.$$ Hence, $\prod_{i=1}^{m+1}(1-p_i)\geq
1-\sum_{i=1}^{m+1}p_i$ which implies that the result holds for $m+1$.
Therefore, we conclude that the result holds for any integer $m \geq 1$.
\end{proof}

The following result is crucial for establishing the properties of the approximation $F^\e$ obtained by 
the MCR smoothing scheme.
\begin{lemma}\label{lemma:cubic_Lipschitz_ineq}
Let $z \in \mathbb{R}^n$ be a random vector with a zero-mean uniform density over an 
$n$-dimensional cube $\prod_{i=1}^N C_{n_i}(0,\e_i)$ for $\e_i>0$ for all $i$. 
Let the function $p_c: \mathbb{R}^n\rightarrow \mathbb{R}$ be the probability density function of 
the random vector $z$:
\begin{align*}
p_c(z)=\left\{\begin{array}{ll} \frac{1}{2^n \prod_{i=1}^N \e_i^{n_i}}
&\hbox{for }z \in \prod_{i=1}^N C_{n_i}(0,\e_i),\cr \hbox{} &\hbox{}\cr 0
&\hbox{otherwise.}\end{array}\right.
\end{align*}

Then, the following relation holds:
\begin{equation*}
\int_{\mathbb{R}^n}|p_c(u-x)-p_c(u-y)|du
\le\frac{\sqrt{n}}{\displaystyle \min_{1\le i\le N\}}\{\e_i\}}\|x-y\|\qquad\hbox{for all }x,y\in \Real^n.
\end{equation*}
\end{lemma}
\begin{proof}
Let $x,y\in\Real^n$ be arbitrary.
To simplify the notation, we
	define sets $S_x=\prod_{i=1}^N C_{n_i}(x_i,\e_i)$ and $S_y=\prod_{i=1}^N
	C_{n_i}(y_i,\e_i)$.
We consider, separately, the case when
the cubes $S_x$ and $S_y$ do not intersect, and 
the case when they do intersect. 
Before we proceed, we prove the following relation
\begin{equation}\label{eq:cubes}
S_x\cap S_y\ne \emptyset
\quad\hbox{if and  only if} \quad\|x_i-y_i\|_\infty\le 2\e_i \ \hbox{for all }i=1,\ldots,N.
\end{equation}
To prove relation~\eqref{eq:cubes}, suppose that the two cubes have nonempty intersection and let 
$u$ be in the intersection, i.e., $u\in S_x\cap S_y$.
Then, by the triangle inequality, we have for all $i=1,\ldots, N$,
\[ \|x_i-y_i\|_\infty\le \|x_i-u_i\|_\infty + \|u_i-y_i\|_\infty\le 2\e_i,\]
where the last inequality follows from the fact that $u$ belongs to each of the two cubes.
Thus, when $S_x\cap S_y\ne\emptyset$, we have $\|x_i-y_i\|_\infty\le 2\e_i$ for all $i$.
Conversely, suppose now that $\|x_i-y_i\|_\infty\le 2\e_i$ holds for all $i=1,\ldots,N.$ Let 
$\bar u=(x+y)/2$, and note that by the convexity of the norm $\|\cdot\|_\infty$, we have
\[\|\bar u_i-x_i\|_\infty= \left\|\frac{y-x}{2}\right\|_\infty\le \frac{1}{2}\|y_i-x_i\|_\infty\le \e_i
\qquad\hbox{for all }i.\] Thus, it follows that $\bar u\in S_x$.
Similarly, we find that $\|\bar u_i-y_i\|_\infty\le\e_i$ for all $i$, which implies that $\bar u\in S_y$.
Hence, $\bar u\in S_x\cap S_y$, thus showing that the two cubes have a nonempty intersection.

\begin{figure}[htb]
 \centering
 \subfloat[MCR scheme]
{\label{fig:cubes}\includegraphics[scale=.3]{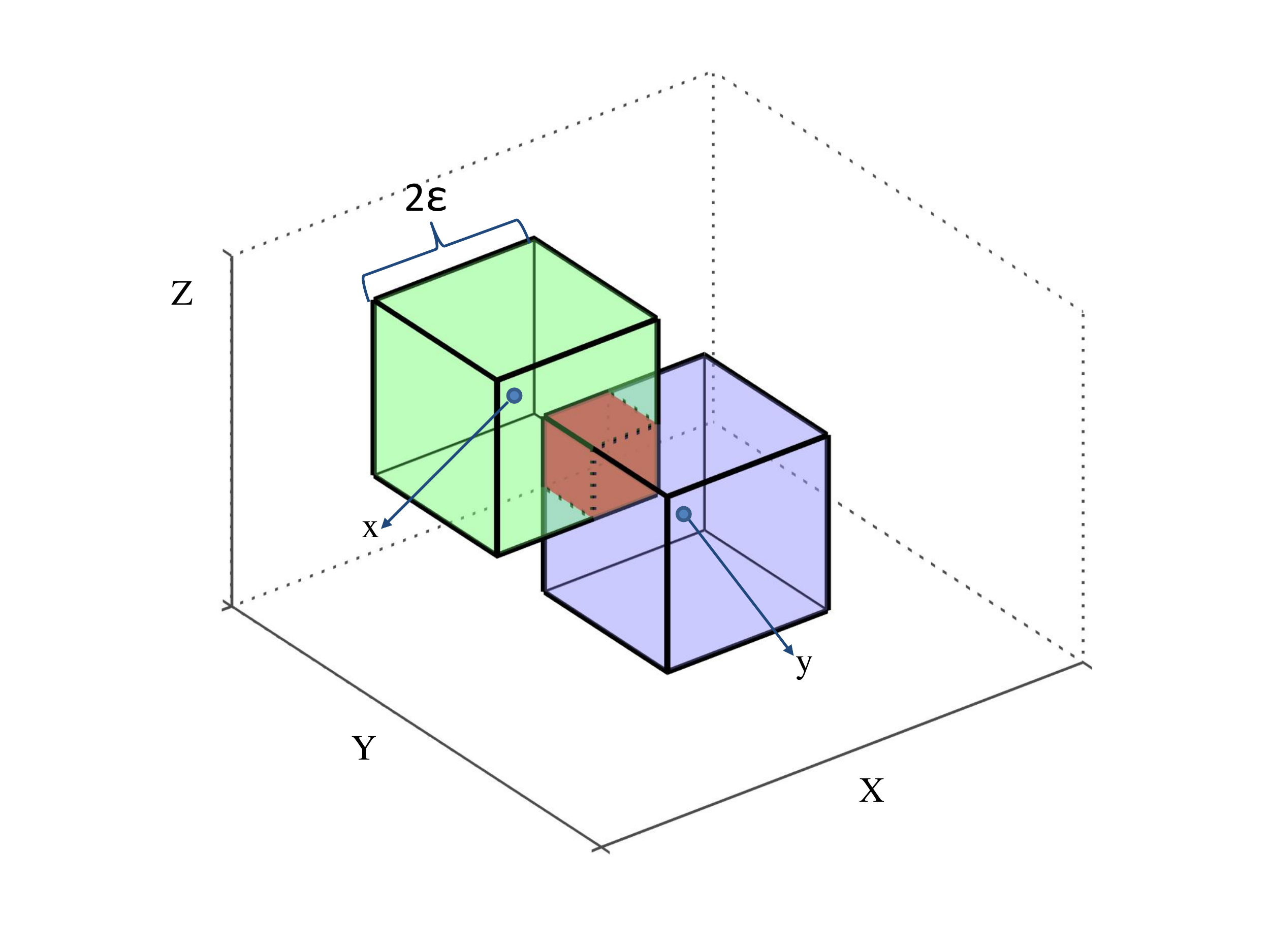}}
 \subfloat[MSR scheme]{\label{fig:spheres}\includegraphics[scale=.3]{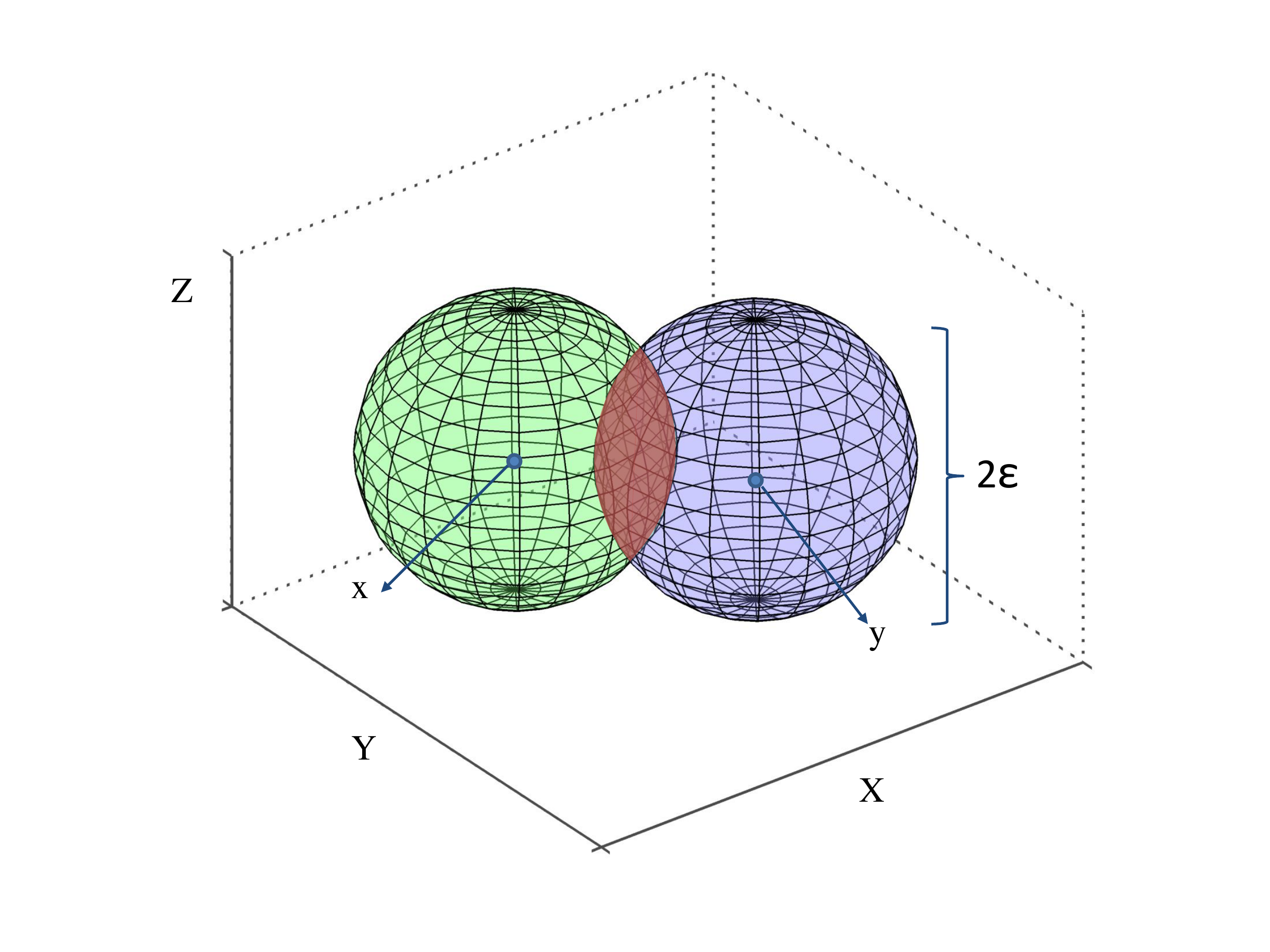}}
\caption{Calculating the Lipschitz constant in the locally randomized schemes.}
\label{fig:intersections}
\end{figure}

We now consider the integral $\int_{\mathbb{R}^n}|p_c(u-x)-p_c(u-y)|du$ for the 
cases when the cubes do not intersect and when they do intersect.\\
\noindent{\it Case 1: $S_x\cap S_y=\emptyset$.}  
In this case, we have
\begin{align*}
\int_{S_x}|p_c(u-x)-p_c(u-y)|du & = \int_{S_x}p_c(u-x)du, \\
\int_{S_y}|p_c(u-x)-p_c(u-y)|du & = \int_{S_y}p_c(u-y)du.\end{align*}
Consequently 
\begin{align}\label{eqn:DRCS-Lipschitz2}
\int_{\mathbb{R}^n}|p_c(u-x)-p_c(u-y)|du
=\int_{S_x}p_c(u-x)du +\int_{S_y}p_c(u-y)du = 2.
\end{align}
By relation~\eqref{eq:cubes}, there must exist some index $i^*\in\{1,\ldots,N\}$ such that 
$\|x_{i^*}-y_{i^*}\|_\infty>2\e_{i^*}$.
Since $\|x-y\|_\infty \ge
\|x_{i^*}-y_{i^*}\|_\infty$, it follows that
$\frac{\|x-y\|_\infty}{\min_{1\le i\le N}\{\e_{i}\}}>2$. Using the
relationship $\|u\|_\infty\le \|u\|$ between the infinity-norm and the Euclidean norm, 
we obtain $\frac{\|x-y\|}{\min_{1\le i\le N}\{\e_{i}\}}>2$. 
Therefore, using (\ref{eqn:DRCS-Lipschitz2}), we have 
 \begin{align}\label{eqn:DRCS-Lipschitz3}
\int_{\mathbb{R}^n}|p_c(u-x)-p_c(u-y)|du <
\frac{1}{\displaystyle \min_{1\le i\le N}\{\e_{i}\}}\|x-y\|.
\end{align}

\noindent{\it Case 2:  $S_x\cap S_y \ne\emptyset$.} 
 Then, we may decompose the  integral as follows: 
\begin{align*}
\int_{\mathbb{R}^n}|p_c(u-x)-p_c(u-y)|du 
&=\int_{S_x \cap S_y}|p_c(u-x)-p_c(u-y)|du + \int_{S_x^c \cap S_y^c}|p_c(u-x)-p_c(u-y)|du\cr 
&+\int_{S_x \setminus S_y}|p_c(u-x)-p_c(u-y)|du +  \int_{S_y \setminus S_x}|p_c(u-x)-p_c(u-y)|du.
\end{align*} 
Note that the first two integrals on the right hand side of the preceding equality are zero
since $p_c(u-x)=p_c(u-y)$ in the corresponding regions. 
Figure~\ref{fig:cubes} illustrates this observation\footnote{Figure~\ref{fig:spheres} provides a similar 
graphic for the MSR scheme.}. Therefore, we have
\begin{align*}
\int_{\mathbb{R}^n}|p_c(u-x)-p_c(u-y)|du 
&=\int_{S_x \setminus S_y}p_c(u-x)du+\int_{S_y \setminus S_x}p_c(u-y)du
=2\,\frac{1}{2^n \prod_{i=1}^N \e_i^{n_i}}\int_{S_x \setminus S_y}du.
\end{align*} 
Note that the value $2^n \prod_{i=1}^N \e_i^{n_i}$ is the volume of the cube $S_x$, denoted by
${\rm vol}(S_x)$. Similarly, the integral
$\int_{S_x \setminus S_y}du$ is equal to the volume of the set $S_x\setminus S_y$.
Thus, we can write
\begin{align*}
\int_{\mathbb{R}^n}|p_c(u-x)-p_c(u-y)|du 
=2\frac{ {\rm vol}(S_x \setminus S_y)} {{\rm vol}(S_x)}
=2\frac{ {\rm vol}(S_x)- {\rm vol}(S_x \cap S_y) }{{\rm vol}(S_x)}
=2\left(1-\frac{ {\rm vol}(S_x \cap S_y) }{{\rm vol}(S_x)}\right).
\end{align*} 
It can be seen that 
\[{\rm vol}(S_x \cap S_y)=\prod_{i=1}^N \prod_{j=1}^{n_i}(2\e_i-|x_i(j) - y_i(j)|),\]
where $w(j)$ denotes the $j$-th coordinate value of a vector $w$.
Therefore, from the preceding two relations and ${\rm vol}(S_x)= 2^n \prod_{i=1}^N \e_i^{n_i}$ we find that
\begin{align}\label{eqn:DRCS-Lipschitz4}
\int_{\mathbb{R}^n}|p_c(u-x)-p_c(u-y)|du
&=2\left(1-\frac{1}{2^n \prod_{i=1}^N \e_i^{n_i}}\left(\prod_{i=1}^N \prod_{j=1}^{n_i}(2\e_i-|x_i(j)-y(j)|)\right)\right)\cr
&=2\left(1-\prod_{i=1}^N \prod_{j=1}^{n_i}\left(1-\frac{|x_i(j)-y_i(j)|}{2\e_i}\right)\right).
\end{align}
Since the cubes $S_x$ and $S_y$ do intersect, 
by relation~\eqref{eq:cubes} there must hold
$\|x_{i}-y_{i}\|_\infty\le 2\e_{i}$ for all $i$.
Hence. $0 \le \frac{|x_i(j)-y_i(j)|}{2\e_i} \le 1$ for all $i$. 
Now, invoking Lemma \ref{lemma:product-sum-ineq}, from (\ref{eqn:DRCS-Lipschitz4}) we obtain
\[\int_{\mathbb{R}^n}|p_c(u-x)-p_c(u-y)|du 
\leq 2\sum_{i=1}^N \sum_{j=1}^{n_i}\frac{|x_i(j)-y_i(j)|}{2\e_i}
=\sum_{i=1}^N \frac{\|x_i-y_i\|_1}{\e_i}\le \sum_{i=1}^N \frac{\sqrt{n_i}}{\e_i}\,\|x_i-y_i\|,\]
where in the last inequality we used the relation between $\|\cdot\|_1$ and the Euclidean norm. 
Using H\"older's inequality, we have 
\[\sum_{i=1}^N \frac{\sqrt{n_i}}{\e_i}\,\|x_i-y_i\|
\le \sqrt{\sum_{i=1}^N \frac{n_i}{\e^2_i}} \ \|x-y\|
\le\frac{\sqrt{n}}{\displaystyle \min_{1\le i \le N}\{\e_i\}}\|x-y\|,\]
implying that 
\begin{align}\label{eqn:DRCS-Lipschitz5} 
\int_{\mathbb{R}^n}|p_c(u-x)-p_c(u-y)|du
\le\frac{\sqrt{n}}{\displaystyle \min_{1\le i\le N\}}\{\e_i\}}\|x-y\|.
\end{align}
By combining (\ref{eqn:DRCS-Lipschitz1}), (\ref{eqn:DRCS-Lipschitz3}),
and (\ref{eqn:DRCS-Lipschitz5}), and using the fact $n\ge1$, we obtain the desired result.
\end{proof}

Analogous to Proposition~\ref{lemma:DRSS-Lipschitz}, the next proposition
	derives the Lipschitz constant and boundedness properties of the
		approximation $F^\e$ under the MCR scheme.  
\begin{proposition}[Lipschitz continuity and boundedness of $F^\e$ under
the MCR scheme]\label{lemma:DRCS-Lipschitz}
 Let Assumption~\ref{assump:DRCS-bounded} hold and define vector $C^\prime\triangleq (C_1^\prime,\ldots,C_N^\prime)$. Then, for any $\e=(\e_1,\ldots,\e_N)>0$ 
 we have the following:
 \begin{itemize}
\item [(a)] $F^\e$ is bounded over the set $X$, i.e.,  $\| F^\e(x)\| \leq \|C^\prime\|$ for all $x \in X$.  
\item [(b)] $F^\e$ is Lipschitz over the set $X$. More precisely, we have
\begin{align}\label{dist_Lip2}\| F^\e(x) - F^\e(y)\| \leq 
\frac{\sqrt{n}\|C^\prime\|}{\min_{j=1,\ldots,N} \{\e_j\}}\|x-y\|
	\qquad \hbox{for all } x, y \in X.\end{align}
\end{itemize}
\end{proposition}
\begin{proof}
(a) \ This result can be shown in a similar fashion to the proof of Proposition~\ref{lemma:DRSS-Lipschitz}a.

\noindent (b) \ 
Since the random vector $z_i$ is uniformly distributed on the set 
$C_{n_i}(0,\e_i)$ for each $i=1,\ldots,N$,  the random vector $z=(z_1;\ldots;z_N)$ 
is uniformly distributed on the set $\prod_{i=1}^N C_{n_i}(0,\e_i)$. 
By the definition of the approximation $F^\e$ in~(\ref{eqn:DLRT-def}),
it follows that for any $x, y \in X$,
	\begin{align*}
\| F^\e(x) -F^\e(y)\| 
	&=\left\|\int_{\mathbb{R}^n}{F(x+z)}p_c(z)dz-\int_{\mathbb{R}^n}{F(y+z)}p_c(z)dz\right\|\cr
	&=\left\|\int_{\mathbb{R}^n}{F(u)}p_c(u-x)du-\int_{\mathbb{R}^n}{F(v)}p_c(v-y)dv\right\|\cr
	&=\left\|\int_{\mathbb{R}^n}{F(u)}(p_c(u-x)-p_c(u-y))du\right\|\cr
	& \le \int_{\mathbb{R}^n}\|F(u)\||p_c(u-x)-p_c(u-y)|du,
\end{align*} 
where in the second equality we let $u=x+z$ and $v=y+z$, while the 
inequality follows from the triangle inequality. 
Invoking Assumption~\ref{assump:DRCS-bounded} we obtain 
\begin{align}\label{eqn:DRCS-Lipschitz1}
\| F^\e(x) -F^\e(y)\|\le \|C^\prime\|\int_{\mathbb{R}^n}|p_c(u-x)-p_c(u-y)|du.
\end{align}
The desired relation follows from relation~\eqref{eqn:DRCS-Lipschitz1} and 
Lemma~\ref{lemma:cubic_Lipschitz_ineq}.
\end{proof}

\subsection{A distributed {locally randomized} SA scheme}\label{sec:conv_smoothing}
The locally randomized schemes presented in
Section~\ref{sec:dist_smoothing} facilitate the construction of  a distributed locally randomized SA scheme. 
Consider the Cartesian stochastic variational inequality problem VI$(X,F^\e)$ 
given in~\eqref{eqn:DLRT-def} where the mapping $F$ is not necessarily
Lipschitz. In this section, we assume that the conditions of the MSR scheme are
	satisfied, i.e., for all $i=1,\ldots,N$, the random vector $z_i$ is
		uniformly distributed over the set $\in B_{n_i}(0,\e_i)$
		independently from the other random vectors $z_j$ for $j\neq i$, and
		the mapping $F$ in (\ref{def-F}) is defined over the set $X_s^{\e}$.
Let the sequence $\{x_k\}$ be given by 
\begin{align}
\begin{aligned}
x_{{k+1},i}  =\Pi_{X_i}\left(x_{k,i}-\g_{k,i} \Phi_i(x_k+z_k,\xi_k)\right),
\end{aligned}\label{eqn:DLRSA}
\end{align}
for all $k\ge 0$ and $i=1,\ldots,N$, where
$\g_{k,i} >0$ denotes the stepsize of the $i$-th agent at iteration
$k$, $x_k =(x_{k,1}; x_{k,2}; \ldots ;x_{k,N})$, and $z_k =(z_{k,1};
		z_{k,2}; \ldots ; z_{k,N})$. The following proposition proves
the almost-sure convergence of the iterates {generated} by 
algorithm~\eqref{eqn:DLRSA} to the solution of the approximation
VI$(X,F^\e)$. In this result, we proceed to show that the {approximation}
does indeed satisfy the assumptions of
Proposition~\ref{prop:rel_bound} and convergence can then be immediately
claimed. We define $\sF'_k$, the history of the method up to time $k$, as
$$\sF'_k\triangleq \{x_0,z_0,\xi_0,z_1,\xi_1,\ldots,z_{k-1},\xi_{k-1}\},$$
for $k\ge 1$ and $\sF'_0=\{x_0\}$.
We assume that, at any iteration $k$, the vectors $z_k$ and $\xi_k$ in~\eqref{eqn:DLRSA}
are independent given the history $\sF'_k$.

\begin{proposition}[Almost-sure convergence of locally randomized DASA
scheme]
\label{prop:DLRSA}
Let Assumptions
	\ref{assum:different_main}a, \ref{assum:step_error_sub}, and
		\ref{assump:DRSS-bounded} hold, 
		and suppose that mapping $F$ is strongly monotone on the set $X_s^\e$ with a constant $\eta>0$.
		Also, assume that, for each $i=1,\ldots,N$, there exists a constant $\nu_i>0$ such that 
\begin{equation}\label{eq:nui}
\EXP{\|\Phi_i(x_k+z_k,\xi_k)-F_i(x_k+z_k)\|^2\mid \sF'_k}\le\nu_i^2
		\qquad a.s.\hbox{ for all $k$.}\end{equation}
Then, the sequence $\{x_k\}$
generated by algorithm~\eqref{eqn:DLRSA} converges almost surely
to the unique solution of VI$(X,F^\e)$.
\end{proposition}
\begin{proof}
Define random vector $\xi^\prime  \triangleq (z_1; z_2; \ldots ;
		z_N ; \xi )$, allowing us to rewrite algorithm (\ref{eqn:DLRSA})
	as follows:
\begin{align}
\begin{aligned}
x_{{k+1},i}  &=\Pi_{X_i}\left(x_{k,i}-\g_{k,i} (F_i^\e(x_k)+w_{k,i}^\prime )\right), \cr
 w_{k,i}^\prime  &\triangleq \Phi_i(x_k+z_k,\xi_k)-F_i^\e(x_k).
\end{aligned}\label{eqn:DLRSA2}
\end{align}
To prove convergence of the iterates
	produced by \eqref{eqn:DLRSA2}, it suffices to show that the
	conditions of Proposition~\ref{prop:rel_bound} are satisfied for the
	set $X$, the mapping $F^\e$, and the stochastic errors $w_{k,i}^\prime
	$. 

\noindent (i) Since Assumption \ref{assump:DRSS-bounded} holds, 
Proposition~\ref{lemma:DRSS-Lipschitz}b implies that the mapping $F^\e$ is
	Lipschitz over the set $X$ with the constant {$\sqrt{N}\|C\|
	\max_{1\le j\le N }\{\kappa_j\frac{n_j!!}{(n_j-1)!!}\,\frac{1}{\e_j}\}$}.
	Thus, Assumption~\ref{assum:different_main}b holds for the mapping~$F^\e$. 
	
	\noindent (ii)  Next, we show that the mapping $F^\e$ is
	strongly monotone over $X$. Since the mapping $F$ is strongly
	monotone over the set $X_s^\e$ with a constant $\eta >0$, 
	for any $u,v \in X_s^\e$, we have
\[(u-v)^T(F(u)-F(v))\geq \eta\|u-v\|^2.\]
Therefore, for any $x,y \in X$ and any realization of the random vector $z$,
	the vectors $x+z$ and $y+z$ belong to the set $X_s^\e$. Consequently, by
		defining $u\triangleq x+z$ and $v\triangleq y+z$, respectively, and nothing that $u-v=x-y$,
	from the previous relation we obtain
\[(x-y)^T(F(x+z)-F(y+z))\geq \eta\|x-y\|^2.\]
Taking expectations on both sides, it follows that 
\[(x-y)^T\left(\EXP{F(x+z)}-\EXP{F(y+z)}\right)\geq \eta\|x-y\|^2,\]
which implies that $F^\e$ is strongly monotone over the set $X$ with
the constant $\eta$. \\

\noindent (iii) The last step of the proof entails showing that the stochastic errors
$w_k^\prime \triangleq (w_{k,1} ; w_{k,2} ; \ldots ; w_{k,N})$ are well-defined, i.e., 
$\EXP{w_{k}^\prime \mid \sF^\prime_k}=0$ and that Assumption
\ref{assum:w_k_bound} holds with respect to the  stochastic error
$w_k^\prime$. Consider
the definition of {$w_{k,i}^\prime$} in (\ref{eqn:DLRSA2}). Taking conditional
expectations on both sides, we have for all $i=1,\ldots, N$
{\begin{align*}
\EXP{w_{k,i}^\prime \mid \sF^\prime_k}  = {\EXPz{\Phi_i(x_k+z_k,\xi_k)}}-F_i^\e(x_k)
  = \EXP{F_i(x_k+z_k)}-F_i^\e(x_k)
 = F_i^\e(x_k)-F_i^\e(x_k) =0,
\end{align*}}
where the last equality is obtained using the definition of $F^\e$ in~(\ref{eqn:DLRT-def}).
	Consequently, 
	it suffices to
	show that the condition of Assumption \ref{assum:w_k_bound}
holds. This may be expressed as follows: 
\begin{align*}
\EXP{\|w_{k}^\prime\|^2 \mid
	\sF^\prime_k}&=\EXP{\sum_{i=1}^N\|w_{k,i}^\prime\|^2\mid \sF^\prime_k} 
	= \EXPz{\sum_{i=1}^N\|\Phi_i(x_k+z_k,\xi_k)-F_i^\e(x_k)\|^2\mid
	\sF^\prime_k}.
	\end{align*}
	By adding and subtracting $F_i(x_k+z_k)$ we obtain
\begin{align*}
\EXP{\|w_{k}^\prime\|^2 \mid
	\sF^\prime_k}
 \leq & 2\EXPz{\sum_{i=1}^N\left( \|\Phi_i(x_k+z_k,\xi_k) - F_i(x_k+z_k) \|^2
		 +\|F_i(x_k+z_k) - F_i^\e(x_k)\|^2 \right)\mid
	\sF^\prime_k}\cr
 = & 2\sum_{i=1}^N \EXP{ \EXP{\|\Phi_i(x_k+z_k,\xi_k) - F_i(x_k+z_k) \|^2\mid
	\sF^\prime_k, z_k} \mid \sF'_k }\cr
&	+ 2 \sum_{i=1}^N\EXP{(\|F_i(x_k+z_k)\|^2-\|F_i^\e(x_k)\|^2)  \mid\sF'_k},
\end{align*}
where the last term is obtained from the following relation: 
\[\EXP{F_i(x_k+z_k)^T F^\e(x_k)\mid\sF_k'}= 
\EXP{F_i(x_k+z_k)^TF^\e(x_k)\mid x_k}=\|F^\e(x_k)\|^2.\]
Using the assumption on the errors given in~\eqref{eq:nui}, we further obtain
\begin{align}\label{eq:one}
\EXP{\|w_{k}^\prime\|^2 \mid
	\sF^\prime_k} \leq   2\sum_{i=1}^N\nu_i^2 + 
	2 \sum_{i=1}^N\EXP{(\|F_i(x_k+z_k)\|^2- \|F_i^\e(x_k)\|^2)  \mid\sF'_k}.\end{align}
Furthermore, we have
\begin{align}\label{eq:one2}
\sum_{i=1}^N\EXP{(\|F_i(x_k+z_k)\|^2- \|F_i^\e(x_k)\|^2)  \mid\sF'_k}
	\le \sum_{i=1}^N\EXP{\|F_i(x_k+z_k)\|^2 \mid\sF'_k}\le C^2,\end{align}
	where we use the fact $x_k+z_k\in X_s^\e$ and 
	the assumption that $F_i$ is uniformly bounded over the set $X_s^\e$ 
	(cf.~Assumption~\ref{assump:DRSS-bounded}). Relations~\eqref{eq:one}--\eqref{eq:one2}
imply that the stochastic errors $\{w_{k}^\prime\}$ satisfy Assumption
\ref{assum:w_k_bound}. Thus, the conditions of Proposition~\ref{prop:rel_bound}
are satisfied for the set $X$, the mapping $F^\e$, and the stochastic
	errors $w_{k,i}^\prime $ and the convergence result follows.
\end{proof}

The distributed locally randomized SA scheme produces a solution that is an approximation to
the true solution. A natural question is whether the sequence of
approximations tends to the solution of VI$(X,F)$ as $\e$, the size of
	the support of the randomization, tends to zero. The following proposition resolves this
question in the affirmative.

\begin{proposition}\label{prop:limit-epsilon}
Let Assumption~\ref{assum:different_main}a hold, and suppose that
	mapping $F$ is a continuous and strongly monotone over the set $X_s^\e$. 
	Let $x^\e$ and $x^*$ denote the solution of VI$(X,F^\e)$ and
	VI$(X,F)$, respectively. Then $x^\e \to x^*$ when $\e \to 0$.
\end{proposition}
\begin{proof}
As showed in the proof of Proposition~\ref{prop:DLRSA}, $F^\e$ is also
strongly monotone over the set $X$ with constant $\eta$. Since set $X$ is
assumed to be closed and convex, the definition of $X_s^\e$ implies that
\sfy{$X_s^\e$} is also closed and convex. Thus, the existence and uniqueness of
the solution to VI$(X,F)$, as well as VI$(X, F^\e)$, is guaranteed by
Theorem 2.3.3 of~\cite{facchinei02finite}. 
	
Let $\e=(\e_{1}, \e_{2}, \ldots,
		\e_{N})$ with $\e_i>0$ for all $i$ be arbitrary, and let
				   $x^{\e}$ denote the solution to
				   VI$(X,F^{\e})$. Let $x^*$ be the solution to VI$(X,F)$.
				   Thus, since $x^{\e}$ is the solution to
				   VI$(X,F^{\e})$, we have
$(x^*- x^\e)^TF^\e(x^\e) \geq 0.$ Similarly, 
since $x^*$ is the solution to VI$(X,F)$, we have $(x^\e-x^*)^TF(x^*) \geq 0$.
Adding the preceding two inequalities, we obtain for any $k \geq 0$,
\[(x^* -x^\e)^T(F^\e(x^\e) - F(x^*)) \geq 0.\]
Adding and subtracting the term $F^\e(x^*)$, we have 
\[(x^* - x^\e)^T(F^\e(x^\e) - F^\e(x^*)) + (x^*- x^\e)^T(F^\e(x^*) - F(x^*))\geq  0,\]
	implying that 
\[(x^* - x^\e)^T(F^\e(x^*) - F(x^*))\geq  (x^* - x^\e)^T(F^\e(x^*)-F^\e(x^\e))\ge\eta\|x^*-x^\e\|^2,\]
where the last inequality follows by the 
strong monotonicity of the mapping $F^\e$.
By invoking the Cauchy-Schwartz inequality, we obtain 
\begin{align}\label{ineq:x_k-x*}
 \|F^\e(x^*)-F(x^*)\| \geq \eta \|x^* - x^\e\|.
\end{align}
Next, we show that $\lim_{\e \to 0}F^\e(x^*) = F(x^*).$  
	By the definition of $F^\e$ and Jensen's inequality, we have
\begin{align}\label{ineq:prop9-1}
\|F^\e(x^*) - F(x^*)\| =\|\EXP{F(x^*+z) - F(x^*)}\|  \leq \EXP{\|F(x^*+z)-F(x^*)\|}.
\end{align}
Then, the expectation on the right-hand side can be expressed as follows:
\begin{align}\label{ineq:prop9-2} 
\EXP{\|F(x^*+z)-F(x^*)\|} 
	 &= \int_{\mathbb{R}^{n_1}}  \ldots \int_{\mathbb{R}^{n_N}}
	 \|F(x^*+z)-F(x^*)\|\left(\prod_{i=1}^N p_u(z_i)\right) dz_1 \cdots dz_N \cr
& = \int_{{ B_{n_1}(0,\e_1)}}  \ldots
\int_{{ B_{n_N}(0,\e_N)}}\|F(x^*+z) - F(x^*)\|\left(\prod_{i=1}^N p_u(z_i)
\right)dz_1 \cdots dz_N, \qquad
\end{align}
where the second equality is a consequence of the definition of the
random vector $z$. Let $\delta>0$ be an arbitrary fixed number.
By the continuity of $F$ over
		{$X_s^\e$}, there exists a $\delta^\prime >0$, such that if
		$\|(x^*+z)-x^*\|\le \delta^\prime$, then $\|F(x^*+z)-F(x^*)\|\le
		\delta$. Therefore, for all $\e=(\e_1,\e_2,\ldots,\e_N)$ with 
		$\|\e\|\le\delta'$ we have $\|z\|\le \|\e\|\le\delta'$ for $z\in\prod_{i=1}^NB_{n_i}(0,\e_i)$,
		 which is equivalent to
		$\|(x^*+z)-x^*\|\le \delta^\prime$. Hence,  $\|F(x^*+z)-F(x^*)\|\le
		\delta$ for all $z\in \prod_{i=1}^NB_{n_i}(0,\e_i)$ with $\e_i$ such that $\|\e\|\le\delta'$.
		Thus, using
				(\ref{ineq:prop9-1}) and (\ref{ineq:prop9-2}), for any $\e=(\e_1,\ldots,\e_N)$
				with $\|\e\|\le\delta'$, we have 
\[\|F^\e(x^*)-F(x^*)\| 
\leq \delta {\int_{{ B_{n_1}(0,\e_{k,1})}}  \ldots \int_{{ B_{n_N}(0,\e_{k,N})}}
\left(\prod_{i=1}^N p_u(z_i)  \right)d{z_{k,1}} \ldots d{z_{k,N}}} =\delta.\]
Since $\delta>0$ was arbitrary, we conclude that $\lim_{\e\to0}\|F_k(x^*)-F(x^*)\|=0$. 
Therefore, taking limits on both
	sides of inequality~(\ref{ineq:x_k-x*}), 
we obtain $\lim_{\e\to0}\|x^*-{x^{\e_k}}\|=0$.
\end{proof}

{\textbf{Remark:} Note that the results of Propostion~\ref{prop:DLRSA} and Proposition~\ref{prop:limit-epsilon}
hold when the random vector $z$ fits the conditions of the MCR scheme.} 
\section{Numerical results}\label{sec:numerics}
In this section, we report the results of our numerical experiments on
two sets of test problems. Of these, the first is a stochastic bandwidth-sharing problem in
communication networks (Sec.  \ref{sec:5.2}), while the second is a
stochastic Nash-Cournot game (Sec. \ref{sec:5.1}). In each instance, we
	compare the performance of the distributed adaptive stepsize SA
	scheme (DASA) given by (\ref{eqn:gi0})--(\ref{eqn:gik}) with that of
	SA schemes with harmonic stepsize sequences (HSA), where agents use
	the stepsize $\frac{\theta}{k}$ at iteration $k$. More precisely, we
	consider three different values of the parameter $\theta$, i.e.,
	$\theta = 0.1$, $1$, and $10$. This diversity of choices allows us to observe
	the sensitivity of the HSA scheme to different settings of the
	parameters. In the context of Nash-Cournot games, we use the distributed locally
	randomized SA scheme described in Sec. \ref{sec:conv_smoothing}
	with the MSR and MCR techniques.   In each
	instance, we conduct a sensitivity analysis where we consider $12$
	different parameter settings,  categorized
	into $4$ sets. In each set, one parameter is changed while other
	parameters are maintained as fixed. We  provide $90\%$ confidence
	intervals of the mean squared error for each of the $12$ settings.
	Our experiments have been done using Matlab
	$7.12$.

\subsection{A bandwidth-sharing problem in computer networks}\label{sec:5.2}
We consider a communication network where users compete for the
bandwidth. Such a problem can be captured by an optimization framework (cf. \cite{Cho05}). Motivated by this model, we consider a network
with $16$ nodes, $20$ links and $5$ users. Figure \ref{fig:network}
shows the configuration of this network. \begin{figure}[htb]
\begin{center}
 \includegraphics[scale=.40]{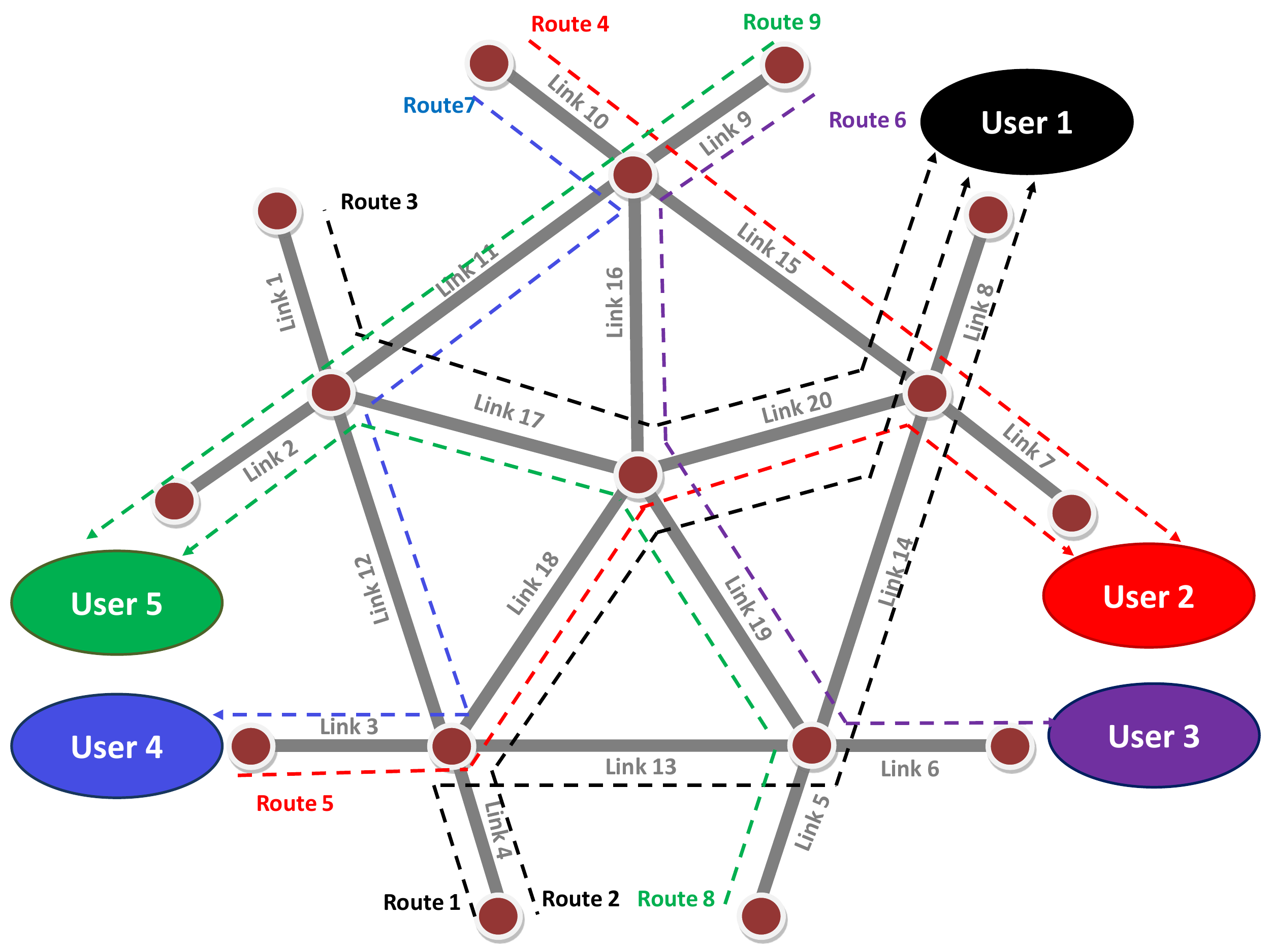}
 \caption{The bandwidth-sharing problem -- the network}
  \label{fig:network}
   \end{center}
 \vspace{-0.1in}  
 \end{figure} Users have access to  different routes as shown in Figure
 \ref{fig:network}. For example, user $1$ can  access routes $1$, $2$, and $3$. 
 Each user is characterized by a cost function. Additionally, there is a
congestion cost function that depends on the aggregate flow. More
specifically, the cost function
user $i$ with flow rate (bandwidth) $x_i$ is defined by
\[f_i(x_i,\xi_i)\triangleq -\sum_{r \in \mathcal{R}(i)} \xi_{i}(r)\log(1+x_i(r)),\]
for $i=1,\ldots, 5$, where $x\triangleq (x_1; \ldots; x_5)$ is the flow decision vector of the users, $\xi \triangleq (\xi_1; \ldots; \xi_5)$ is a random parameter corresponding to the different users, 
$\mathcal{R}(i)=\{1,2,\ldots,n_i\}$ is the set of routes assigned to the $i$-th user, $x_{i}(r)$ and $\xi_{i}(r)$ 
are the $r$-th element of the decision vector $x_i$ and the random vector $\xi_i$, respectively. 
We assume that
$\xi_i(r)$ is drawn from a uniform distribution for each $i$ and $r$. More precisely, $\xi_1(1)$, $\xi_1(2)$, and $\xi_1(3)$ are i.i.d.\ and uniformly distributed in $[1-0.1, 1+0.1]$, 
$\xi_2(1)$ and $\xi_2(2)$ are i.i.d.\ and uniformly distributed in $[1.4-0.2, 1.4+0.2]$, 
$\xi_3(1)$ and $\xi_4(1)$ are i.i.d.\ and uniformly distributed in $[0.8-0.05, 0.8+0.05]$ and $[1.6-0.2, 1.6+0.2]$, respectively, and 
$\xi_5(1)$ and $\xi_5(2)$ are i.i.d and uniformly distributed in $[1.2-0.1, 1.2+0.1]$. 

The links have limited capacities, which are given by
\[b=(10; 15; 15; 20; 10; 10; 20; 30; 25; 15; 20; 15; 10; 10; 15; 15; 20; 20; 25; 40).\]   
We may define the routing matrix $A$ that describes the relation between set of routes $\mathcal{R}=\{1,2,\ldots,9\}$ and set of links $\mathcal{L}=\{1,2,\ldots,20\}$. Assume that $A_{lr}=1$ if route $r \in \mathcal{R}$ goes through link $l \in \mathcal{L}$ and $A_{lr}=0$ otherwise. Using this matrix, the capacity constraints of the links can be described by $Ax \leq b$.

We formulate this model as a stochastic optimization problem given by
\begin{align}\label{eqn:network-prob}
\displaystyle \mbox{minimize} \qquad &  \sum_{i=1}^N \EXP{f_i(x_i,\xi_i)} +c(x) \\
\mbox{subject to} \qquad& Ax \leq b \notag\\
& x \geq 0,\notag
\end{align}
where $c(x)$ is the network congestion cost.
We consider this cost of the form $c(x)=\|Ax\|^2$. 
Problem (\ref{eqn:network-prob}) is a convex optimization problem and the optimality conditions can be stated as a variational inequality given by $\nabla f(x^*)^T(x-x^*) \geq 0$, where $f(x) \triangleq \sum_{i=1}^N \EXP{f_i(x_i,\xi_i)} +c(x)$. Using our notation in Sec. \ref{sec:2.2}, we have $$F(x)\triangleq \nabla f(x)=-\left( 
\frac{\bar \xi_1(1)}{1+x_1(1)};\ldots;
\frac{\bar \xi_i(r_i)}{1+x_i(r_i)};
\ldots;
\frac{\bar\xi_5(2)}{1+x_5(2)} \right)+2A^TAx,$$
where $\bar \xi_i(r_i) \triangleq \EXP{\xi_i(r_i)}$ for any $i=1,\ldots,5$ and $r_i=1,\ldots,n_i$.
We now show that the mapping $F$ is Lipschitz and strongly monotone. 
Using the preceding relation, triangle inequality, and Cauchy-Schwartz inequality, for any $x,y \in X\triangleq \{x \in \mathbb{R}^N|Ax \leq b, x\geq0 \}$, we have
\begin{align*}
& \|F(x)-F(y)\| \\ &=\left\|-\left(\bar \xi_1(1)\left(\frac{1}{1+x_1(1)}-\frac{1}{1+y_1(1)}\right); \ldots; \bar \xi_5(2)\left(\frac{1}{1+x_5(2)}-\frac{1}{1+y_5(2)}\right) \right) +2A^TA(x-y)\right\|\\
&\leq \left\|\left(\bar \xi_1(1)\frac{x_1(1)-y_1(1)}{(1+x_1(1))(1+y_1(1))}; \ldots;\bar \xi_5(2) \frac{x_5(2)-y_5(2)}{(1+x_5(2))(1+y_5(2))} \right)\right\| +2\|A^TA\|\|x-y\|.
\end{align*}
Using nonnegativity constraints, from the preceding relation we obtain
\begin{align*} \|F(x)-F(y)\| \leq \max_{i,r_i}\bar \xi_i(r_i) \|x-y\| +2\|A^TA\|\|x-y\| 
= \left(\max_{i,r_i}\bar \xi_i(r_i) +2\|A^TA\|\right)\|x-y\|, 
\end{align*}
implying that $F$ is Lipschitz with constant $\max_{i,r_i}\bar \xi_i(r_i) +2\|A^TA\|$. To show the monotonicity of $F$, we write
\begin{align*}
& (F(x)-F(y))^T(x-y)\\
&=\left(\left(\bar \xi_1(1)\frac{x_1(1)-y_1(1)}{(1+x_1(1))(1+y_1(1))}; \ldots;\bar \xi_5(2) \frac{x_5(2)-y_5(2)}{(1+x_5(2))(1+y_5(2))} \right)+2A^TA(x-y) \right)^T(x-y)\\
& =\sum_{i,r}\bar \xi_i(r) \frac{(x_i(r)-y_i(r))^2}{(1+x_i(r))(1+y_i(r))}+2(x-y)^T(A^TA)(x-y)\\
& \geq \frac{\min_{i,r_i}\bar \xi_i(r_i)}{(1+\max_{l}{b(l)})^2}\|x-y\|^2+2(x-y)^T(A^TA)(x-y) \\
& =(x-y)^T\left(\frac{\min_{i,r_i}\bar \xi_i(r_i)}{(1+\max_{l}{b(l)})^2}{\bf{I}}_{N}  +2A^TA\right)(x-y). 
\end{align*}
Our choice of matrix $A$ is such that $A^TA$ is positive definite. Thus, the
preceding relation implies that $F$ is strongly monotone with parameter
$$\eta=\frac{\min_{i,r}\bar \xi_i(r_i)}{(1+\max_{l}{b(l)})^2}
+2\lambda_{\min}(A^TA),$$ 
where $\lambda_{\min}(A^TA)$ is the minimum eigenvalue of the matrix $A^TA$.

\subsubsection{Specification of parameters}
In this experiment, the optimal solution $x^*$ of the problem (\ref{eqn:network-prob}) is calculated by sample
average approximation (SAA) method using the nonlinear programming
solver \texttt{knitro} \cite{knitro}. Our goal lies in comparing the
performance of the DASA scheme given by (\ref{eqn:gi0})--(\ref{eqn:gik})
	with that of SA schemes using harmonic stepsize sequences of the form $\gamma_k = \frac{\theta}{k}$,
	referred to as HSA schemes.
	We consider
	three values for $\theta$ and observe the performance of HSA scheme
	in each case. To calculate the stepsize sequence in DASA scheme,
	other than $\eta$ and $L$ obtained in the previous part, parameters
	$c$, $r_i$, $D$, and $\nu$ need to be evaluated. We assume that
	$c=\frac{\eta}{4}$ and $r_i$ is uniformly drawn from the interval
	$[1,1+\frac{\eta-2c}{L}]$ for each user. We let the starting point
	of all SA schemes be zero, i.e., $x_0=0$. Thus, $D=\max_{x \in
		X}{\|x\|}$. Since the routing matrix $A$ has  binary entries,
		from $Ax \leq b$, one may conclude that $\sqrt{N}\max_{l}b(l)$
		can be chosen as $D$. To calculate $\nu$, for any $k \geq 0$ we
		have
\begin{align*}
\EXP{\|w_k\|^2\mid \sF_k}& =\EXP{\|\sfy{\Phi(x_k},\xi_k)-F(x_k)\|^2\mid \sF_k}\cr
& =\EXP{\left\|\left( \frac{\xi_{k,1}(1)-\bar \xi_{k,1}(1)}{1+x_{k,1}(1)}; \ldots; \frac{\xi_{k,5}(2)-\bar \xi_{k,5}(2)}{1+x_{k,5}(2)} \right)\right\|^2\mid \sF_k} \cr
& =\EXP{\sum_{i=1}^N\sum_{r=1}^{n_i} \left( \frac{\xi_{k,i}(r)-\bar \xi_{k,i}(r)}{1+x_{k,i}(r)} \right)^2\mid \sF_k} \cr
& = \sum_{i=1}^N\sum_{r=1}^{n_i}\frac{\mbox{var}(\xi_{k,i}(r))}{(1+x_{k,i}(r))^2}\cr
& \leq \sum_{i=1}^N\sum_{r=1}^{n_i}\mbox{var}(\xi_{k,i}(r)),
\end{align*}
where the last inequality is obtained using $x_{k,i}(r) \geq 0$. Thus,
	  $\sqrt{\sum_{i=1}^N\sum_{r=1}^{n_i}\mbox{var}(\xi_{k,i}(r))}$ is a
	  candidate for parameter $\nu$. On the other hand, $\nu$ needs to
	  satisfy $\nu \geq \frac{LD}{\sqrt{2}}$ from Theorem
	  \ref{prop:DASA}. Therefore, we set $\nu$ as follows: 
	  $$\nu=\max\left\{\sqrt{\sum_{i=1}^N\sum_{r=1}^{n_i}\mbox{var}(\xi_{k,i}(r))},
	  \frac{LD}{\sqrt{2}}\right\}.$$
\subsubsection{Sensitivity analysis}\label{sec:5.1.2}
We solve the bandwidth-sharing problem for $12$
different settings of parameters shown in Table
\ref{tab:network_errors1}. We consider $4$ parameters in our model that
scale the problem. Here, $m_b$ denotes the multiplier of the capacity
vector $b$, $m_c$ denotes the multiplier of the congestion cost function
$c(x)$, and $m_\xi$ and $d_\xi$ are two multipliers that parametrize the
random variable $\xi$. More precisely, if $i$-th user in route $r$ is
uniformly distributed in $[a-b,a+b]$, here we assume that it is
uniformly distributed in $[m_\xi a-d_\xi b, m_\xi a+d_\xi b]$. $S(i)$
denotes the $i$-th setting of parameters. For each of these $4$
parameters, we consider $3$ settings where one parameter changes and
other parameters are fixed. This allows us to observe the sensitivity of
the algorithms with respect to each of these parameters. 
\begin{table}[htb] 
\vspace{-0.05in} 
\tiny 
\centering 
\begin{tabular}{|c|c|c|c|c|c|} 
\hline 
-&S$(i)$ &  $m_b$ & $m_c$ & $m_\xi$ & $d_\xi$
\\ 
\hline 
\hline
\sfy{$m_b$} &1 &  1 & 1 & 5 & 2
  \\

\hbox{ }& 2 &  0.1 & 1 & 5 & 2
  \\

\hbox{ }& 3 &  0.01 & 1 & 5 & 2  \\ 
\hline 

$m_c$ &4 &  0.1 & 2 & 2 & 1
  \\

\hbox{ }& 5 &  0.1 & 1 & 2 & 1
  \\

\hbox{ }& 6 &  0.1 & 0.5 & 2 & 1 \\ 
\hline 

$m_\xi$ &7 &  1 & 1 & 1 & 5
  \\

\hbox{ }& 8 &  1 & 1 & 2 & 5
  \\

\hbox{ }& 9 &  1 & 1 & 5 & 5 \\ 
\hline 
$d_\xi$ &10 &   1 & 0.01 & 1 & 1
  \\

\hbox{ }& 11 & 1 & 0.01 & 1 & 2
  \\

\hbox{ }& 12 &  1 & 0.01 & 1 & 5\\
\hline

\end{tabular} 
\caption{The bandwidth-sharing problem: Parameter settings} 
\label{tab:network_errors1} 
\vspace{-0.1in} 
\end{table}	
The SA algorithms are terminated after $4000$ iterates.
To measure the error of the schemes, we run each scheme $25$ times and
then compute the mean squared error (MSE) using the metric $\frac{1}{25}\sum_{i=1}^{25}\|x_k^i-x^*\|^2$ for any $k=1,\ldots,4000$, where $i$ denotes the $i$-th sample.
Table \ref{tab:Traffic2} shows the $90\%$ confidence intervals (CIs) of the error for the DASA and HSA schemes. 
\begin{table}[htb] 
\vspace{-0.05in} 
\tiny 
\centering 
\begin{tabular}{|c|c||c||c||c|c|} 
\hline 
-&S$(i)$ &   DASA - $90\%$ CI &  HSA with $\theta=0.1$- $90\%$ CI & HSA with $\theta=1$ - $90\%$ CI & HSA with $\theta=10$ - $90\%$ CI
\\ 
\hline 
\hline
  \sfy{$m_b$} &1 & [$2.97 $e${-6}$,$4.66 $e${-6}$] &  [$1.52 $e${-6}$,$2.37 $e${-6}$] &  [$1.70 $e${-6}$,$2.97 $e${-6}$]&  [$1.33 $e${-5}$,$1.81 $e${-5}$]
  \\

\hbox{ }&2 & [$2.97 $e${-6}$,$4.66 $e${-6}$] &  [$1.52 $e${-6}$,$2.37 $e${-6}$] &  [$1.70 $e${-6}$,$2.97 $e${-6}$]&  [$1.33 $e${-5}$,$1.81 $e${-5}$]
  \\

\hbox{ }&3& [$1.15 $e${-7}$,$3.04 $e${-7}$] &  [$2.12 $e${-8}$,$4.92 $e${-8}$] &  [$4.66 $e${-8}$,$1.17 $e${-7}$]&  [$8.07 $e${-7}$,$2.43 $e${-6}$]
  \\ 
  
  \hline 
  
  $m_c$ &4 &  [$4.39 $e${-7}$,$6.55 $e${-7}$] &  [$1.33$e${-6}$,$1.80 $e${-6}$] &  [$4.71 $e${-7}$,$8.75 $e${-7}$]&  [$3.84 $e${-6}$,$5.38 $e${-6}$]
  \\

\hbox{ }&5 &  [$1.29 $e${-6}$,$1.97 $e${-6}$] &  [$9.00$e${-6}$,$1.20 $e${-5}$] &  [$7.88 $e${-7}$,$1.36 $e${-6}$]&  [$5.61 $e${-6}$,$7.98 $e${-6}$]
  \\

\hbox{ }&6 &  [$3.44 $e${-6}$,$5.36 $e${-6}$] &  [$2.26$e${-4}$,$2.53 $e${-4}$] &  [$1.25 $e${-6}$,$1.99 $e${-6}$]&  [$7.34 $e${-6}$,$1.12 $e${-5}$]
  \\ 
  \hline

 $m_\xi$ &7& [$4.29 $e${-5}$,$6.40 $e${-5}$] &  [$7.92 $e${-5}$,$1.49 $e${-4}$] &  [$2.83 $e${-5}$,$4.75 $e${-5}$]&  [$1.84 $e${-4}$,$2.75 $e${-4}$] 
  \\

\hbox{ }&8 &[$3.18 $e${-5}$,$4.83 $e${-5}$] &  [$3.46 $e${-5}$,$6.07 $e${-5}$] &  [$1.97 $e${-5}$,$3.39 $e${-5}$]&  [$1.40 $e${-4}$,$1.99 $e${-4}$] 
  \\

\hbox{ }&9 & [$1.83 $e${-5}$,$2.88 $e${-5}$] &  [$6.12 $e${-6}$,$9.99 $e${-6}$] &  [$1.06 $e${-5}$,$1.85 $e${-5}$]&  [$8.33 $e${-5}$,$1.13 $e${-4}$]
  \\
  \hline
  
  $d_\xi$ &10 &[$3.82 $e${-4}$,$5.91 $e${-4}$] &  [$2.86 $e${+1}$,$2.86 $e${+1}$] &  [$5.50 $e${-1}$,$5.70 $e${-1}$]&  [$7.23 $e${-5}$,$9.64 $e${-5}$]
  \\ 
 
\hbox{ }&11 & [$9.81 $e${-4}$,$1.44 $e${-3}$] &  [$2.86 $e${+1}$,$2.86 $e${+1}$] &  [$5.45 $e${-1}$,$5.85 $e${-1}$]&  [$2.85 $e${-4}$,$3.80 $e${-4}$]
  \\  
\hbox{ }&12 &[$6.26 $e${-3}$,$8.44 $e${-3}$] &  [$2.85 $e${+1}$,$2.86 $e${+1}$] &  [$5.47 $e${-1}$,$6.44 $e${-1}$]&  [$1.77 $e${-3}$,$2.36 $e${-3}$]
  \\ 

  \hline

\end{tabular} 
\caption{The bandwidth-sharing problem -- $90\%$ CIs for DASA and HSA schemes} 
\label{tab:Traffic2} 
\vspace{-0.1in} 
\end{table}	

\subsubsection{Results and insights} 
We observe that DASA scheme performs favorably and is far more robust in
comparison with the HSA schemes with different choice of $\theta$.
Importantly, in most of the settings, DASA stands close to the HSA
scheme with the minimum MSE.  Note that when $\theta=1$ or $\theta=10$,
the stepsize $\frac{\theta}{k}$ is not within the interval
$(0,\frac{\eta-\beta L}{(1+\beta)^2L^2}]$ for small $k$ and is not
		feasible in the sense of Prop.  \ref{prop:rec_results}.
		Comparing the performance of each HSA scheme in different
		settings, we observe that HSA schemes are fairly sensitive to
		the choice of parameters.  For example, HSA with $\theta=0.1$
		performs very well in settings S$(1)$, S$(2)$, and S$(3)$, while
		its performance deteriorates in settings S$(10)$, S$(11)$, and
		S$(12)$. A similar discussion holds for other two HSA schemes. A
		good instance of this argument is shown in Figure
		\ref{fig:traffic_all}.
 \begin{figure}[htb]
 \centering
 \subfloat[Setting S$(1)$]
{\label{fig:traffic_prob1}\includegraphics[scale=.25]{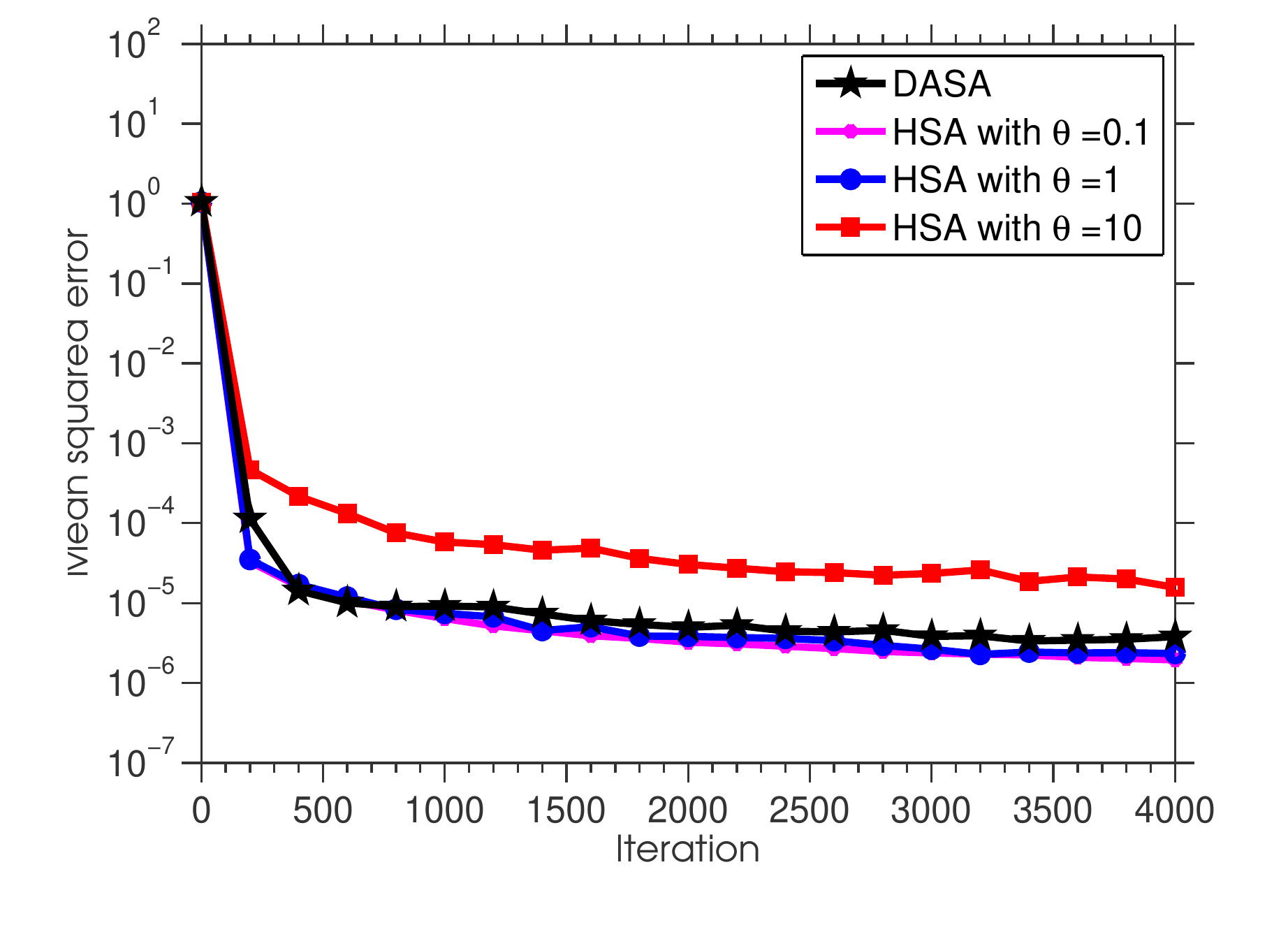}}
 \subfloat[Setting S$(4)$]
{\label{fig:traffic_prob4}\includegraphics[scale=.25]{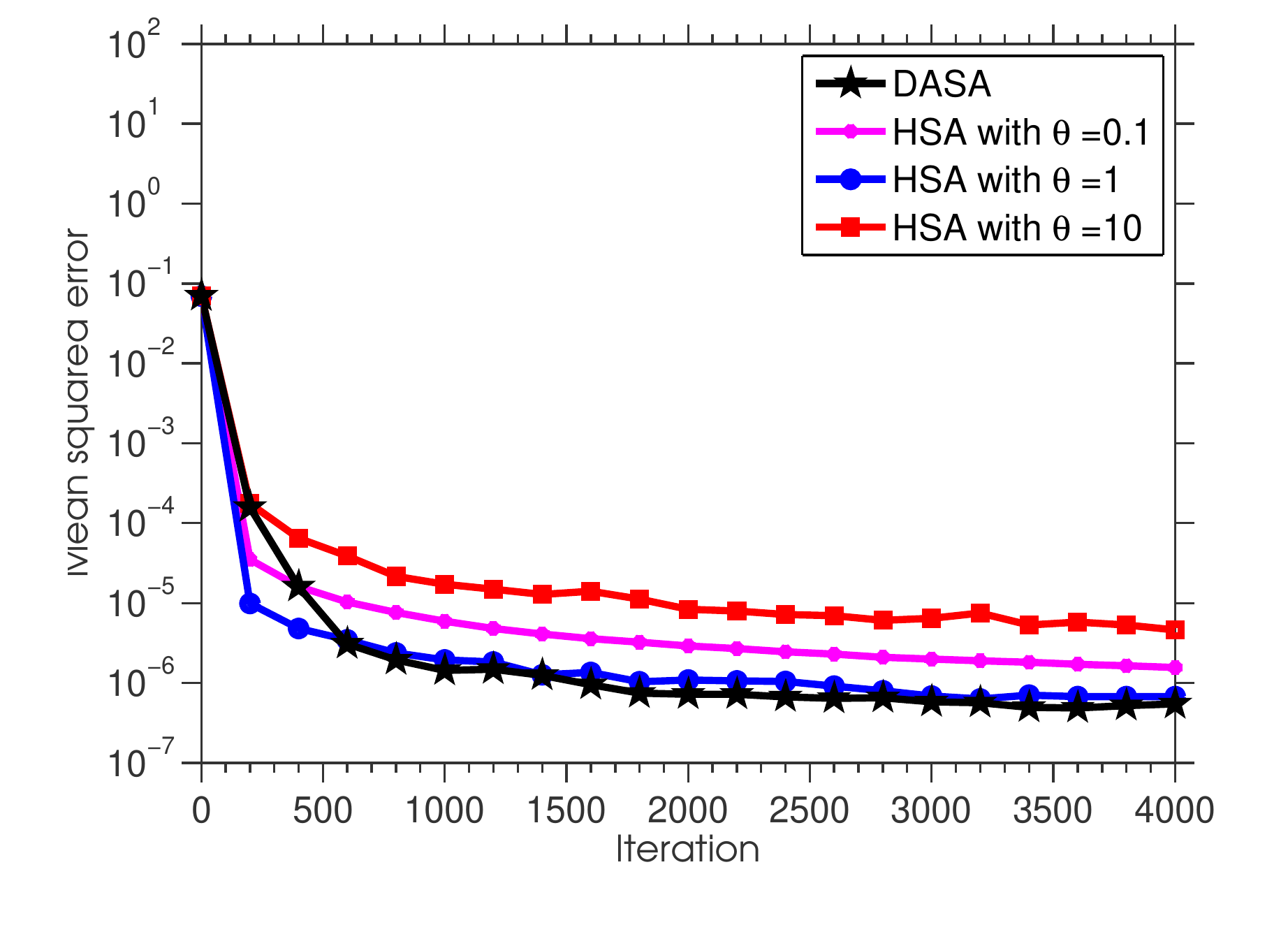}}
 \subfloat[Setting S$(11)$]{\label{fig:traffic_prob11}\includegraphics[scale=.25]{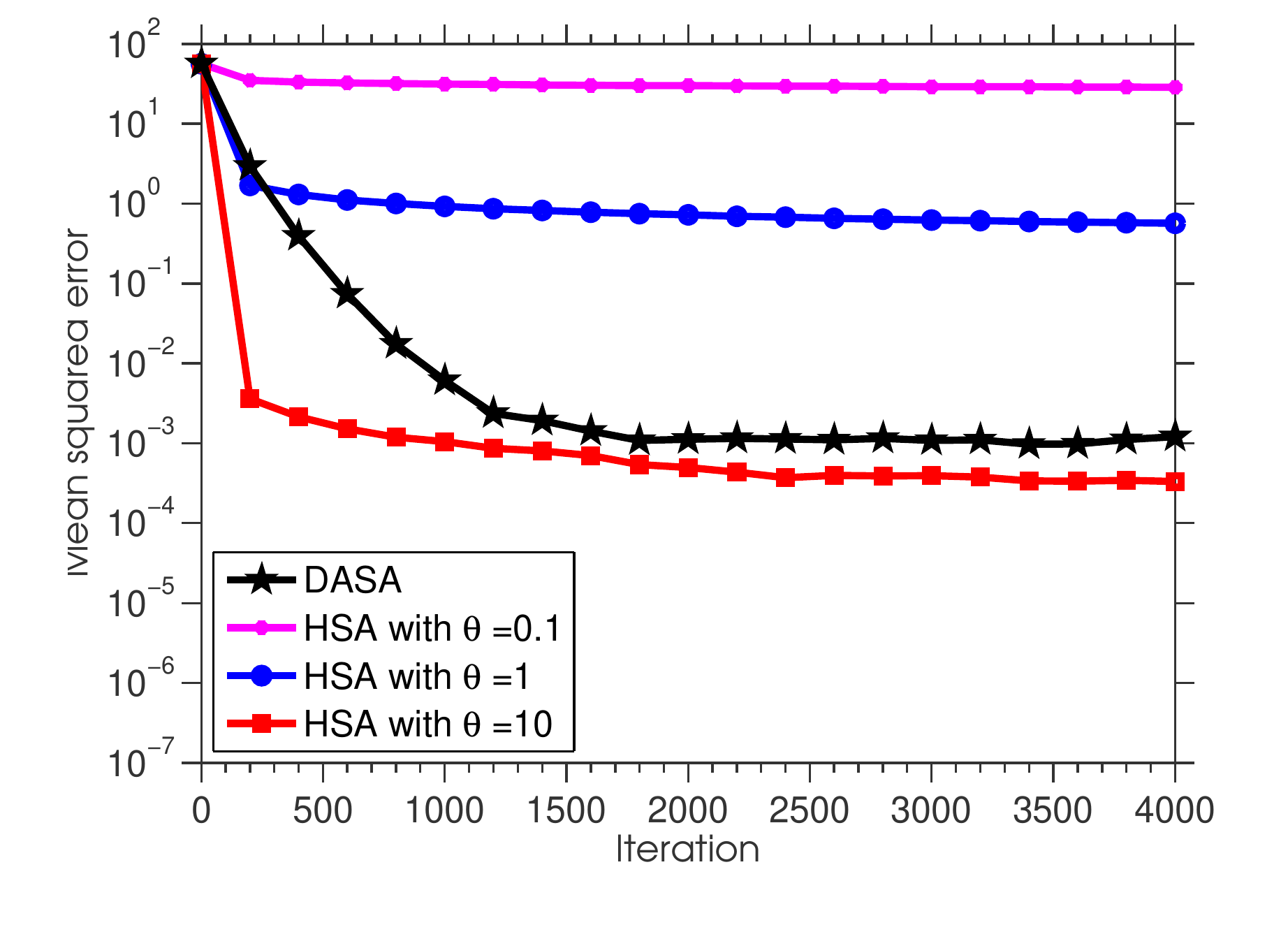}}
\caption{The bandwidth-sharing problem -- MSE -- DASA vs. HSA schemes}
\label{fig:traffic_all}
\end{figure} For example, HSA scheme with $\theta=10$ performs poorly in settings S$(1)$ and S$(4)$, while it outperforms other schemes in setting S$(11)$. We also observe that changing \sfy{$m_b$} from $1$ to $0.1$ does not affect the error. This is because the optimal solution $x^*$ remains feasible for a smaller vector $B$. On the other hand, the error decreases when we use \sfy{$m_b=0.01$}. Figure \ref{fig:traffic_flow} presents the flow rates of the users in different routes for the setting $S(4)$. \begin{figure}[htb]
 \centering
 \subfloat[DASA]
{\label{fig:traffic_DASA}\includegraphics[scale=.25]{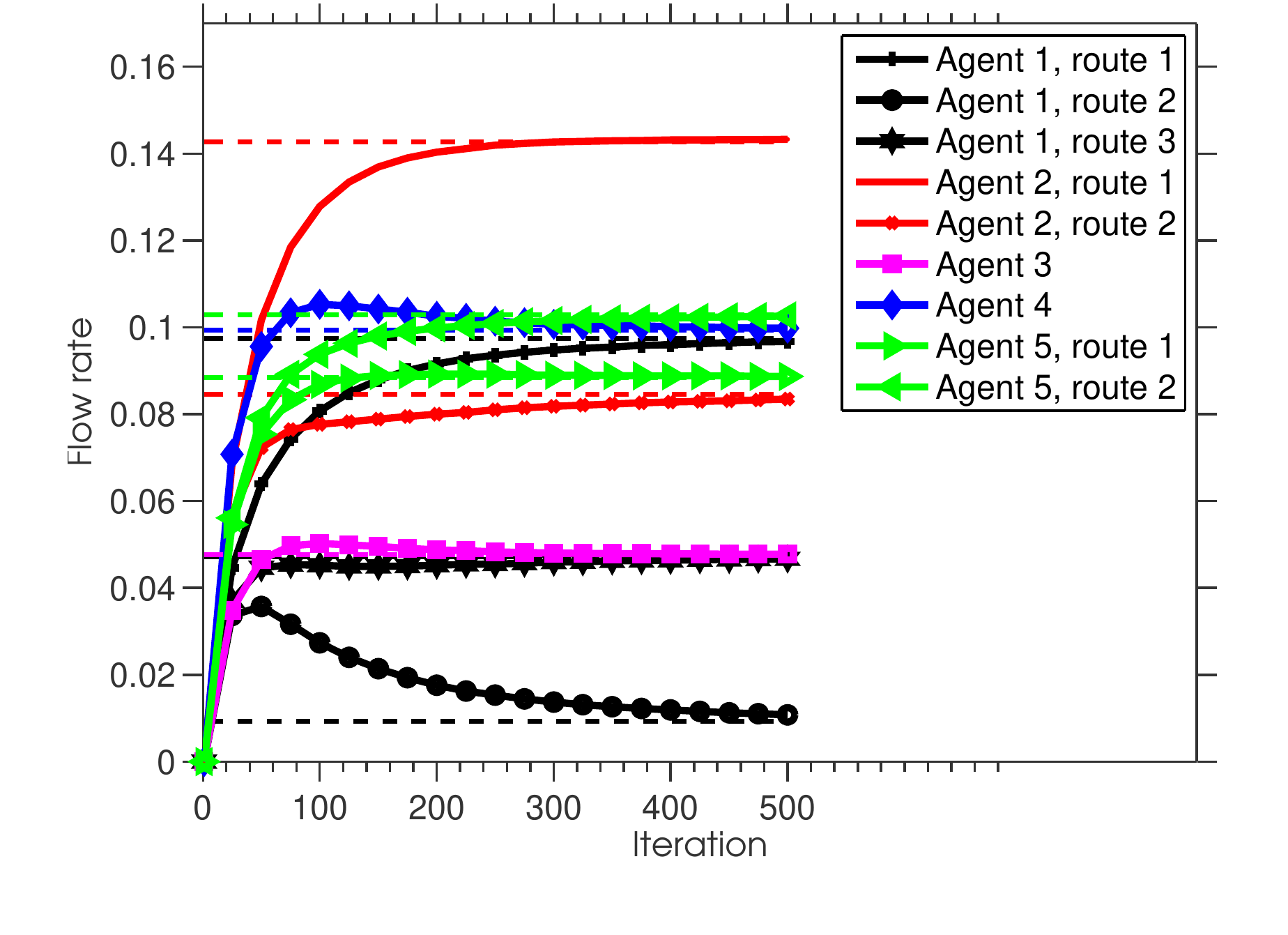}}
 \subfloat[HSA with $\theta=1$]
{\label{fig:traffic_hsa1}\includegraphics[scale=.25]{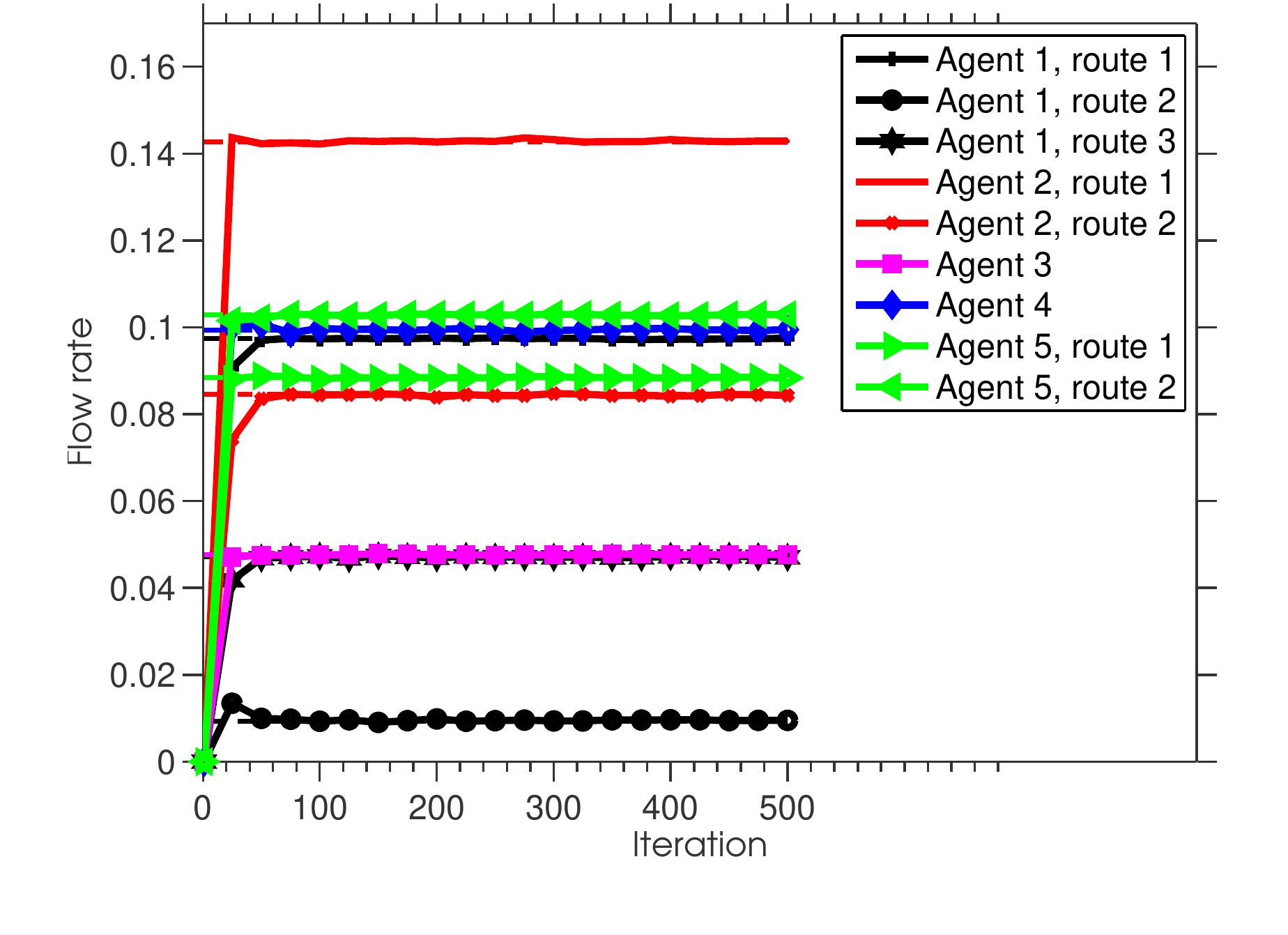}}
 \subfloat[HSA with $\theta=10$]{\label{fig:traffic_hsa10}\includegraphics[scale=.25]{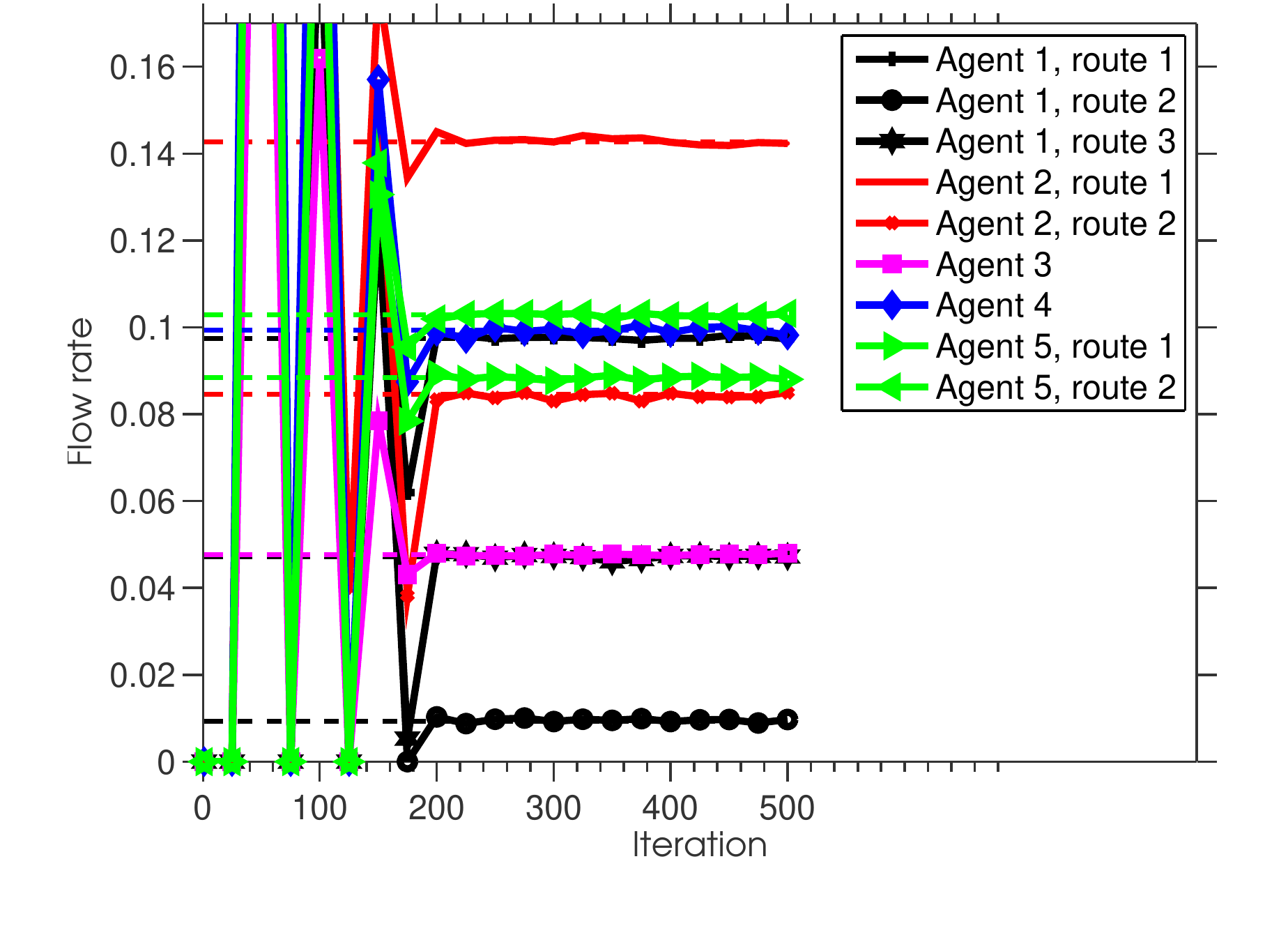}}
\caption{The bandwidth-sharing problem -- flow rates for the setting S$(4)$}
\label{fig:traffic_flow}
\end{figure}One immediate observation is that the flow rates of HSA scheme with $\theta=10$ fluctuates noticeably in the beginning due to a very large stepsize. Figure \ref{fig:traffic_all_conf} provides an image of the $90\%$ CIs for the setting $S(4)$.   	\begin{figure}[htb]
 \centering
 \subfloat[DASA]
{\label{fig:traffic_DASA_conf}\includegraphics[scale=.20]{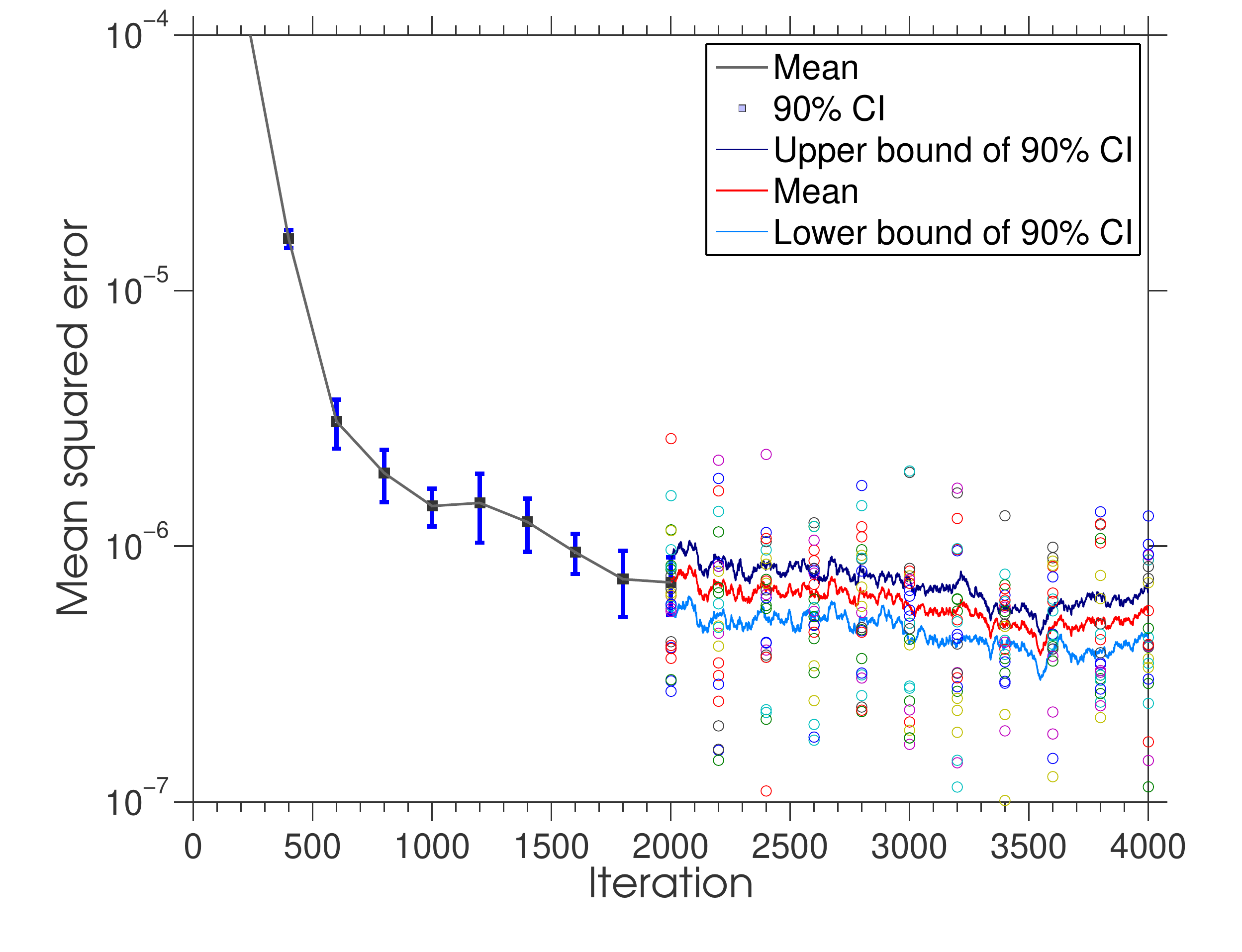}}
 \subfloat[HSA with $\theta=1$]
{\label{fig:traffic_hsa1_conf}\includegraphics[scale=.20]{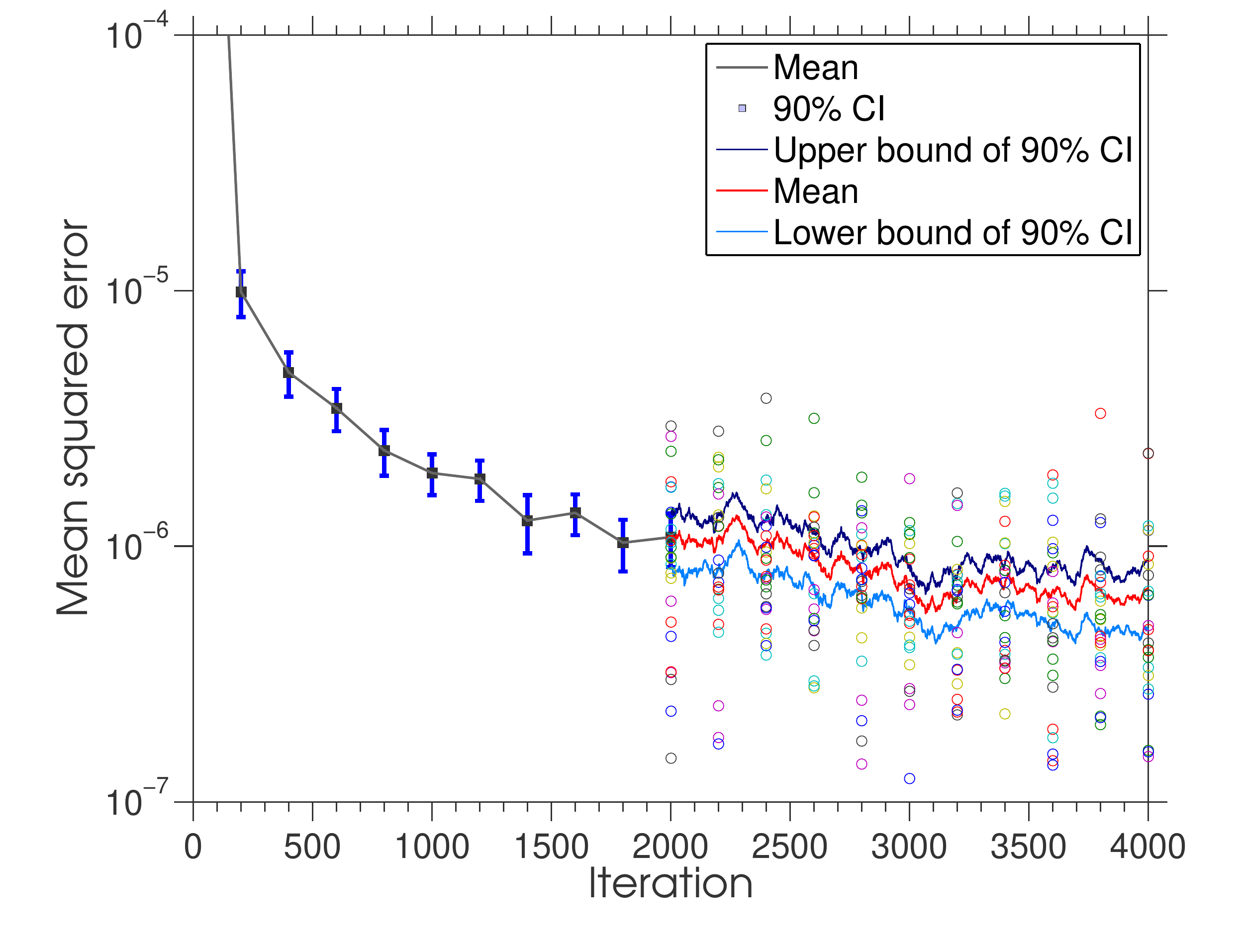}}
 \subfloat[HSA with $\theta=10$]{\label{fig:traffic_hsa10_conf}\includegraphics[scale=.20]{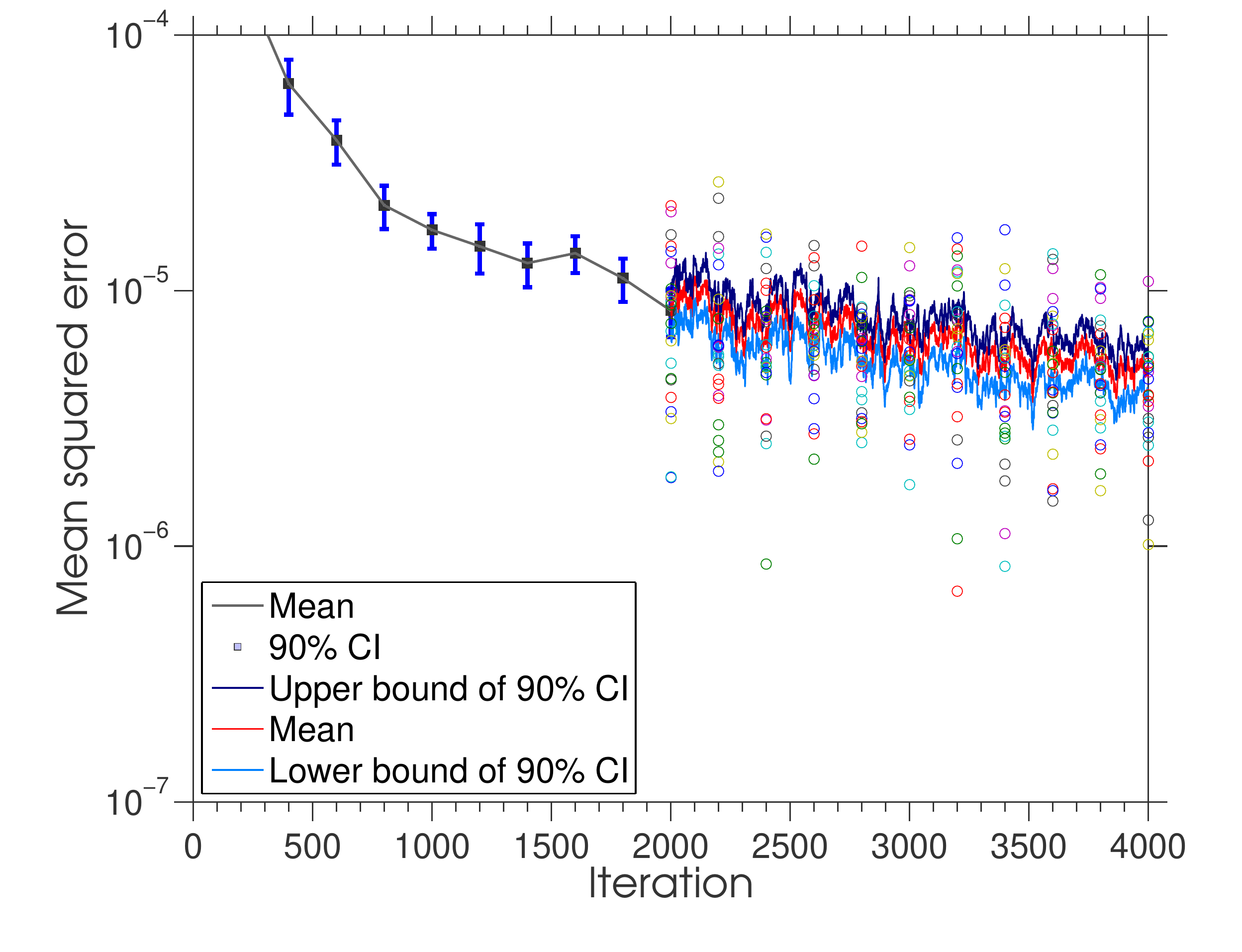}}
\caption{The bandwidth-sharing problem -- $90\%$ CIs for the setting S$(4)$}
\label{fig:traffic_all_conf}
\end{figure}We used two formats to present the intervals. The left-hand side half of each plot shows the intervals  with line segments, while the other half shows the lower and upper bound of the intervals continuously. The colorful points represent the $25$ sample errors at corresponding iterations. We see that the DASA scheme and HSA scheme with $\theta=1$ have CIs with similar size and a smooth mean while the mean in HSA scheme with $\theta=10$ is nonsmooth and oscillates more as the algorithm proceeds.

	   \subsection{A networked stochastic Nash-Cournot
		   game}\label{sec:5.1}
Consider a networked Nash-Cournot game akin to that described in Example
\ref{ex:st-nc}. Specifically, let firm $i$'s generation and
sales decisions at node $j$  be given by $g_{ij}$ and $s_{ij}$,
	  respectively. Suppose the price function $p_{j}$ is given by $p_{j}(\bar s_j,a_j,b_j)=a_j-b_j\bar
	  s_{j}^\sigma$, where  $\bar s_{j}=\sum_{i}s_{ij}$, $\sigma \geq 1$
	  and $a_j$ and $b_j$ are uniformly distributed random variables defined over the
	  intervals $[lb_j^a,ub_j^a]$ and $[lb_j^b,ub_j^b]$, respectively.
	  For purposes of simplicity, we assume that the generation cost is
	  linear and is given by $c_{ij} g_{ij}$. We also impose a bound
	  on sales decisions, as specified $s_{ij}\leq cap_{ij}^\prime$ for
	  all $i$ and $j$. Note that sales decisions are always bounded by
	  aggregate generation capacity. The optimization model for the $i$-th firm is given by:
\begin{align}\label{eqn:ncproblem2}
\displaystyle \mbox{minimize} & \qquad  \EXP{ \sum_{j =1}^M \left( c_{ij}
		g_{ij} - s_{ij}(a_j-b_j\bar
	  s_{j}^\sigma)
		\right)} \\
\mbox{subject to} & \qquad x_i=(s_{i\cdot};g_{i\cdot}) \in X_i \triangleq \left\{\begin{aligned}  & \sum_{j=1}^M
g_{ij}   	= \sum_{j=1}^M s_{ij}, \notag \\ 
& g_{ij}  \leq \mathrm{cap}_{ij},\quad s_{ij}  \leq \mathrm{cap}_{ij}^\prime, \qquad j = 1,
	\hdots, M, 		 \notag	\\
&	g_{ij}, s_{ij}   \geq 0, \qquad j = 1,
	\hdots, M. 		 \notag	
		\end{aligned}\right\}.
\end{align}
As discussed in \cite{kannan10online}, when $1<\sigma \leq 3$ and $M
\leq \frac{3\sigma-1}{\sigma-1}$, the mapping $F$ is strictly monotone and strong
monotonicity can be induced using a regularized mapping, given that our
interest lies in strongly monotone problems. On the other hand, when $\sigma >1$, it is difficult to check that mapping $F$ has Lipschitzian property. This motivates us to employ the distributed locally randomized SA schemes introduced in Sec.
\ref{sec:conv_smoothing}. Now, using regularization and randomized schemes, we would like to solve the
VI$(X,F^\e+\eta\bf{I})$, where $\eta>0$ is the regularization parameter
and $F^\e$ is defined by (\ref{eqn:DLRT-def}). As a consequence, this
problem admits a unique solution denoted by $x^*_{\eta,\e}$.

\subsubsection{SA algorithms}\label{Sec:5.2.1}
In this experiment, we use four different SA schemes for solving
VI$(X,F^\e+\eta\bf{I})$ described in Sec. \ref{sec:dasa} and Sec.
\ref{sec:convergence-no-Lip}: 
\paragraph{MSR-DASA scheme.} In this scheme, we employ the algorithm
(\ref{eqn:DLRSA}) and assume that the random vector $z$ is generated via
the MSR scheme, i.e., $z_i$ is uniformly distributed on the set
$B_{n_i}(0,\e_i)$ while the mapping $F_\e$ is defined by (\ref{eqn:DLRT-def}). One
	immediate benefit of applying this scheme is that the Lipschitzian
	parameter can be estimated from Prop.
	\ref{lemma:DRSS-Lipschitz}b. Moreover, we assume that the stepsizes $\g_{k,i}$ are given by (\ref{eqn:gi0})-(\ref{eqn:gik}). The multiplier $r_i$ is randomly
	chosen for each firm within the prescribed range. The constant $c$
	is maintained at $\frac{\eta}{4}$. Parameters $D$ and $\nu$ need to
	be estimated, while the Lipschitzian parameter $L$
	is obtained by Prop. \ref{lemma:DRSS-Lipschitz}b, i.e., \[L=\sqrt{N}\|C\|
\max_{j=1,\ldots,N}\left\{\kappa_j\frac{n_j!!}{(n_j-1)!!}\,\frac{1}{\e_j}\right\}.\]

	\paragraph{MSR-HSA schemes.} Analogous to the MSR-DASA scheme, this
	scheme uses the distributed locally randomized SA algorithm
	(\ref{eqn:DLRSA}) where for any $i=1,\ldots, N$, the random vector
	$z_i$ is uniformly drawn from the ball $B_{n_i}(0,\e_i)$ and mapping
	$F_\e$ is defined by (\ref{eqn:DLRT-def}). The difference is that
	here we use the harmonic stepsize of the form $\frac{\theta}{k}$ at
	$k$-th iteration for any firm, where $\theta>0$. 
	  
	\paragraph{MCR-DASA scheme.} This scheme is similar to the MSR-DASA
	scheme with one key difference. We assume that random vector $z$ is
	generated by the MCR scheme, i.e., for any $i=1,\ldots, N$, random
	vector $z_i$ is uniformly drawn from the cube $C_{n_i}(0,\e_i)$
	independent from any $z_j$ with $j\neq i$. The Lipschitz constant
	$L$ required for calculating the stepsizes is given by Prop.
	\ref{lemma:DRCS-Lipschitz}b:
	\[L=\frac{\sqrt{n}\|C^\prime\|}{\min_{j=1,\ldots,N} \{\e_j\}}.\]
	
	\paragraph{MCR-HSA schemes.} This scheme uses the algorithm
	(\ref{eqn:DLRSA}) with multi-cubic uniform random variable $z$. The
	stepsizes in this scheme are harmonic of the form
	$\frac{\theta}{k}$. 

To obtain the solution $x^*_{\eta,\e}$, we use the HSA scheme with the
stepsizes $\frac{1}{k}$ using $20000$ iterations. Note that in this
experiment, when we use the DASA scheme, we allow that the condition
$\nu \geq \frac{DL}{\sqrt{2}}$ is violated and we replace it with $\nu
\geq D$. The condition $\nu \geq D$ keeps the adaptive stepsizes
positive for any $k$. As a consequence of ignoring $\nu \geq
\frac{DL}{\sqrt{2}}$, the adaptive stepsizes become larger and in the
order of the harmonic stepsizes in our analysis. Note that by this
change, the convergence of the DASA algorithm is still guaranteed, while
the result of Theorem \ref{prop:DASA}d does not hold necessarily.

\subsubsection{Sensitivity analysis}
We consider a Nash-Cournot game with $5$ firms over a network with $3$
nodes. We set $\sigma=1.1$, $lb_j^b=0.04$, $ub_j^b=0.05$, and $lb_j^a=1$ for any $j$ and $ub^a=(1.5;2;2.5)$. Having these parameters fixed, our test problems
are generated by changing other model's parameters. These parameters are
as follows: the parameter of locally randomized schemes $\e$, the regularization
parameter $\eta$, the starting point of the SA algorithm $x_0$, and the multiplier $M_a$ for the random variable $a_j$ for any $j$. We also consider two different settings for $\mathrm{cap}_{ij}$ and $\mathrm{cap}^\prime_{ij}$. Note that when $\mathrm{cap}_{ij}=1$, the constraints $s_{ij} \leq 3$ are redundant and can be removed. In our analysis we assume that $\e_i \triangleq\e$ is identical for all firms. 
\begin{table}[htb] 
\vspace{-0.05in} 
\tiny 
\centering 
\begin{tabular}{|c|c|c|c|c|c|c|c|} 
\hline 
-&S$(i)$ &  $\epsilon$ & $\eta$ & $x_0$ & $M_a$  & $\mathrm{cap}_{ij}$ & $\mathrm{cap}_{ij}^\prime$ 
\\ 
\hline 
\hline
$\e$ &1&  0.1 & 0.1 & $P_1$ & $1$  & $1$ & $3$
  \\

\hbox{ }& 2 & 0.001 & 0.1 & $P_1$ & $1$  & $1$ & $3$
  \\

\hbox{ }& 3 & 0.0001& 0.1 & $P_1$ & $1$  & $1$ & $3$ \\ 
\hline 

$\eta$ &4&  0.1& 0.1 & $P_2$ & $1$  & $10$ & $1$
  \\

\hbox{ }& 5 & 0.1& 0.05 & $P_2$ & $1$  & $10$ & $1$
  \\

\hbox{ }& 6 & 0.1& 0.01 & $P_2$ & $1$  & $10$ & $1$\\ 
\hline 
$x_0$ &7&  0.1 & 1 & $P_1$ & $6$  & $10$ & $1$
  \\ 
 
\hbox{ }& 8 & 0.1 & 1 & $P_2$ & $6$  & $10$ & $1$
  \\

\hbox{ }& 9 &0.1 & 1 & $P_3$ & $6$  & $10$ & $1$ \\ 
\hline

$M_a$ &10&  0.01 & 0.5 & $P_2$ & $2$  & $1$ & $3$
  \\ 
 
\hbox{ }& 11& 0.01 & 0.5 & $P_2$ & $4$  & $1$ & $3$
  \\

\hbox{ }& 12 & 0.01 & 0.5 & $P_2$ & $6$  & $1$ & $3$\\ 
\hline
 

\end{tabular} 
\caption{The stochastic Nash-Cournot game -- settings of parameters} 
\label{tab:Nash-Cournot} 
\vspace{-0.1in} 
\end{table}	
Similar to the first experiment in Sec. \ref{sec:5.1.2}, we consider a set of test problems
corresponding to each of these parameters. In each set, one parameter
changes and takes $3$ different values, while other parameters are
fixed. Table \ref{tab:Nash-Cournot} represents $12$ test problems as
described. Note that $P_1$, $P_2$, and $P_3$ are three different
feasible starting points. More precisely, $P_1=0$, $P_2=0.5(\mathrm{cap}^\prime;\mathrm{cap})$, and $P_3=(\mathrm{cap}^\prime;\mathrm{cap})$. 
Similar to the first experiment, the termination criteria is running the SA algorithms for $4000$ iterates. We run each algorithm $25$ times and then we obtain the MSE of the form $\frac{1}{25}\sum_{i=1}^{25}\|x_k^i-x^*_{\eta,\e}\|^2$ for any $k=1,\ldots,4000$.
Table \ref{tab:Nash-Cournot3} and Table \ref{tab:Nash-Cournot2} show the $90\%$ CIs of the error for the described schemes. 

\subsubsection{Results and insights}
Table \ref{tab:Nash-Cournot3} presents the simulation results for the test problems using the MSR-DASA and MSR-HSA schemes. One observation is the effect of changing the parameter $\e$ on the error of the schemes is negligible. We only see a slight change in the error of MSR-HSA scheme with $\theta=10$. Comparing the order of the error, we notice that the MSR-DASA scheme is placed second among all schemes of the first set of the test problems. In the second set, by decreasing $\eta$ the error of all the schemes, except for the MSR-HSA scheme with $\theta=0.1$, first decreases and then increases. This is not an odd observation since we used $x_{\eta,\e}^*$ instead of $x^*$ to measure the errors and $x_{\eta,\e}^*$ changes itself when $\eta$ or $\e$ changes. In this set, the MSR-DASA scheme still has the second best errors among all schemes. The schemes are not much sensitive to the choice of $x_0$ and we observe that the second place is still reserved by the MSR-DASA scheme. Finally, in the last set, we see that increasing the factor $M_a$, as we expect, increases the error in most of the schemes. The reason is that increasing the order of $M_a$ increases both mean and variance of the random variable $a$. Importantly, we observe that our MSR-DASA scheme remains very robust among the MSR-HSA scheme. 
\begin{table}[htb] 
\vspace{-0.05in} 
\tiny 
\centering 
\begin{tabular}{|c|c||c||c||c|c|} 
\hline 
-&S$(i)$ &   DASA - $90\%$ CI &  HSA with $\theta=0.1$- $90\%$ CI & HSA with $\theta=1$ - $90\%$ CI & HSA with $\theta=10$ - $90\%$ CI
\\ 
\hline 
\hline
  $\e$ &1 & [$1.38 $e${-2}$,$2.37 $e${-2}$] &  [$1.83 $e${+1}$,$1.87 $e${+1}$] &  [$1.60 $e${-1}$,$2.15 $e${-1}$]&  [$3.07 $e${-3}$,$5.33 $e${-3}$]
  \\

\hbox{ }&2 & [$1.38 $e${-2}$,$2.37 $e${-2}$] &  [$1.83 $e${+1}$,$1.87 $e${+1}$] &  [$1.60 $e${-1}$,$2.15 $e${-1}$]&  [$3.04 $e${-3}$,$5.30 $e${-3}$]
  \\

\hbox{ }&3& [$1.38 $e${-2}$,$2.37 $e${-2}$] &  [$1.83 $e${+1}$,$1.87 $e${+1}$] &  [$1.60 $e${-1}$,$2.15 $e${-1}$]&  [$3.04 $e${-3}$,$5.30 $e${-3}$]
  \\ 
  
  \hline 
  
  $\eta$ &4 & [$1.92 $e${-3}$,$3.98 $e${-3}$] &  [$1.63$e${-0}$,$1.71 $e${-0}$] &  [$8.43 $e${-3}$,$1.62 $e${-2}$]&  [$5.28 $e${-4}$,$1.08 $e${-3}$]
  \\

\hbox{ }&5 & [$1.42 $e${-3}$,$3.12 $e${-3}$] &  [$1.84$e${-0}$,$1.93 $e${-0}$] &  [$7.43 $e${-3}$,$1.44 $e${-2}$]&  [$2.59 $e${-4}$,$5.76 $e${-4}$]
  \\

\hbox{ }&6 & [$5.61 $e${-3}$,$1.62 $e${-2}$] &  [$2.33$e${-0}$,$2.44 $e${-0}$] &  [$1.61 $e${-2}$,$2.39 $e${-2}$]&  [$5.06 $e${-4}$,$8.65 $e${-4}$]
  \\ 
  \hline

 $x_0$ &7& [$2.68 $e${-6}$,$3.48 $e${-6}$] &  [$4.37 $e${-1}$,$5.13 $e${-1}$] &  [$1.37 $e${-6}$,$1.92 $e${-6}$]&  [$6.71 $e${-6}$,$9.21 $e${-6}$] 
  \\

\hbox{ }&8 &[$2.68 $e${-6}$,$3.48 $e${-6}$] &  [$2.22 $e${-5}$,$2.91 $e${-5}$] &  [$1.37 $e${-6}$,$1.92 $e${-6}$]&  [$6.71 $e${-6}$,$9.21 $e${-6}$] 
  \\

\hbox{ }&9 & [$2.68 $e${-6}$,$3.48 $e${-6}$] &  [$2.22 $e${-5}$,$2.91 $e${-5}$] &  [$1.37 $e${-6}$,$1.92 $e${-6}$]&  [$6.71 $e${-6}$,$9.21 $e${-6}$] 
  \\
  \hline
  
  $M_a$ &10 & [$4.45 $e${-3}$,$9.25 $e${-3}$] &  [$5.79 $e${-1}$,$9.25 $e${-1}$] &  [$1.67 $e${-3}$,$5.72 $e${-3}$]&  [$2.72 $e${-5}$,$2.07 $e${-2}$]
  \\

\hbox{ }&11 & [$8.85 $e${-3}$,$1.73 $e${-2}$] &  [$1.25 $e${-0}$,$2.12 $e${-0}$] &  [$9.38 $e${-4}$,$1.82 $e${-2}$]&  [$4.52 $e${-3}$,$3.22 $e${-2}$]
  \\

\hbox{ }&12 & [$1.92 $e${-2}$,$3.91 $e${-2}$] &  [$8.51 $e${-1}$,$2.31 $e${-0}$] &  [$1.87 $e${-3}$,$4.15$e${-2}$]&  [$1.04 $e${-2}$,$7.23 $e${-2}$]
  \\ 

  \hline

\end{tabular} 
\caption{The stochastic Nash-Cournot game -- $90\%$ CIs for MSR-DASA and MSR-HSA schemes} 
\label{tab:Nash-Cournot3} 
\vspace{-0.1in} 
\end{table}	
Table \ref{tab:Nash-Cournot2} shows the error estimations using the MCR-DASA and MCR-HSA schemes. Comparing these results with the MSR schemes in Table \ref{tab:Nash-Cournot3}, we see that the sensitivity of the MCR schemes to the parameters is very similar to that of MSR schemes and the MCR-DASA scheme performs as the second best among all MCR schemes. We also see that in most of the settings, the error of the MSR-DASA scheme is slightly smaller than the error of the MCR-DASA scheme. One reason can be that the MSR scheme has a smaller Lipschitz constant than the MCR scheme for our problem settings.
\begin{table}[htb] 
\vspace{-0.05in} 
\tiny 
\centering 
\begin{tabular}{|c|c||c||c||c|c|} 
\hline 
-&S$(i)$ &   DASA - $90\%$ CI &  HSA with $\theta=0.1$- $90\%$ CI & HSA with $\theta=1$ - $90\%$ CI & HSA with $\theta=10$ - $90\%$ CI
\\ 
\hline 
\hline
$\e$ &1 & [$1.22 $e${-2}$,$2.55 $e${-2}$] &  [$1.84 $e${+1}$,$1.88 $e${+1}$] &  [$1.78 $e${-1}$,$2.29 $e${-1}$]&  [$2.42 $e${-3}$,$4.21 $e${-3}$]
  \\

\hbox{ }&2 &[$1.21 $e${-2}$,$2.53 $e${-2}$] &  [$1.84 $e${+1}$,$1.88 $e${+1}$] &  [$1.78 $e${-1}$,$2.28 $e${-1}$]&  [$2.37 $e${-3}$,$4.13 $e${-3}$]
  \\

\hbox{ }&3 & [$1.21 $e${-2}$,$2.53 $e${-2}$] &  [$1.84 $e${+1}$,$1.88 $e${+1}$] &  [$1.78 $e${-1}$,$2.28 $e${-1}$]&  [$2.37 $e${-3}$,$4.13 $e${-3}$]
  \\ 
  
  \hline 
$\eta$ &4 & [$4.17 $e${-3}$,$9.50 $e${-3}$] &  [$1.65$e${-0}$,$1.74 $e${-0}$] &  [$9.37 $e${-3}$,$1.84 $e${-2}$]&  [$7.38 $e${-4}$,$1.73 $e${-3}$]
  \\

\hbox{ }&5 & [$1.41 $e${-3}$,$4.06 $e${-3}$] &  [$1.85$e${-0}$,$1.93 $e${-0}$] &  [$6.88 $e${-3}$,$1.32 $e${-2}$]&  [$2.85 $e${-4}$,$5.06 $e${-4}$]
  \\

\hbox{ }&6 & [$8.19 $e${-3}$,$1.88 $e${-2}$] &  [$2.37$e${-0}$,$2.46 $e${-0}$] &  [$1.85 $e${-2}$,$3.10 $e${-2}$]&  [$4.18 $e${-4}$,$7.05 $e${-4}$]
  \\ 
  \hline

 $x_0$ &7 &[$2.25 $e${-5}$,$2.88 $e${-5}$] &  [$4.31 $e${-1}$,$5.12 $e${-1}$] &  [$9.41 $e${-6}$,$1.18 $e${-5}$]&  [$3.99 $e${-5}$,$5.27 $e${-5}$] 
  \\

\hbox{ }&8 & [$2.25 $e${-5}$,$2.88 $e${-5}$] &  [$1.13 $e${-4}$,$1.58 $e${-4}$] &  [$9.40 $e${-6}$,$1.18 $e${-6}$]&  [$3.99 $e${-5}$,$5.27 $e${-5}$] 
  \\

\hbox{ }&9 & [$2.25 $e${-5}$,$2.88 $e${-5}$] &  [$1.13 $e${-4}$,$1.58 $e${-4}$] &  [$9.40 $e${-6}$,$1.18 $e${-5}$]&  [$3.99 $e${-5}$,$5.27 $e${-5}$] 
  \\ 
  
  \hline
  
  $M_a$ &10 & [$1.66 $e${-3}$,$4.29 $e${-3}$] &  [$6.17 $e${-1}$,$8.88 $e${-1}$] &  [$4.21 $e${-4}$,$1.79 $e${-3}$]&  [$3.82 $e${-4}$,$8.30 $e${-3}$]
  \\

\hbox{ }&11 & [$3.03 $e${-3}$,$1.22 $e${-2}$] &  [$1.29 $e${-0}$,$2.23 $e${-0}$] &  [$9.63 $e${-4}$,$5.77 $e${-3}$]&  [$2.48 $e${-3}$,$2.52 $e${-2}$]
  \\

\hbox{ }&12 & [$6.05 $e${-3}$,$2.60 $e${-2}$] &  [$8.50 $e${-1}$,$2.49 $e${-0}$] &  [$2.27 $e${-3}$,$1.29$e${-2}$]&  [$5.54 $e${-3}$,$5.67 $e${-2}$]
  \\ 
  \hline
\end{tabular} 
\caption{The stochastic Nash-Cournot game -- $90\%$ CIs for MCR-DASA and MCR-HSA schemes} 
\label{tab:Nash-Cournot2} 
\vspace{-0.1in} 
\end{table}	

	\begin{figure}[htb]
 \centering
 \subfloat[Setting S$(5)$]
{\label{fig:prob5}\includegraphics[scale=.45]{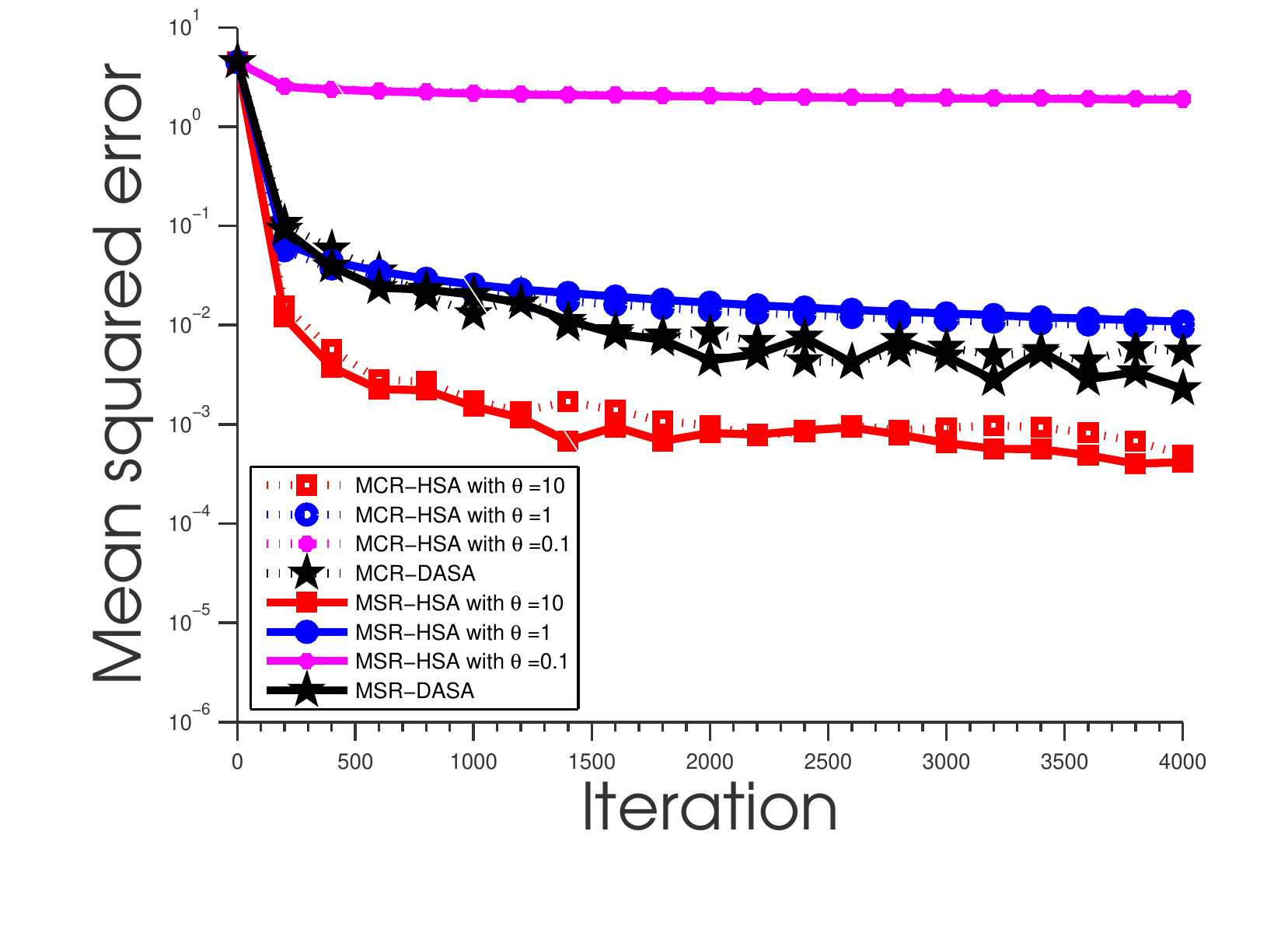}}
 \subfloat[Setting S$(8)$]{\label{fig:prob8}\includegraphics[scale=.45]{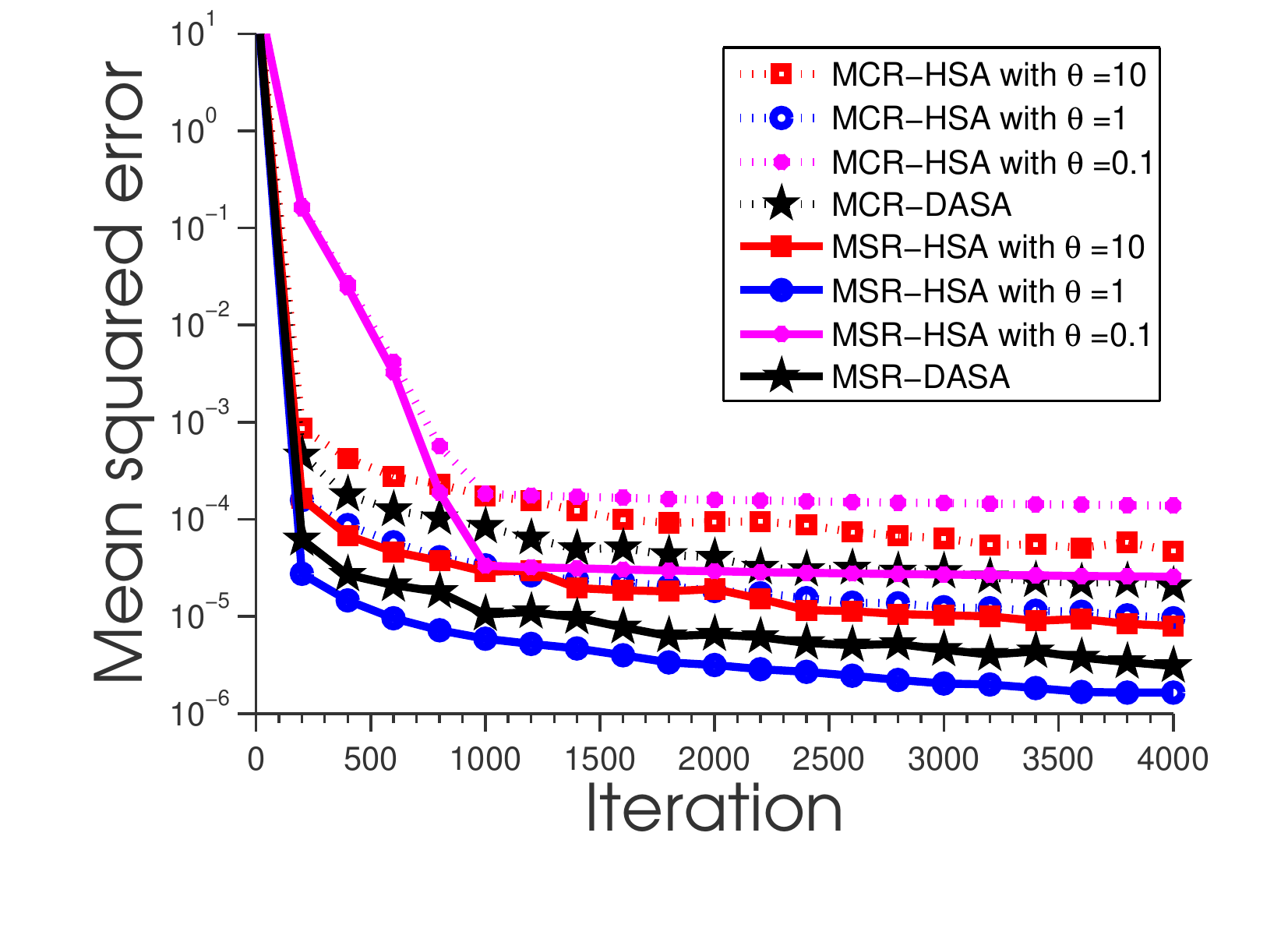}}
\caption{The stochastic Nash-Cournot game -- comparison among all the schemes}
\label{fig:cornout_all}
\end{figure}
Figure \ref{fig:cornout_all} illustrates a comparison among the
different schemes described in Sec. \ref{Sec:5.2.1} for the case of
setting S$(5)$ and S$(8)$. All the MSR schemes are shown with solid
lines, while the MCR schemes are presented with dashed lines. There are
some immediate observations here. Regarding the order of the error, in
both if the settings S$(5)$ and S$(8)$, the schemes with the distributed
adaptive stepsizes given by (\ref{eqn:gi0})-(\ref{eqn:gik}) are the
second best scheme among each of MSR and MCR schemes. This indicates the
robustness of the DASA scheme compared with the HSA schemes. We also
observe that in the setting S$(5)$, the HSA schemes with $\theta=10$
(both MSR and MCR) have the minimum error, while in setting S$(8)$, the
HSA schemes with $\theta=1$ has the minimum error. This is an
illustration of sensitivity of HSA schemes to the setting of problem
parameters. Let us now compare the MSR schemes with the MCR schemes. In
the setting S$(5)$, the MSR and MCR schemes perform very closely and in
fact, it is hard to distinguish the difference between  their errors. On
the other hand, in the setting S$(8)$, we see that the MSR schemes have
a better performance than their MCR counterparts.	
\begin{figure}[htb]
 \centering
 \subfloat[MSR-DASA scheme]
{\label{fig:cournot_MSR-DASA}\includegraphics[scale=.3]{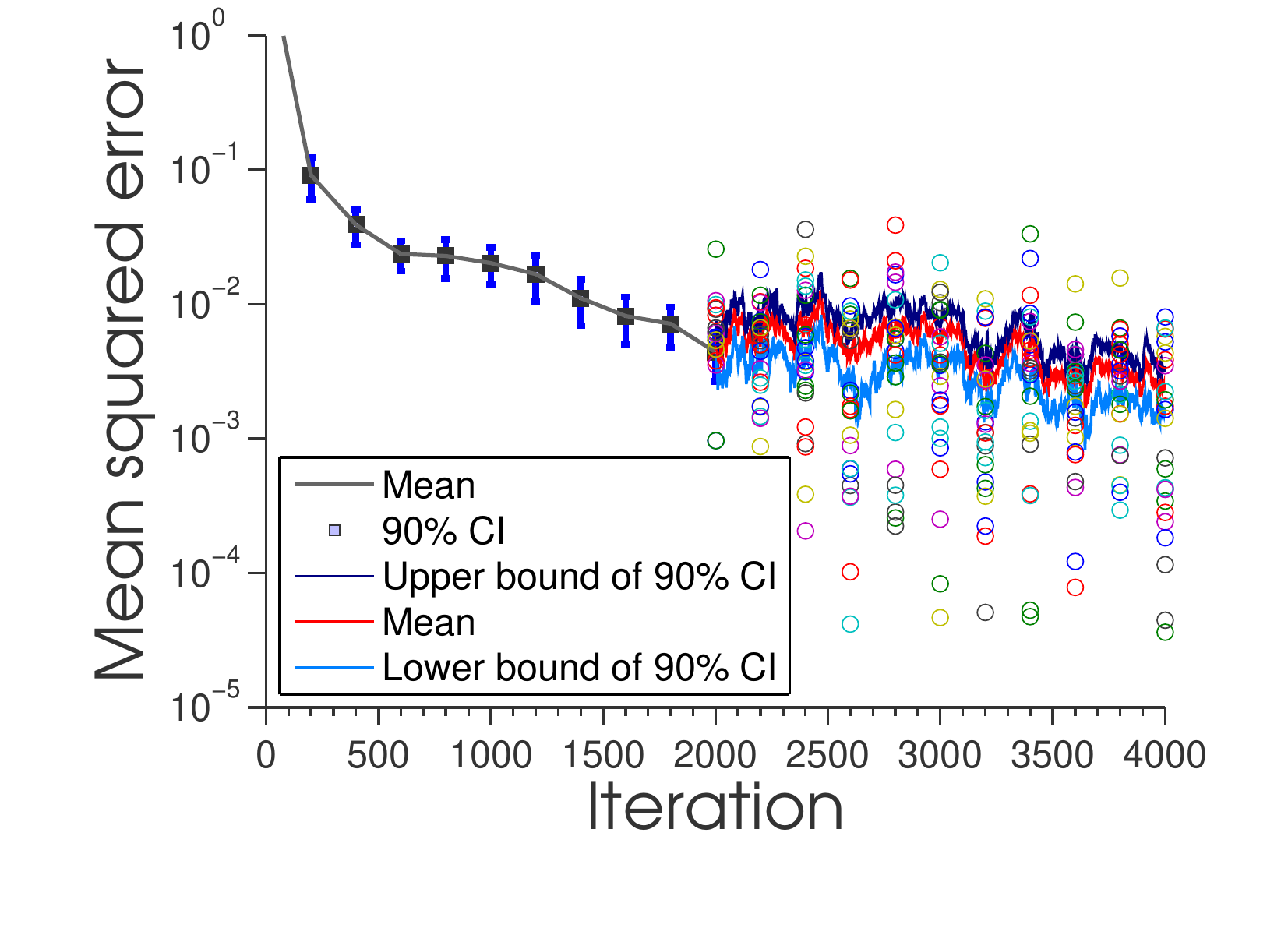}}
 \subfloat[MSR-HSA scheme with $\theta=1$]
{\label{fig:cournot_MSR-HSA1}\includegraphics[scale=.3]{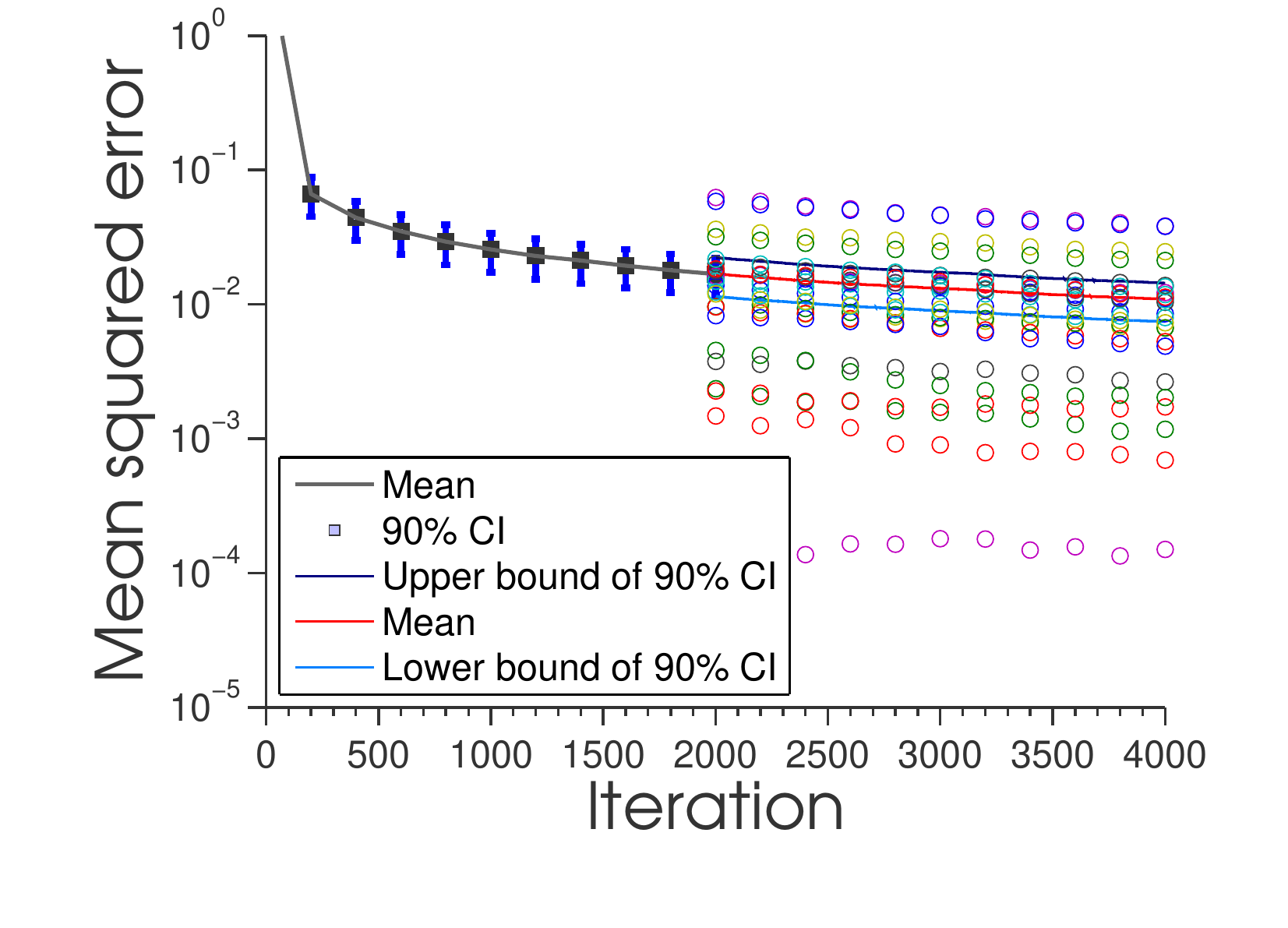}}
 \subfloat[MSR-HSA scheme with $\theta=10$]{\label{fig:cournot_MSR-HSA10}\includegraphics[scale=.3]{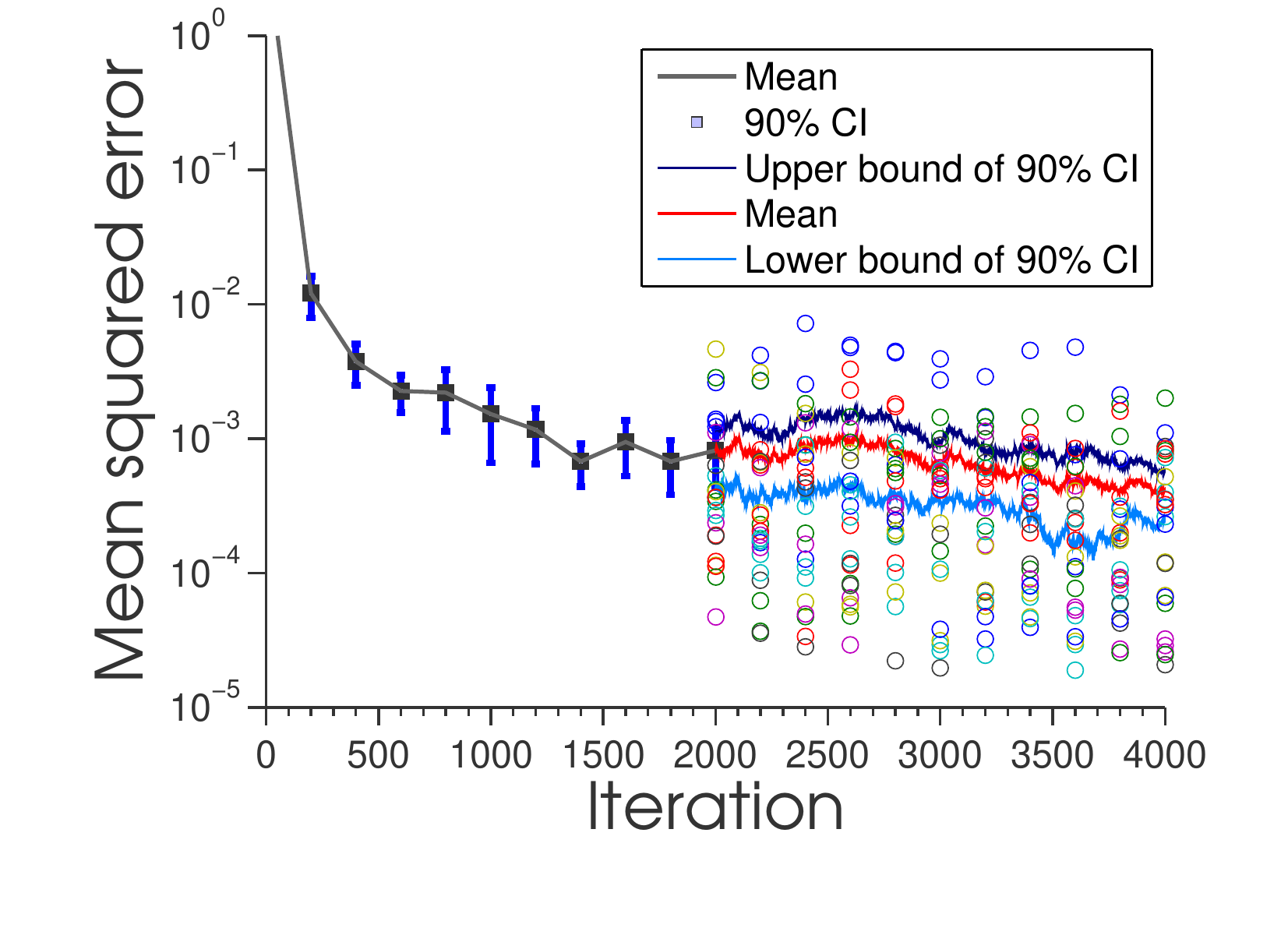}}
\caption{The stochastic Nash-Cournot game -- setting S$(5)$ -- MSR-DASA vs. MSR-HSA schemes}
\label{fig:cournot-conf-MSR}
\end{figure}
Figure \ref{fig:cournot-conf-MSR} illustrates the $90\%$ confindence intervals for the MSR schemes with the setting S$(5)$. Teese intervals are shown with line segments in the left-hand side half of each plot and shown with continious bouns in the right-hand side half. The colourful points present the samples at each level of iterates. Impostantly, we observe that the CIs of MSR-DASA scheme are as tight as the MSR-HSA scheme with $\theta=1$ and they are tighter than the ones in the MSR-HSA scheme with $\theta=10$. Figure \ref{fig:cournot-conf-MCR} shows the similar comparison for the MCR schemes.
	\begin{figure}[htb]
 \centering
 \subfloat[MCR-DASA scheme]
{\label{fig:cournot_DASAMSR}\includegraphics[scale=.3]{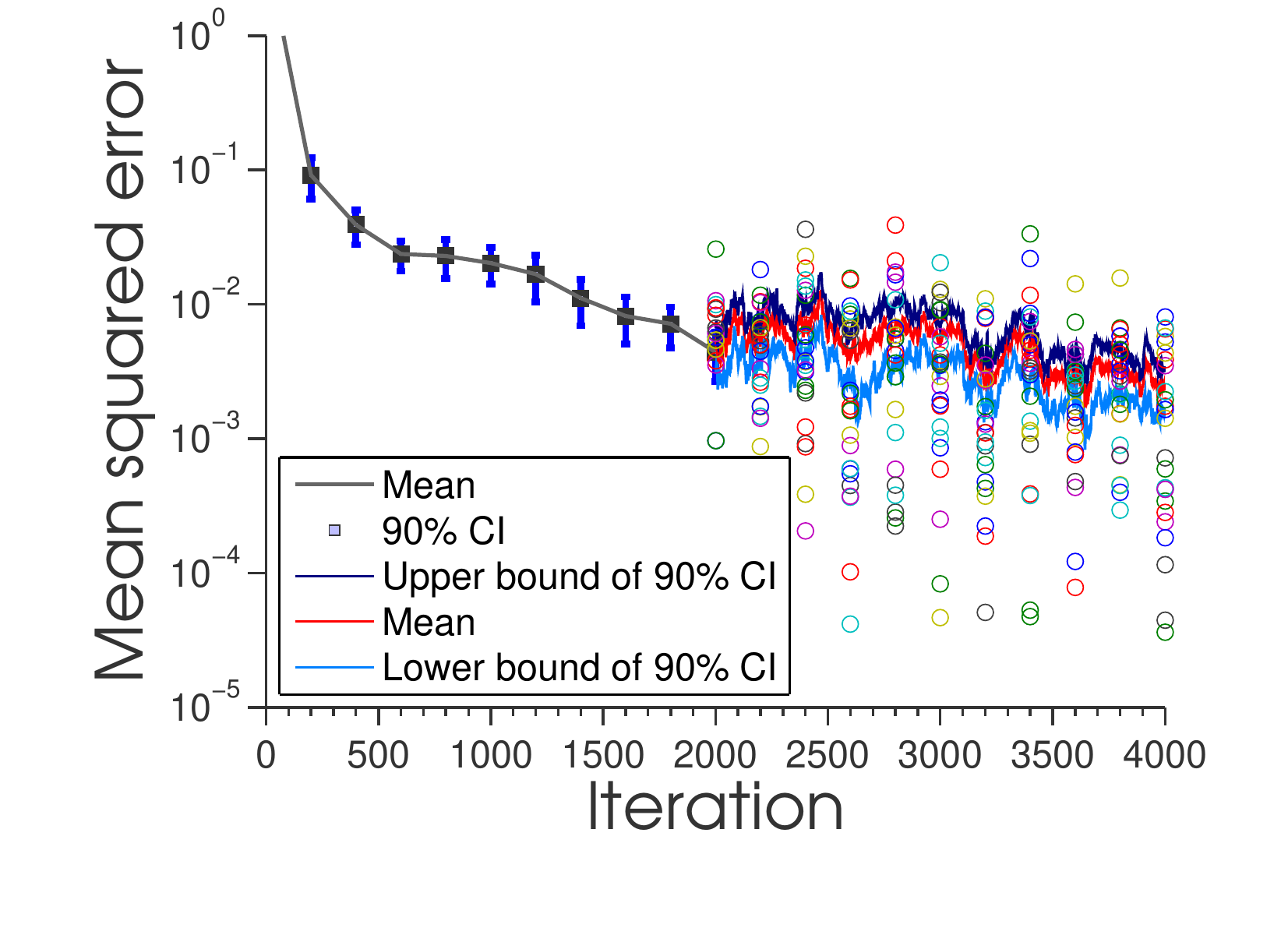}}
 \subfloat[MCR-HSA scheme with $\theta=1$]
{\label{fig:cournot_HSA1MSR}\includegraphics[scale=.3]{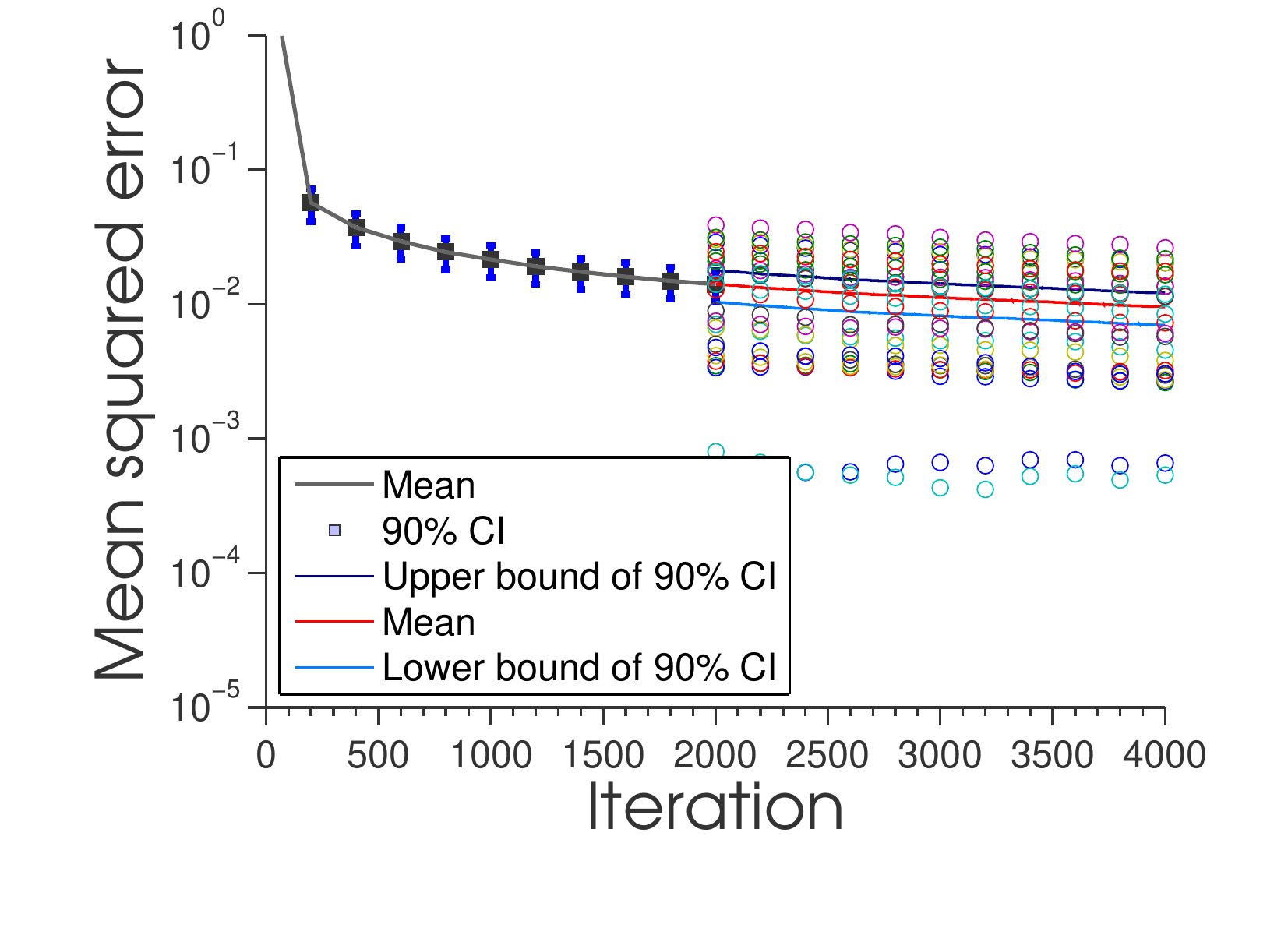}}
 \subfloat[MCR-HSA scheme with $\theta=10$]{\label{fig:cournot_HSA10MSR}\includegraphics[scale=.3]{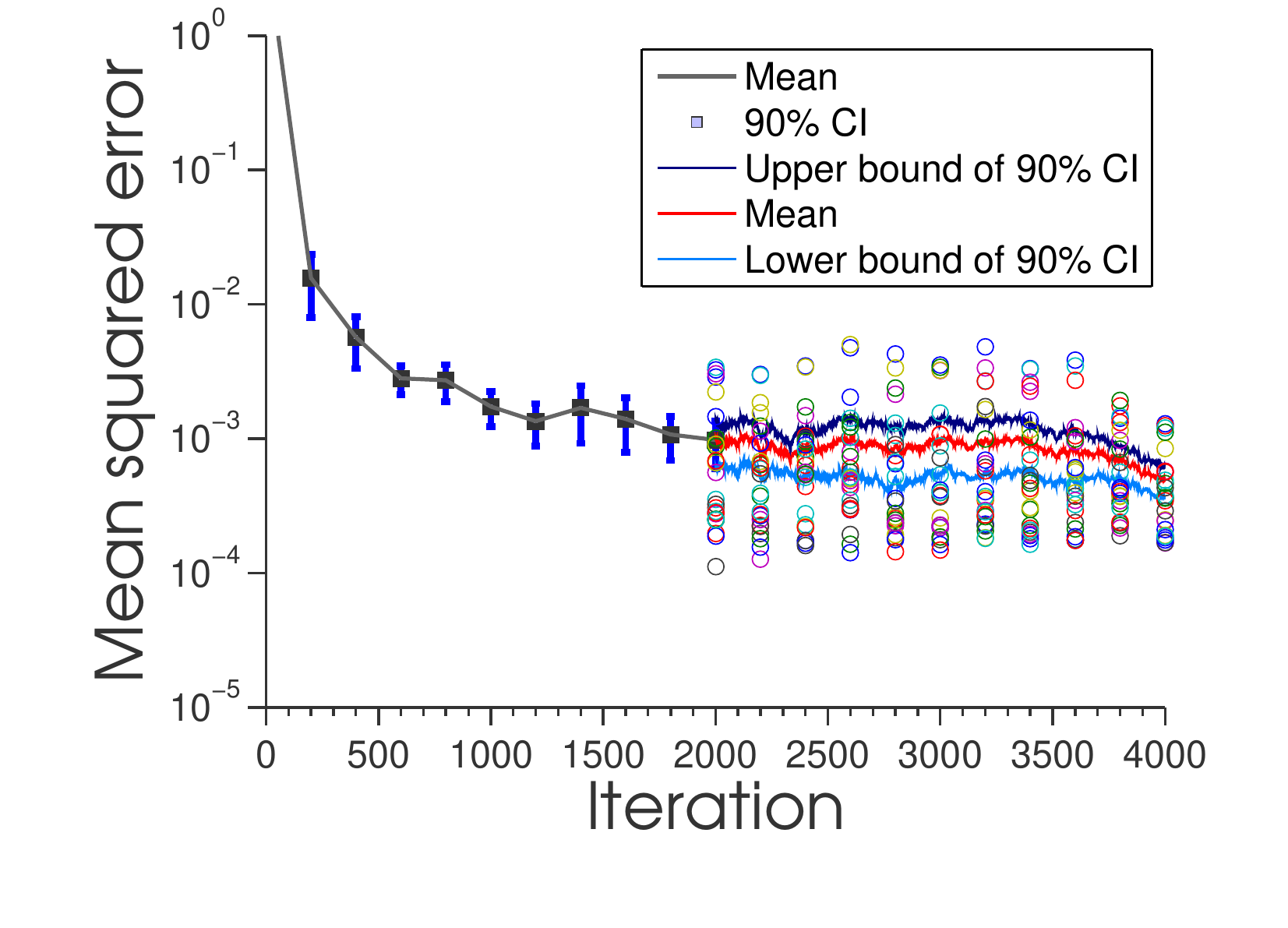}}
\caption{The stochastic Nash-Cournot game -- setting S$(5)$ -- MCR-DASA vs. MCR-HSA schemes}
\label{fig:cournot-conf-MCR}
\end{figure}

\section{Concluding remarks} 
We consider the solution of strongly monotone Cartesian stochastic
variational inequality problems through stochastic approximation (SA)
	schemes.  Motivated by the naive stepsize rules employed in most SA
	implementations, we develop a recursive rule that adapts to problem
	parameters such as the Lipschitz and monotonicity constants of the
	map and ensures almost-sure convergence of the iterates to the
	unique solution. An extension to the distributed multi-agent regime
	is provided. A shortcoming of this approach is the reliance on the
	availability of a Lipschitz constant. This motivates the
	construction of two locally randomized techniques to cope with
	instances where the mapping is either not Lipschitz or estimating
	the parameter is challenging. In each of these techniques, we show
	that an approximation of the original mapping is Lipschitz
	continuous with a prescribed constant. We utilize these techniques
	in developing a distributed locally randomized adaptive steplength
	SA scheme where we perturb the mapping at each iteration by a
	uniform random variable over a prescribed distribution. It is shown
	that this scheme produces iterates that converge to a solution of an
	approximate problem, and the sequence of approximate solutions
	converge to the unique solution of the original stochastic
	variational problem. In Sec.  \ref{sec:numerics}, we apply our
	schemes on two sets of problems, a bandwidth-sharing problem in
	communication networks and a networked stochastic Nash-Cournot game.
	Through these examples, we observed that the adaptive distributed
	stepsize scheme displays far more robustness than the standard
	implementations that leverage harmonic stepsizes of the form
	$\frac{\theta}{k}$ in both problems.  Furthermore, the randomized
	smoothing techniques assume utility in the Cournot regime where
	Lipschitz constants cannot be easily
	derived.  \bibliographystyle{siam}
	\bibliography{ref3,Nash_adapt,ref2} \end{document}